\documentclass[reqno,11pt]{amsart}
\usepackage[paper=letterpaper,margin=1in]{geometry}

\newcommand{\e}{\varepsilon}

\newcommand{\h}{\mathfrak{h}}
\newcommand{\aw}{a}

\newcommand{\ind}{\mathbf{1}}

\usepackage{amsmath,amssymb,amsthm,amsfonts}
\usepackage{hyperref}
\usepackage{graphicx,color,colortbl}
\usepackage{upgreek}
\usepackage[mathscr]{euscript}
\usepackage{hhline}
\DeclareMathAlphabet{\mathpzc}{OT1}{pzc}{m}{it}
\usepackage{mathbbol,mathtools,calc}
\usepackage[labelfont={rm}]{subfig}
\usepackage{bm}
\usepackage{overpic}
\usepackage{musicography}
\usepackage{tikz-cd} 
\usepackage{multicol}

\usepackage{blkarray}

\usepackage{soul}

\usepackage{subfig}

\DeclareMathAlphabet{\mathpzc}{OT1}{pzc}{m}{it}
\DeclareMathAlphabet{\mathdutchcal}{U}{dutchcal}{m}{n}
\SetMathAlphabet{\mathdutchcal}{bold}{U}{dutchcal}{b}{n}

\renewcommand{\Pr}{\mathbb{P}}
\newcommand{\Ex}{\mathbb{E}}
\newcommand{\Con}{C}

\renewcommand{\d}{\diff}
\newcommand{\logq}{\eta_q}

\renewcommand{\hat}{\widehat}
\newcommand{\lln}{\gamma_\kappa}

\newcommand{\gint}{\mathfrak{M}(a;N,M-1)}

\allowdisplaybreaks
\numberwithin{equation}{section}
\usepackage{ytableau}
\usepackage{mathdots}
\usepackage{tikz}
\usetikzlibrary{
	shapes,
	arrows,
	positioning,
	decorations.markings,
	decorations.pathmorphing,
	circuits.logic.US,
	circuits.logic.IEC,
	fit,
	calc,
	plotmarks,
	matrix
}

\tikzset{
>=stealth',
help lines/.style={dashed, thick},
axis/.style={<->},
important line/.style={thick},
connection/.style={thick, dotted},
punkt/.style={
rectangle,
rounded corners,
draw=black, thick,
text width=4.5em,
minimum height=2em,
text centered,
},
pil/.style={
->,
thick,
gray,
shorten <=2pt,
shorten >=2pt,}
}

\usepackage{array}
\usepackage{adjustbox}
\usepackage{cleveref}
\usepackage{enumitem}

\usepackage{fancyhdr}

\newtheorem{proposition}{Proposition}[section]
\newtheorem{lemma}[proposition]{Lemma}

\newtheorem{theorem}[proposition]{Theorem}
\newtheorem*{theorem*}{Theorem}

\theoremstyle{definition}
\newtheorem{definition}[proposition]{Definition}

\newtheorem*{remark*}{Remark}

\renewcommand{\i}{\infty}

\newcommand{\R}{\mathbb{R}}

\newcommand{\Z}{\mathbb{Z}}

\newcommand*\diff{\mathop{}\!\mathrm{d}}

\title{Lower tail Large Deviations of the Stochastic Six Vertex Model}

\author[S.\ Das]{Sayan Das}
\address{S.\ Das,
	Department of Mathematics, University of Chicago,
	\newline\hphantom{\quad \ \ S. Das}
	5734 S.~University Avenue, Chicago, IL 60637, USA
}
\email{sayan.das@columbia.edu}

\author[Y. Liao]{Yuchen Liao}\address{Y. Liao, 
School of Mathematical Sciences, University of Science and Technology of China, \newline\hphantom{\quad \ \ Y. Liao}
No.96 Jinzhai Road,
Hefei, Anhui 230026, China}\email{ycliao@ustc.edu.cn}

\author[M. Mucciconi]{Matteo Mucciconi}\address{M. Mucciconi, 
Department of Mathematics,
National University of Singapore, \newline\hphantom{\quad \ \ M. Mucciconi}
S17, 10 Lower Kent Ridge Road, 119076, Singapore}\email{matteomucciconi@gmail.com}

\date{}

\begin{document}

\maketitle

\begin{abstract} 
In this paper, we study lower tail probabilities of the height function $\h(M,N)$ of the stochastic six-vertex model. We introduce a novel combinatorial approach to demonstrate that the tail probabilities $\Pr(\h(M,N) \ge r)$ are log-concave in a certain weak sense. We prove further that for each $\kappa>0$ the lower tail of $-\h(\lfloor \kappa N \rfloor, N)$ satisfies a Large Deviation Principle (LDP) with speed $N^2$ and a rate function $\Phi_\kappa^{(-)}$, which is given by the infimal deconvolution between a certain energy integral and a parabola.
\\
\indent
Our analysis begins with a distributional identity from \cite{BO2016_ASEP}, which relates the lower tail of the height function, after a random shift, with a multiplicative functional of the Schur measure. Tools from potential theory allow us to extract the LDP for the shifted height function.  We then use our weak log-concavity result, along with a deconvolution scheme from our earlier paper \cite{das2023large}, to convert the LDP for the shifted height function to the LDP for the stochastic six-vertex model height function.
\end{abstract}

\tableofcontents

\section{Introduction}

\subsection{The model and main results} \label{subs:model and main result}  The stochastic six-vertex model (S6V) was introduced in the physics literature by Gwa and Spohn in \cite{GwaSpohn1992} as a stochastic variant of the square-ice model \cite{pauling1935structure}. It is defined as a probability measure on configurations of up-right directed paths (viewed as a string of arrows) on the quadrant ${\Lambda}:=\Z_{\ge 0}^2$ that satisfies the following two properties:
\begin{itemize}[leftmargin=20pt]
    \item Every path begins from $x$-axis or $y$-axis and leaves the co-ordinate axes immediately.
    \item Paths do not share edges, but they may share vertices.
\end{itemize}
An example of an ensemble of admissible paths is shown in \Cref{fig:s6v_conf}. Due to the aforementioned conditions, each vertex in $\Z_{\ge 1}^2$ has six possible configurations, shown in \Cref{fig:vertex weights}. 

To define the probability measure, it is necessary to specify an initial condition for the arrows. In this paper, we consider the \textit{step initial condition} in which all the vertices along $\{(1,n):n\in \mathbb{Z}_{\geq 1}\}$ have an incoming horizontal arrow from left, but none of the vertices in $\{(m,1):m\in \mathbb{Z}_{\geq 1}\}$ has any incoming vertical arrow from below,  see \Cref{fig:s6v_conf} for an example.  Given this initial condition, S6V model is defined in the following Markovian manner. Fix $\aw, q \in (0,1)$ and set $\mathbb{T}_n:=\{(x,y) \in \Lambda : x+y\le n\}$. Conditional on the incoming arrows from $\mathbb{T}_{n-1}$ we sample all the vertices in $\mathbb{T}_n\setminus \mathbb{T}_{n-1}$ according to the probabilities given in \Cref{fig:vertex weights}. The S6V model is then defined as a limit of these measures as $n\to\infty$.

  \begin{figure}[t]
      \centering
      \includegraphics[width = \linewidth]{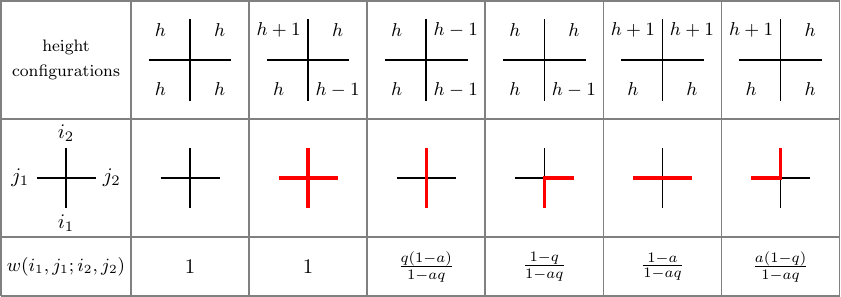}
      \caption{{Six possible local configurations and corresponding weights}}
      \label{fig:vertex weights}
  \end{figure}

    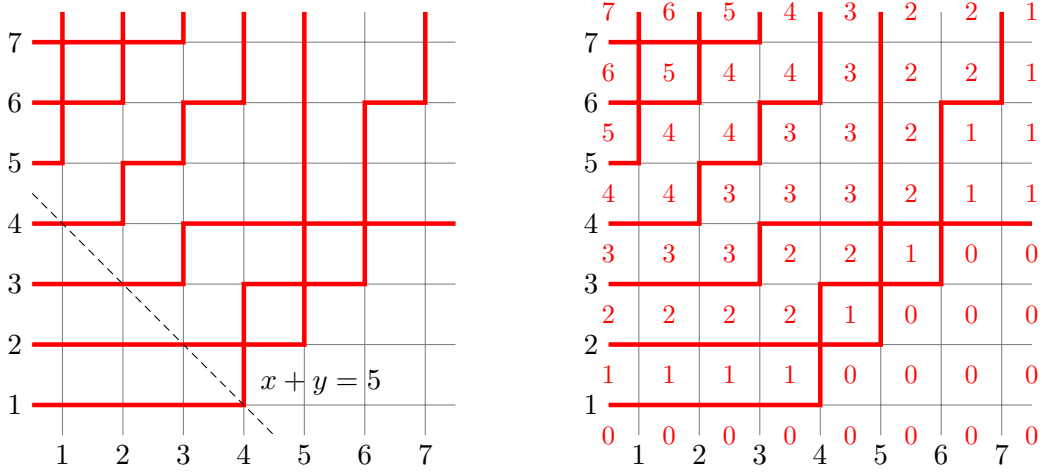
\begin{figure}[t]
        \centering
        \setlength{\tabcolsep}{2em}
        \begin{tabular}{cc}
        \begin{tikzpicture}[scale=.8]
            \foreach \xxx in {1,...,7}
			{
				\node[below] at (\xxx-2,-5.5) {$\xxx$};
                \node[left] at (-1.5,-6+\xxx) {$\xxx$};
			}
    		\draw[gray,step=1] (-1.5,-5.5) grid (5.5,1.5);
    						
    		\draw[red,line width=.06cm] (-1.5,1) -- (-1,1) -- (-1,1.5);
    		\draw[red,line width=.06cm] (-1.5,0) -- (-1,0) -- (-1,1) -- (0,1) -- (0,1.5);
    		\draw[red,line width=.06cm] (-1.5,-1) -- (-1,-1) -- (-1,0) -- (0,0) -- (0,1) -- (1,1) -- (1,1.5);
    		
    		\draw[red,line width=.06cm] (-1.5,-2) -- (0,-2) -- (0,-1) -- (1,-1) -- (1,0) -- (2,0) -- (2,1.5);
    		\draw[red,line width=.06cm] (-1.5,-3) -- (1,-3) -- (1,-2) -- (3,-2) -- (3,1.5);
    		\draw[red,line width=.06cm] (-1.5,-4) -- (2,-4) -- (2,-3) -- (3,-3) -- (3,-2) -- (4,-2) -- (4,0) -- (5,0) -- (5,1.5);
    		
    		\draw[red,line width=.06cm] (-1.5,-5) -- (2,-5) -- (2,-4) -- (3,-4) -- (3,-3) -- (4,-3) -- (4,-2) -- (5.5,-2);
            \draw[densely dashed] (-1.5,-1.5) --++ (4,-4) node[above,anchor=west,yshift=20,xshift=-9] {$x+y=5$};
        \end{tikzpicture}
        &
        \begin{tikzpicture}[scale=.8]
            \foreach \xxx in {1,...,7}
			{
				\node[below] at (\xxx-2,-5.5) {$\xxx$};
                \node[left] at (-1.5,-6+\xxx) {$\xxx$};
			}
    		\draw[gray,step=1] (-1.5,-5.5) grid (5.5,1.5);
    						
    		\draw[red,line width=.06cm] (-1.5,1) -- (-1,1) -- (-1,1.5);
    		\draw[red,line width=.06cm] (-1.5,0) -- (-1,0) -- (-1,1) -- (0,1) -- (0,1.5);
    		\draw[red,line width=.06cm] (-1.5,-1) -- (-1,-1) -- (-1,0) -- (0,0) -- (0,1) -- (1,1) -- (1,1.5);
    		
    		\draw[red,line width=.06cm] (-1.5,-2) -- (0,-2) -- (0,-1) -- (1,-1) -- (1,0) -- (2,0) -- (2,1.5);
    		\draw[red,line width=.06cm] (-1.5,-3) -- (1,-3) -- (1,-2) -- (3,-2) -- (3,1.5);
    		\draw[red,line width=.06cm] (-1.5,-4) -- (2,-4) -- (2,-3) -- (3,-3) -- (3,-2) -- (4,-2) -- (4,0) -- (5,0) -- (5,1.5);
    		
    		\draw[red,line width=.06cm] (-1.5,-5) -- (2,-5) -- (2,-4) -- (3,-4) -- (3,-3) -- (4,-3) -- (4,-2) -- (5.5,-2);

            \foreach \p in {(-1.5,-5.5), (-.5,-5.5), (.5,-5.5), (1.5,-5.5), (2.5,-5.5), (3.5,-5.5),(4.5,-5.5),(5.5,-5.5), (2.5,-4.5), (3.5,-4.5), (4.5,-4.5), (5.5,-4.5), (3.5,-3.5), (4.5,-3.5), (5.5,-3.5), (4.5,-2.5), (5.5,-2.5)}
            {
                \node[red] at \p {\small $0$};
            }

            \foreach \p in {(-1.5,-4.5),(-.5,-4.5),(.5,-4.5), (1.5,-4.5), (2.5,-3.5), (3.5,-2.5), (4.5,-1.5), (5.5,-1.5), (4.5,-.5), (5.5,-.5),(5.5,.5), (5.5,1.5)}
            {
                \node[red] at \p {\small $1$};
            }

            \foreach \p in {(-1.5,-3.5), (-.5,-3.5), (.5,-3.5), (1.5,-3.5), (1.5,-2.5), (2.5,-2.5), (3.5,-1.5), (3.5,-.5), (3.5,.5), (4.5,.5), (3.5,1.5), (4.5,1.5)}
            {
                \node[red] at \p {\small $2$};
            }

            \foreach \p in {(-1.5,-2.5), (-.5,-2.5),(.5,-2.5),(.5,-1.5),(1.5,-1.5),(2.5,-1.5),(1.5,-.5),(2.5,-.5),(2.5,.5),(2.5,1.5)}
            {
                \node[red] at \p {\small $3$};
            }

            \foreach \p in {(-1.5,-1.5),(-.5,-1.5),(-.5,-.5),(.5,-.5),(.5,.5),(1.5,.5),(1.5,1.5)}
            {
                \node[red] at \p {\small $4$};
            }

            \foreach \p in {(-1.5,-.5),(-.5,.5),(.5,1.5)}
            {
                \node[red] at \p {\small $5$};
            }

            \foreach \p in {(-1.5,.5),(-.5,1.5)}
            {
                \node[red] at \p {\small $6$};
            }

            \foreach \p in {(-1.5,1.5)}
            {
                \node[red] at \p {\small $7$};
            }
        \end{tikzpicture}
        \end{tabular}
            \caption{A sample of the stochastic six vertex model. In the right panel red numbers denote the height function.}
        \label{fig:s6v_conf}
    \end{figure}

The  main observable of interest in the S6V model is the height function $\h : \Lambda  \to \mathbb{Z}_{\ge 0}$ defined as follows:
\begin{align} \label{eq:height}
    \h(M,N):= \mbox{number of paths that pass through or to the right of }(M,N).
\end{align}
For the step initial condition, we have $\mathfrak{h}(M,N)\leq N$ since there is no path coming from the horizontal axis.

It was predicted in \cite{GwaSpohn1992} and later proven by Borodin, Corwin, and Gorin \cite{BCG6V} that the model belongs to the so-called Kardar-Parisi-Zhang (KPZ) universality class -- a class of models that exhibit universal scaling exponents and limiting
statistics first discovered in random matrix theory \cite{quastel_introduction_to_KPZ,CorwinKPZ,zygouras_review,ganguly2021random}. \cite{BCG6V} showed  that the height function satisfies the following convergence in probability: 
\begin{align*}
   \lim_{N\to \infty} \frac{\h(\lfloor\kappa N\rfloor,N)}{N} =\lln := \begin{dcases}
        1-\kappa, & \kappa\in (0,\aw), \\ 
        \frac{(1-\sqrt{\aw\kappa})^2}{1-\aw} & \kappa \in (\aw,\aw^{-1}), \\ 
        0 & \kappa \in (\aw^{-1},\infty).
    \end{dcases}
\end{align*}
The above convergence was upgraded to almost sure convergence in  \cite{drillick2023strong}. The choice of parameter $\kappa\in (a,a^{-1})$ is known as the \emph{liquid region} where the limit shape is curved and there \cite{BCG6V} proved that the height function has Tracy-Widom GUE fluctuations (with a negative sign) of order $N^{1/3}$. The complementary situations when parameter $\kappa\le a$ and $\kappa\ge a^{-1}$ correspond to the \emph{frozen region} where the height function is flat and one expects the fluctuations to be exponentially small. Since the work \cite{BCG6V}, there has been immense progress in understanding various aspects of this model. Connections to determinantal point processes \cite{borodin2016stochastic_MM,BO2016_ASEP}, limit shape and fluctuation theorems \cite{Amol2016Stationary,aggarwal2020limit,corwin2020stochastic,dimitrov2023two}, and boundary-induced phase transitions \cite{abphase} have been established. In a very recent breakthrough \cite{ach24}, the S6V model has been shown to converge to the directed landscape \cite{DOV18}, a universal scaling limit in the KPZ universality class.

The present paper focuses on the study of large deviations of the S6V model: the rare events where the height function 
$\h(\lfloor\kappa N\rfloor,N)$ deviates by an order $N$ from its mean $\lln N$. Interestingly, we expect different speeds for the upper and lower deviations. As $N\to \infty$ we expect
\begin{align}\tag{Upper Tail}
   & \Pr(\h(\lfloor\kappa N\rfloor,N) \le sN) \asymp e^{-N\Phi_\kappa^{(+)}(s)} & s\in [0,\lln], \\ \tag{Lower Tail} 
    & \Pr(\h(\lfloor\kappa N\rfloor,N) \ge sN) \asymp e^{-N^2\Phi_\kappa^{(-)}(s)} & s\in [\lln,1].
\end{align}
Since $-\h(\lfloor\kappa N\rfloor,N)$ converges to the Tracy-Widom GUE after centering and scaling, we refer to $\{\h(\lfloor\kappa N\rfloor,N) \le sN\}$ as the upper tail and $\{\h(\lfloor\kappa N\rfloor,N) \ge sN\}$ as the lower tail, aligning with the existing literature. The asymmetry above in the speeds, $N$ versus $N^2$, is a hallmark of the KPZ universality class.  Heuristically, this can be understood from the path configurations of the S6V model. To achieve a large height at a given point, all paths need to stay low, whereas a small height can be attained if the lowest path trends upward.

We focus on the lower tail large deviations which is arguably more challenging and often requires more involved machineries. Our first main result in this direction shows that the lower-tail probabilities satisfy certain weak log-concavity.

\begin{theorem}\label{t.main0}
    Fix any $M,N \ge 8$ and $r_1, r_2\ge 0$. Then, we have
    \begin{align} \label{eq:weaklcasepbdintro}
        \Pr\big(\h(M,N) \ge  r_1\big)\Pr\big(\h(M,N) \ge  r_2\big) \le N^2\mathsf{C}^{(MN)^{4/5}} \left[\Pr\Big(\h(M,N) \ge \frac{r_1+r_2}{2}-4(MN)^{2/5}\Big)\right]^2,
    \end{align}
    where $\mathsf{C}:=\aw^{-2}q^{-2}(1-\aw)^{-2}(1-q)^{-2}(1-\aw q)^{4}$.
\end{theorem}

The exponents $4/5$ and $2/5$ are not optimal. Note that in the lower-tail large deviation regime, i.e., when $M=\lfloor \kappa N\rfloor$ and $r_1,r_2>\lln N$, the probabilities appearing above are of order $e^{-O(N^2)}$ and thus the prefactor $N^2\mathsf{C}^{(MN)^{4/5}}$ is negligible. 

\medskip

\Cref{t.main0} is a key ingredient in proving a lower tail Large Deviation Principle (LDP) for the stochastic six-vertex model stated below.

\begin{theorem}\label{thm.main} For all $s\in [\lln,1]$ we have
    \begin{align}\label{e.main}
        \lim_{N\to \infty}-\frac1{N^2}\log\Pr(\h(\lfloor \kappa N \rfloor,N) \ge sN) = \Phi_\kappa^{(-)}(s).
    \end{align}
    $\Phi_\kappa^{(-)}$ is a non-decreasing, non-negative convex function with $\Phi_\kappa^{(-)}(\lln)=0$ and $\Phi_\kappa^{(-)}(1)=\kappa \log\frac{1-aq}{1-a}$. It has the following variational form
    \begin{align}\label{defphi}
        \Phi_\kappa^{(-)}(s):=\sup_{y\in \R} \left[\mathcal{F}_\kappa(y)-(\log q^{-1})\frac{(s-y)^2}{2}\right]
    \end{align}
    where $\mathcal{F}_\kappa(y)$ is a certain energy integral explained in \Cref{sec.rate}
\end{theorem}

While we do not address the upper tail LDP in this paper, it should be achievable using a perturbative analysis similar to \cite{dt21,lin2020kpz,dz1,das2023large}, leveraging the Fredholm determinant formula available for the S6V model \cite{borodin2016stochastic_MM,BO2016_ASEP}.

The lower tail case, on the other hand, is more subtle; perturbative analysis of the Fredholm determinant is no longer feasible. In this context, different techniques using connections to determinantal point processes have proven fruitful. This connection goes back to the work of Borodin and Olshanski \cite{BO2016_ASEP}, where they observed that the $q$-Laplace transform of the height function of the S6V model can be viewed as an expectation of a certain multiplicative functional under Schur measures. See also \cite{borodin2016stochastic_MM} for more general moment matching formulas between higher-spin S6V models and Macdonald measures.
It was first noted in \cite{lwtail} that it is worthwhile to analyze such multiplicative functionals to obtain lower tail estimates of the observables. This approach was carried out in \cite{lwtail} for the KPZ equation introduced in \cite{KPZ1986}, where the $q$-Laplace transform becomes the usual Laplace transform and the underlying point process is the $\mathrm{Airy}_2$ point process. Such analysis is by no means trivial, and a host of sophisticated techniques, ranging from the analysis of the stochastic Airy operator to Riemann-Hilbert problems, have been developed to analyze the underlying functional of the $\mathrm{Airy}_2$ point process \cite{tsai_lower_tail,cafasso_claeys_KPZ}.

The approach we take in the current paper is also based on the connection from \cite{borodin2016stochastic_MM,BO2016_ASEP} where the underlying point process is the Meixner ensemble. While the applicability of the aforementioned techniques to the Meixner ensemble is yet to be explored, the analysis of S6V model presents another challenge due to the presence of the $q$-Laplace transform instead of the usual Laplace transform. Loosely speaking, the $q$-Laplace transform of the height function observable is related to the lower tail probability of a randomly shifted height function. The shift has Gaussian tails and has a non-trivial effect in the rate function. In essence, the lower-tail rate functions of the shifted and unshifted height function \textit{are not the same} (unlike in the case of the KPZ equation). 

For the $q$-PNG model, where one faces similar challenges, we proposed a strategy in our earlier work \cite{das2023large} to negate the effect of this shift. This approach crucially uses the fact that the tail probabilities of the $q$-PNG height function (after a shift) are log-concave, stemming from the log-concavity of the Schur polynomials \cite{Okounkov1997,Lam_Postnikov_Pylyavskyy_concavity}. {In the present paper the same idea cannot be applied. In fact, while there still exists a precise connection between S6V and {periodic Schur measures}, guaranteed by a result of \cite{IMS_KPZ_free_fermions}, the latter in this instance are not non-negative. Therefore their log-concavity cannot be leveraged to deduce probabilistic properties of the S6V height function.

To bypass this problem we devise an alternative approach, that proves \Cref{t.main0} without relying on special algebraic structures which are consequence of the integrability of the S6V. The proof, which is outlined in \Cref{subs:convexity intro} and fully fleshed out in \Cref{sec:ltconvex}, can be thought as a multi-dimensional variant of the Gessel-Viennot involution \cite{gessel1985binomial}, which is an elegant and simple operation commonly seen in algebraic combinatorics. Despite the inspiration coming from algebraic combinatorics, our approach primarily depends on the Gibbsian structure of the growth process. It is amenable to broad generalizations, including Gibbsian growth processes of a more general nature in arbitrary dimensions -- ideas which we hope to pursue further in future works.}

As already pointed out the bound \eqref{eq:weaklcasepbdintro} is not optimal since it includes (a) an exponentially diverging factor in the right hand side and (b) a sublinearly diverging shift in the argument of the right hand side probability. A valid question, to which we have no answer, is if such diverging terms can be abated or removed all together. Such a question connects the world of integrable probability with that of the raising trend in algebraic combinatorics of studying log-concavity of algebraic objects such as multiplicities of representations or special symmetric polynomials \cite{Okounkov1997,Okounkov2003_why_would,huh_et_al_schur_log_concavity}. In fact, it is known \cite{BorodinBufetovWheeler2016} that certain multi-point distributions of the height function of the S6V is described by the Hall-Littlewood polynomials, which are one parameter deformations of the Schur polynomials. A pure log-concavity result for the tails of the height function $\mathfrak{h}$, would correspond to a log-concavity result for the Hall-Littlewood polynomials, which indeed appears as a challenging though enticing direction to pursue (see also \cite{BufetovGorinLogConcavity} for similar considerations).

Soon after our work, \cite{bkm24} proved several log-concavity results for a class of one dimensional log-gases and of Tracy-Widom type laws (see also \cite{bls17}). Although there is no overlap between results of this paper and \cite{bkm24}, the construction we present could be adapted, in the case $q=0$, to give an alternative construction of some of their results, including log concavity of Tracy Widom distributions.

{Recently, a variety of probabilistic techniques have emerged, successfully yielding lower tail estimates in the moderate deviations (MD) regime for models within the KPZ class \cite{aggarwal2023asep,ganguly_hegde2023optimal,ganguly2022sharp,landon2022tail,landon2023upper,corwin_hegde2022lower,landon2023tail,gs24,hindy24}. The work of \cite{landon2023tail} derives tail estimates in the MD regime for the height function of the stationary S6V model using coupling arguments. For the step initial condition, S6V tail estimates in the MD regime appear in the forthcoming works \cite{hindy24,gs24}. Both papers start from the same identity from \cite{BO2016_ASEP} that we use, which relates the $q$-Laplace transform to a multiplicative statistic of the holes of the Meixner ensemble. \cite{hindy24} uses determinantal techniques to derive tail estimates for the position of the smallest hole in the Meixner ensemble and thereby deduces tail estimates of $\h(\lfloor \kappa N\rfloor,N)$ in the MD regime. \cite{gs24}, on the other hand, uses Riemann-Hilbert problem machineries to derive precise asymptotic behaviors of these multiplicative statistics in the MD regime. It would be interesting to see how much of these techniques can be extended to the large deviation regime.}

Techniques we develop in this paper could, in principle, lead to establish large deviation principle for the Asymmetric Simple Exclusion Process (ASEP), which arises as a scaling limit of the S6V. Under such scaling limit the connection between the height function of the ASEP and a multiplicative functional of a determinantal point process survives, the latter being the discrete Laguerre ensemble, which is a certain limit of the Meixner ensemble at the hard edge. For this, two approaches appear to be natural: constructing a coupling between the S6V and the ASEP strong enough to capture rare events or alternatively study the large deviations of the discrete Laguerre ensemble. Each of these approaches present issues one needs to resolve. Couplings between S6V and ASEP which have been produced in literature (such as \cite{Aggarwal_6v_to_ASEP,ach24}), while useful to control fluctuations, do not seem to be suitable to control rare events. Thus some of their modifications should be considered. On the other hand, establishing a large deviation principle of the discrete Laguerre ensemble, which is not a log gas, appears to require the use of more sophisticated techniques, e.g., the Riemann-Hilbert techniques, instead of simpler Laplace-type arguments. We leave this for future work.

We end our discussion by mentioning few works related to our paper. For zero temperature models such as the totally asymmetric simple exclusion process, one-point LDPs were obtained using various techniques in \cite{logan_shepp1977variational,sep98a,seppalainen_98_increasing,deuschel_zeitouni_1999,johansson2000shape} and process-level LDPs were studied in \cite{jen00,var04,ot19,qt21}. The lower-tail LDP for the first passage percolation was established in \cite{basu2021upper} at the one-point level and recently in \cite{verges2024large} at the metric level. For the universal scaling limit, a metric level LDP for the directed landscape was recently proven in \cite{das2024upper}, leading to a process-level LDP for the $\mathrm{Airy}_2$ process in \cite{das2024solving}. Besides the above mentioned papers, one-point tails for the KPZ equation were also studied in the physics works \cite{ledoussal16long,ledoussal16short,krajenbrink17short,sasorov2017large,corwin2018coulomb,krajenbrink18half,krajenbrink18simple,krajenbrink2018systematic} and in the mathematics works \cite{dt21,kim21,corwin20general,gl20}, and process-level large deviations were studied in the physics works \cite{kolokolov07,kolokolov09,meerson16,meerson17,meerson18,kamenev16,hartmann19,krajenbrink21,krajenbrink22flat,krajenbrink23} and the mathematics works \cite{lin21,gaudreaulamarre2023kpz,lin22,ganguly2022sharp,ganguly2023brownian,lin23,tsai2023high}.

\subsubsection{Description of the rate function} \label{sec.rate}

To describe $\mathcal{F}_\kappa$ appearing in \eqref{defphi}, we need some basic notions and results from potential theory. For context and significance on potential theory, we refer to the excellent monograph of Saff and Totik \cite{saff1997logarithmic}.

Given a continuous function $V: [0,\infty)\to \mathbb{R}$ (often called external field or potential) satisfying $V(x) \ge 2\log(1+x^2)$ for large enough $x$, and  a probability measure on $\mu\in (\mathbb{R}_+,\mathcal{B}(\mathbb{R}_+))$, we define the logarithmic energy integral
\begin{align}\label{e.IVmu}
    I_V(\mu):=\iint_{x\neq y} k_V(x,y)\diff \mu(x) \diff \mu(y), \qquad k_V(x,y):=-\log|x-y|+\frac12V(x)+\frac12V(y).
\end{align}
For $s\in (0,\infty]$, let $\mathcal{A}_s$ denote the set of all $\phi\in L^1([0,s])$ satisfying
\begin{align}\label{cala}
0 \le \phi(x) \le 1 \mbox{ for Lebesgue a.e.}~x\in [0,s], \mbox{ and }\int_0^s \phi(x)\diff x=1.
\end{align}
Given a $\phi\in \mathcal{A}_{s}$, we will denote $\mu_{\phi} \in ([0,s],\mathcal{B}([0,s]))$ to be the probability measure with density $\phi$. 

One of the fundamental problem of potential theory is to minimize the energy integral over all possible measures. We quote a classical result from \cite{ds97} in this direction: There exists a unique $\phi_V \in \mathcal{A}_\infty$ with compact support such that
\begin{align}\label{l.ds97}
    \inf_{\phi \in \mathcal{A}_\infty} I_V(\mu_\phi)=I_V(\mu_{\phi_V})=:F_V \in \R. 
\end{align}
$\mu_{\phi_V}$ and $F_V$ are known as the equilibrium measure and the equilibrium energy for the potential $V$. 

\medskip

We now describe $\mathcal{F}_\kappa$ appearing in \eqref{defphi}. For $y\in (-\infty,\infty]$ and $\kappa\ge 1$ we define the external field:
\begin{equation}\label{def.vs}
\begin{aligned}
    \mathpzc{V}_y(x) & =(\log a^{-1}) x+(\log q^{-1})\min(x,y) -(\kappa-1)[\log(x+\kappa-1)-1]+x\log \frac{x}{x+\kappa-1}.
\end{aligned}
\end{equation}
 We have suppressed the dependency on $\kappa$ above. For $\kappa\ge 1$, define
\begin{align}
    \label{def:fs}
     \mathcal{F}_\kappa(y):=(\log q^{-1})\frac{y^2-2y}{2}-\kappa\log (1-a)+\frac{\log(a)}{2}+C_{\kappa}+F_\kappa(y),
\end{align}
where $F_\kappa(y):=F_{\mathpzc{V}_y}$ comes from \eqref{l.ds97} (with $V\mapsto \mathpzc{V}_y$) and 
\begin{align}\label{def.ca}
    C_{\kappa}:=-\frac{(\kappa -1)^2 \log (\kappa -1) -\kappa^2  \log (\kappa )+3\kappa}{2}.
\end{align}
When $\kappa\in (0,1)$ define
\begin{equation}\label{def:fs2}
   \begin{aligned}
    \mathcal{F}_\kappa(y) & :=(\log q^{-1})\frac{y^2-2y\kappa+2\kappa-2\kappa^2}{2}-(\log q^{-1})\int_{\kappa}^1 (y+x-1)_+\diff x \\ & \hspace{2cm}-\kappa\log (1-a)+\kappa^2\frac{\log(a)}{2}+\kappa^2C_{\kappa^{-1}}+\kappa^2 F_{\kappa^{-1}}(y\kappa^{-1}-\kappa^{-1}+1).
\end{aligned} 
\end{equation}

The following theorem summarizes the properties of $\mathcal{F}_\kappa$.

\begin{theorem}\label{thm.main2} The function $\mathcal{F}_\kappa$ is a non-decreasing, non-negative, convex, and continuously differentiable function with derivative being $(\log q^{-1})$-Lipschitz. We have $\mathcal{F}_\kappa(\lln)=0$ and there exists $x_0(a,q,\kappa)>0$ such that for all $x\ge x_0$
\begin{align}\label{parab}
    \mathcal{F}_\kappa(x)=(\log q^{-1})\frac{(x-1)^2}{2}+\kappa \log \frac{1-aq}{1-a}.
\end{align}
\end{theorem}

\subsection{Proof ideas.} \label{sec:pfidea} In this section we discuss the key ideas behind the proof of our main theorems.

\begin{figure}
\setlength{\tabcolsep}{2em}
\begin{tabular}{cc}
  \includegraphics[width=60mm]{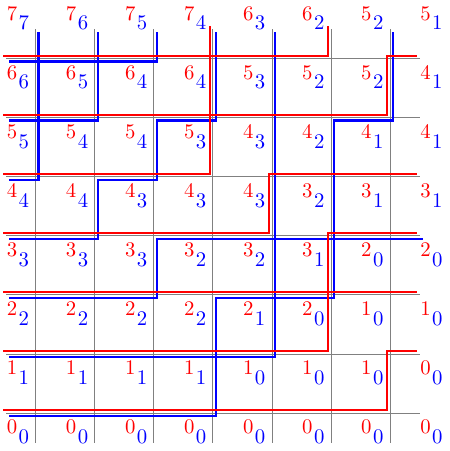} &    \includegraphics[width=60mm]{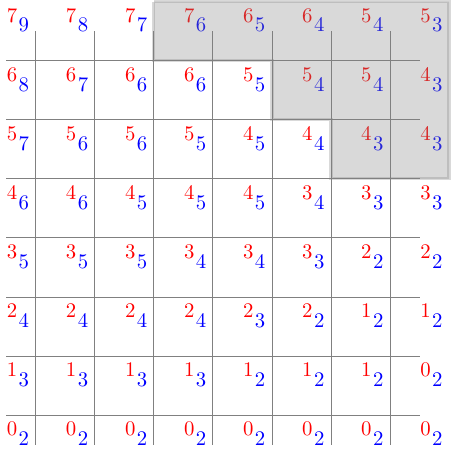} \\
(a) & (b) \\[6pt]
 \includegraphics[width=60mm]{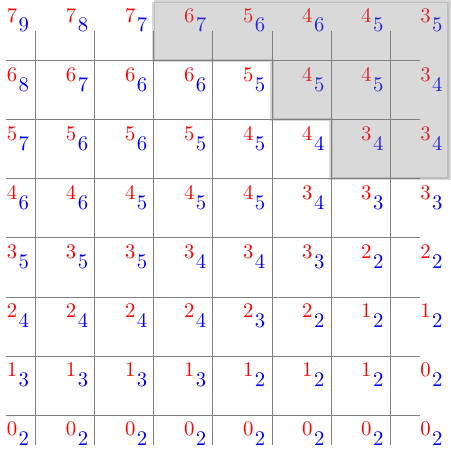} &   \includegraphics[width=60mm]{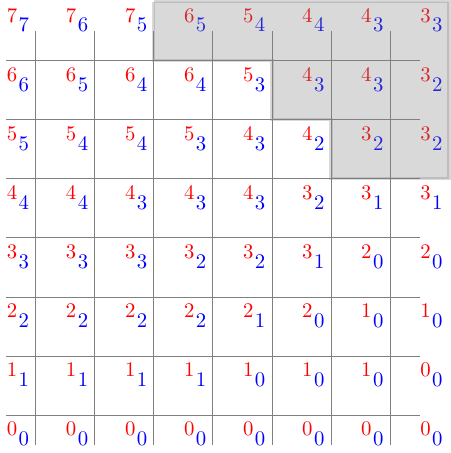} \\
(c) & (d) \\[6pt]
\includegraphics[width=60mm]{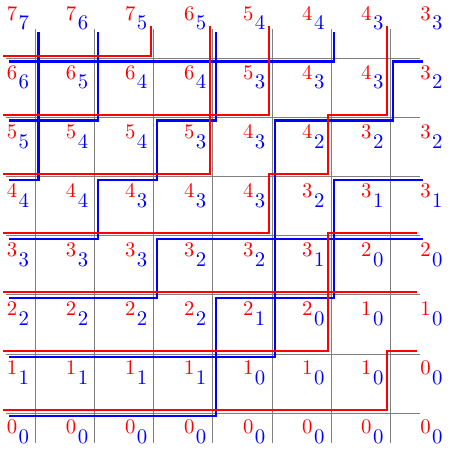} &    \\
(e) &  \\[6pt]
\end{tabular}
\caption{A depiction of the injective mapping $(h_1,h_2)\mapsto (\widetilde{h}_1,\widetilde{h}_2)$ described in \Cref{subs:convexity intro}. In panel (a) the blue and red heights are respectively $h_1$ and $h_2$, while in panels (d), (e) blue and red heights are respectively $\widetilde{h}_1$ and $\widetilde{h}_2$. Here we have $N=M=7$ and with the choice $k=2$. We end up with $\widetilde{h}_1(M,N)=\widetilde{h}_2(M,N) = 3$, which equals the midpoint $\frac{1}{2}\left( h_1(M,N) + h_2(M,N) \right)$.} \label{fig:injection}
\end{figure}

\subsubsection{Log concavity of lower-tail probabilities} \label{subs:convexity intro}
 The proof of \Cref{t.main0} relies on an injective argument. We prove that, given a pair of height functions $h_1,h_2$ associated with stochastic six vertex models with step boundary conditions and such that $h_1(M,N)=r$, $h_2(M,N)=r'$, we can construct, injectively, another pair of heights $\widetilde{h}_1, \widetilde{h}_2$ such that
    \begin{equation} \label{eq:inequality htilde intro}
         \frac{r+r'}{2} - (MN)^{2/5} \le \widetilde{h}_1(M,N),  \widetilde{h}_2(M,N) \le  \frac{r+r'}{2} + (MN)^{2/5}
    \end{equation}
    and importantly
    \begin{equation}\label{eq:inequality weight intro}
        \prod_{v=1}^2 \mathbb{P}(\mathfrak{h}(i,j)=h_v(i,j); (i,j)\in \Lambda_{M,N}) \le \mathsf{C}^{(MN)^{4/5}} \prod_{v=1}^2 \mathbb{P}(\mathfrak{h}(i,j)= \widetilde{h}_v(i,j); (i,j)\in \Lambda_{M,N}).
    \end{equation}
From the above relation, a simple computation allows one to extend weak log-concavity for lower-tail probabilities.  The injective mapping of height functions $(h_1,h_2) \to (\widetilde{h}_1,\widetilde{h}_2)$ is described in detail in \Cref{sec:ltconvex} and it is reminiscent of the Gessel-Viennot involution \cite{gessel1985binomial} for pairs of one dimensional paths. The procedure is illustrated in the various panels of \Cref{fig:injection} and we elaborate it next. Given the two heights $h_1,h_2$ we consider their overlayed plot as in \Cref{fig:injection}(a). Assuming that $h_1(M,N)=r < r' =h_2(M,N)$, we lift up rigidly the entire height function $h_1(\cdot,\cdot)$ by $k$ units, denoting the result by $\tau_k(h_1)$, where the specific $k$ will be decided at a later stage; this is shown in \Cref{fig:injection}(b), where we took $k=2$. By construction $\tau_k(h_1)(M,N)=r+k$. We will assume that $k < r'-r$, in which case we have $\tau_k(h_1)(M,N)<h_2(M,N)$ and therefore there exists a non-empty region
    \begin{equation*}
         \left\{ (i,j): \tau_k(h_1)(i,j) < h_2(i,j)\right\},
    \end{equation*}
    which possesses a connected component (in the topology induced by the Manhattan metric) containing $(M,N)$ which we denote by $R_{M,N}(\tau_k(h_1)<h_2)$. In \Cref{fig:injection}(b) such region is colored in gray. We now swap the values of the two heights $\tau_k(h_1), h_2$ only on the region $R_{M,N}(\tau_k(h_1)<h_2)$, as shown in \Cref{fig:injection}(c). Such operation is denoted by $(\overline{h}_1, \overline{h}_2)=\iota_{M,N}(\tau_k(h_1),h_2)$. We finally proceed to shift downward by $k$ the height $\overline{h}_1$, operation which re-establishes the correct boundary conditions; see \Cref{fig:injection}(d). Keeping the notation consistent with \Cref{sec:ltconvex}, this final operation defines
    \begin{equation*}
        (\widetilde{h}_1,\widetilde{h}_2) = \Upsilon_{(M,N)}^{k} (h_1,h_2),
        \qquad \text{with} \qquad \Upsilon_{(M,N)}^{k} = (\tau_{-k}\times \mathrm{Id}) \circ \iota_{(M,N)} \circ (\tau_{k}\times \mathrm{Id}).
    \end{equation*}
    We can finally transform the height configurations back to paths configurations as done in \Cref{fig:injection}(e).

    The sequence of transformations described above is not necessarily weight preserving, in the sense that it is not true that the weights corresponding to $\widetilde{h}_1,\widetilde{h}_2$ are just a permutation of those corresponding to $h_1,h_2$. Nevertheless, it is not hard to show that differences between  weights corresponding to pairs $h_1,h_2$ and $\widetilde{h}_1,\widetilde{h}_2$ are only possible along the boundaries of the region $R_{M,N}(\tau_k(h_1)<h_2)$. One can be convinced of this claim after checking in \Cref{fig:injection}. This is the point where we choose our shifting index $k$: we pick $k=k^*_{M,N}(h_1,h_2)$ to be the minimal number greater than $\frac{r'-r}{2}$ such that the length of the boundary of $R_{M,N}(\tau_k(h_1)<h_2) \cap \Lambda_{M,N}$ does not exceed $(MN)^{4/5}$. A pigeonhole principle argument (see the proof of Lemma \ref{lem:lift by 4} and Lemma \ref{lem:bound k} for the details) shows that such choice of $k$ will satisfy the bound
    \begin{equation*} 
        \left\lceil \frac{r'-r}{2} \right\rceil \le k^*_{M,N}(h_1,h_2) \le \left\lceil \frac{r'-r}{2} \right\rceil + (MN)^{2/5},
    \end{equation*}
    guaranteeing that the values of the transformed heights $\widetilde{h}_1,\widetilde{h}_2$ satisfy the desired inequalities \eqref{eq:inequality htilde intro} and simultaneously \eqref{eq:inequality weight intro}. Here the factor $\mathsf{C}^{(MN)^{4/5}}$ appears to bound the maximal difference in  weights which the injection produces. Since the choice of $k$ forces the number of vertices where such  weight mismatch can occur to be less than $(MN)^{4/5}$, the exponent is also justified.

\subsubsection{Lower tail LDP} The starting point of our proof of \Cref{thm.main} is an identity from \cite{borodin2016stochastic_MM,BO2016_ASEP} {which, in view of an identity from \cite{IMS_matching}}, can be read as follows:
\begin{align}\label{pfidea1}
    \Pr\big(\mathfrak{h}(\lfloor\kappa N\rfloor,N)-\chi-S \ge sN\big) = \left[\prod_{j=0}^\infty \frac{1}{1+q^{j-sN}}\right]\mathbb{E}\left[\prod_{j=0}^{N-1}\Big(1+q^{\lambda_{N-j}+j-sN}\Big)\right].
\end{align}
Here $\chi+S$ is a random shift independent of $\mathfrak{h}(\cdot,\cdot)$ with an explicit distribution (given in \eqref{chisd}) and in the right-hand side  $\lambda$ is distributed as a Schur measure with a certain positive specialization (see \eqref{eq:schurmeasure}). Since the work of \cite{johansson2000shape}, it has been known that such Schur measures can be viewed as \textit{discrete log-gases}.  Discrete and continuous log-gases, on the other hand, are extensively studied in random matrix theory and beyond (see the books \cite{AndersonGuionnetZeitouniBook,deift2000orthogonal,Forrester-LogGas,mehta2004random,pastur2011eigenvalue} for a review), and have a close connection with potential theory. This connection allows one to derive large deviations for the empirical measure $\sum_{j=0}^{N-1} \delta\big((\lambda_{N-j}+j)/N\big)$. Using a Varadhan's lemma type argument from this point,  the asymptotics of the right-hand side of \eqref{pfidea1} can be determined. This eventually leads to
 \begin{align}\label{pfidea15}
        \lim_{N\to \infty}-\frac1{N^2}\log\Pr(\h(\lfloor \kappa N \rfloor,N)-\chi-S \ge sN) = \mathcal{F}_\kappa(s),
    \end{align}
where $\mathcal{F}_\kappa(s)$ is given in \eqref{def:fs} and \eqref{def:fs2}. The details of the derivation of \eqref{pfidea15} are presented in \Cref{sec.2.1}. 

\medskip

To prove \eqref{e.main} from here, we need to negate the effect of the shift $\chi+S$. It turns out that the lower tail of $\chi+S$ is Gaussian, with $\Pr(\chi+S \le -yN) \asymp e^{-N^2g(y)}$ where $g(y):=\big(\log q^{-1}\big)y^2/2$, and hence {the presence of this random shift} has a non-trivial effect {on the limit \eqref{pfidea15}}. Indeed, if we define
\begin{equation*}
    \Phi_{\kappa,N}(y) \coloneqq -\frac1{N^2}\log\Pr(\h(\lfloor \kappa N \rfloor,N) \ge yN),
\end{equation*}
we can deduce (\Cref{prop:f_limit_g+T} in the text) that
\begin{align}\label{pfidea2}
    \mathcal{F}_\kappa(y) =\lim_{N\to \infty}\inf_{x\in \R} \{g(x)+\Phi_{\kappa,N}(y-x)\}.
\end{align}
To show that $\Phi_{\kappa,N}$ has a limit, we use a deconvolution scheme from \cite{das2023large} which relies on two key components:
\begin{itemize}
\setlength\itemsep{0.2em}
    {
    \item \emph{The rate function $\mathcal{F}_\kappa(y)$ possesses, for $y$ large, an explicit parabolic behavior}.  
    \item \emph{The pre-limit rate function $\Phi_{\kappa,N}$ is approximately convex}, as implied by \Cref{t.main0}.
    }
\end{itemize}
{In order to establish the precise behaviour at large $y$ of the rate function $\mathcal{F}_\kappa(y)$, we show that, for all large enough $y$, we have $F_\kappa(y)=F_\kappa(\infty)$, where $F_{\kappa}$ is the energy integral defined below \eqref{def:fs}, and then compute $F_\kappa(\infty)$ explicitly. This relies heavily on potential theory, and we refer the readers to \Cref{sec.2.2} for details. A result of such computation is that the exact expression in \eqref{parab} in fact is equal to $g(y-1)+\Phi_{\kappa,N}(1)$. This fact is quite important; as in view of \eqref{pfidea2}, it implies an approximate continuity of $\Phi_{\kappa,N}$ at $1$: 
\begin{align*}
    & \mbox{For any $\e>0$, there exists $\delta>0$ (free of $N$) such that } 
\Phi_{\kappa,N}(1)-\Phi_{\kappa,N}(1-\delta) \le \e \mbox{ for all large }N.
\end{align*} 
Approximate convexity of $\Phi_{\kappa,N}$ then allows us to upgrade the above approximate continuity to equicontinuity of $\{\Phi_{\kappa,N}\}_N$ and hence $\Phi_{\kappa,N}$ has a limit{, by the Ascoli-Arzela theorem}. {Finally, any such limit can be expressed uniquely, using its convexity and properties of the Legendre-Fenchel transform, as the infimal deconvolution between $\mathcal{F}_\kappa$ and the parabola $g$, as in \eqref{defphi}, hence proving that it is unique.} The details {of this procedure} are given in \Cref{prop:equicont}.
}

\medskip

{We mention that our deconvolution scheme bears resemblance to Seppäläinen's deconvolution idea, used in proving the upper-tail LDP for the polynuclear growth (PNG) model \cite{seppalainen_98_increasing}. Seppäläinen's proof was based on a coupling between the PNG model and a suitable particle system that admits certain known stationary initial conditions. This coupling provides a variational relation between the upper-tail rate function of the PNG model and the stationary model (similar to \eqref{pfidea2}). The rate function in the stationary model can be computed explicitly. Using the variational relation, one then derives an explicit form for the rate function in the original model.} 
 
 {The coupling approach has been fruitful in obtaining upper-tail LDPs for various other models in the KPZ universality class \cite{sepp98mprf,georgiou2013large,janjigian2015large,emrah,ciech,jan19}. We stress that the existence of the upper-tail rate function in these models can be deduced from a standard subadditive argument \cite{kesten}. Thus, it is the explicit form of the upper-tail rate function that is of interest in these models. In our case, however, the existence of the rate function is non-trivial. Due to the lack of an explicit form for $\mathcal{F}_\kappa$, we currently do not have an explicit form for $\Phi_\kappa^{(-)}$.}

\subsection*{Outline.} In \Cref{sec:ltconvex}, we derive our weak log-concavity result for the tail probabilities of the height function: \Cref{t.main0}. In \Cref{sec: LDP_schur} we prove the lower-tail LDP for the shifted height function.  This section can be read independently of \Cref{sec:ltconvex}. Finally, in \Cref{sec.4} we apply the deconvolution scheme and prove \Cref{thm.main} and \Cref{thm.main2}.
The appendices \Cref{sec.appb} and \Cref{sec.app} contain, respectively, some potential theory estimates and equilibrium measure calculations. As both involve fairly standard arguments, we relegate them to the appendix.

\subsection*{Notation and Conventions} Throughout the paper we fix $\aw,q\in (0,1)$ and $\kappa\in (0,\infty)$. We use the notation $\logq:=\log q^{-1}$. We use $\Con(x,y,z,\ldots)>0$ to denote a generic
deterministic positive finite constant that is dependent on the designated variables $x,y,z,\ldots.$ All our constants may depend on $a,q,\kappa$; we will not mention this further. $\#|A|$ denotes the cardinality of a finite set $A$.

\subsection*{Acknowledgements}
{We thank Ivan Corwin for his feedback and suggestions on an earlier draft of the paper. We also thank Hindy Drillick and Promit Ghosal for sharing their upcoming works, \cite{hindy24} and \cite{gs24} respectively, with us. We thank the anonymous referees for their careful reading and useful comments
on improving our manuscript.
The work of YL was partially supported by EPSRC via grant EP/X03237X/1.}

\section{Log-concavity bounds for the lower tail probabilities}  \label{sec:ltconvex}

The goal of this section is to deduce log-concavity bounds for the tail probabilities of the S6V model. As explained in the introduction, the proof relies on an injective argument that heavily utilizes surface geometry of the height function of the S6V model. In \Cref{sec3.1}, we derive certain properties of the underlying geometry of height surfaces and define certain class of bijective contraction maps. In \Cref{sec3.2}, using the maps and the geometry, we prove \Cref{t.main0}.

\subsection{Surface geometry of height function} \label{sec3.1} In this section, we introduce a space $\mathscr{H}$ for the height functions of the S6V model and discuss some geometric properties of the space $\mathscr{H}$. 

Towards this end, we define  the dual lattice $\Lambda' :=  \mathbb{Z}_{\ge0}' \times \mathbb{Z}_{\ge0}'$, where $\mathbb{Z}'_{\ge0}=\mathbb{Z}_{\ge0}+\frac{1}{2}$, which we view as the set of faces of $\Lambda=\Z_{\ge 0}^2$; see \Cref{fig:boundary set}.
For every vertex $v \in \Lambda$ its adjacent faces are defined as
\begin{center}
    \begin{minipage}[c]{0.3\textwidth}
        \centering
        \begin{tikzpicture}
            \draw[thick] (-1.5,0) -- (1.5,0);
            \draw[thick] (0,-1.5) -- (0,1.5);
            \draw[fill] (0,0 ) circle(.1cm);
            \node[] at (-.8,.8) {$\mathsf{NW}(v)$};
            \node[] at (.8,.8) {$\mathsf{NE}(v)$};
            \node[] at (-.8,-.8) {$\mathsf{SW}(v)$};
            \node[] at (.8,-.8) {$\mathsf{SE}(v)$};
            \node[] at (.25,.25) {$v$};
        \end{tikzpicture}
    \end{minipage}
    \begin{minipage}[c]{0.6\textwidth}
        \centering
        \begin{equation}
        \label{dir}
        \begin{split}
            \mathsf{NW}(v) = v - \frac{\mathbb{e}_1}{2} + \frac{\mathbb{e}_2}{2},
            \qquad
            &\mathsf{NE}(v) = v + \frac{\mathbb{e}_1}{2} + \frac{\mathbb{e}_2}{2}
            \\
            \mathsf{SW}(v) = v - \frac{\mathbb{e}_1}{2} - \frac{\mathbb{e}_2}{2},
            \qquad
            &\mathsf{SE}(v) = v + \frac{\mathbb{e}_1}{2} - \frac{\mathbb{e}_2}{2}.
        \end{split}
        \end{equation}
    \end{minipage}
\end{center}
For any region $R \subset \Lambda'$, we define its boundary $\partial R \subset \Lambda$ as the set of vertices $v \in \mathbb{Z}_{\ge 1}^2$ adjacent to both faces in $R$ and in $R^\mathrm{c}$; see \Cref{fig:boundary set}.

{For a function $h:\Lambda' \to \Z$ we introduce the discrete gradients
\begin{equation*}
    \nabla_1 h(p) \coloneqq h(p + \mathbb{e}_1) - h(p),
    \qquad
    \nabla_2 h(p) \coloneqq h(p + \mathbb{e}_2) - h(p)
\end{equation*}
and we define the collection
\begin{align}
    \mathscr{H} & :=\Big\{ h : \Lambda' \to \mathbb{Z} : \nabla_1 h(p) \in \{0,-1\}, \, \nabla_2 h(p) \in \{0,1\}\Big\}.
\end{align}
}
$\mathscr{H}$ is the set of all possible stochastic six vertex model height functions. We also define
    \begin{equation}\label{eq: H0}
        \mathscr{H}_0 := \left\{ h \in \mathscr{H} : h(m,1/2) = 0, \, h(1/2,m) = m-1/2, \, \text{for all } m \in \mathbb{Z}_{\ge 0}' \right\}
    \end{equation}
    to denote the set of stochastic six vertex model height functions obtained from step initial condition.

In our arguments, we shall typically consider two height functions $h,h'\in \mathscr{H}$ simultaneously which would often require us to consider regions where one of the heights dominates the other. We introduce these regions next and then derive some of their geometric properties.
    
    \begin{definition}
    Given two height functions $h,h' \in \mathscr{H}$ we define the region of $\Lambda'$
    \begin{equation*}
        R(h<h') = \left\{ p \in \Lambda' : h(p) < h'(p) \right\}.
    \end{equation*}
    In general the region $R(h<h')$ is the disjoint union of connected components, in the topology of $\Lambda'$ induced by the nearest neighbor metric.   For any $p\in \Lambda'$ we will denote the connected component of $R(h<h')$ containing $p$ by $R_p(h<h')$. In case $p \notin R(h<h')$, by agreement we set $R_p(h<h')=\varnothing$. Similarly, we define 
    \begin{equation*}
        R(h\neq h') = R(h<h') \cup R(h'<h)
    \end{equation*} 
    and we denote by $R_p(h \neq h')$ the connected component of $R(h\neq h')$ containing the face $p$.
\end{definition}
The following proposition states a simple property of the boundary of the regions $R(h<h')$.

\begin{proposition} \label{lem:adjacent faces}
    Fix $h,h' \in \mathscr{H}$. If $p,p' \in \Lambda'$ are adjacent faces such that $p \in R(h<h')$ and $p' \notin R(h<h')$, then $h'(p) = h(p)+1$ and $h'(p') = h(p')$.
\end{proposition}
\begin{proof}
    Consider the case $p'=p+\mathbb{e}_1$. Calling $r=h(p)$ and $r'=h'(p)$, then we have $h(p')$ is either $r$ or $r-1$ and $h'(p')$ is either $r'$ or $r'-1$ as depicted below
    \begin{equation*}
        \begin{tikzpicture}
            \draw[thick] (0,0) -- (1,0) -- (1,1) -- (0,1) -- (0,0);
            \node[] at (.5,.3) {\small \textcolor{blue}{$r$}};
            \node[] at (.5,.7) {\small \textcolor{red}{$r'$}};
            \draw[thick] (1,0) -- (2,0) -- (2,1) -- (1,1) -- (1,0);
            \node[] at (1.5,.3) {\small \textcolor{blue}{$r$}};
            \node[] at (1.5,.7) {\small \textcolor{red}{$r'$}};
            \node[below] at (.5,-.12) {$p$};
            \node[below] at (1.5,0) {$p'$};

            \draw[thick,xshift=3cm] (0,0) -- (1,0) -- (1,1) -- (0,1) -- (0,0);
            \node[xshift=3cm] at (.5,.3) {\small \textcolor{blue}{$r$}};
            \node[xshift=3cm] at (.5,.7) {\small \textcolor{red}{$r'$}};
            \draw[thick,xshift=3cm] (1,0) -- (2,0) -- (2,1) -- (1,1) -- (1,0);
            \node[xshift=3cm] at (1.5,.3) {\small \textcolor{blue}{$r-1$}};
            \node[xshift=3cm] at (1.5,.7) {\small \textcolor{red}{$r'$}};
            \node[below,xshift=3cm] at (.5,-.12) {$p$};
            \node[below,xshift=3cm] at (1.5,0) {$p'$};

            \draw[thick,xshift=6cm] (0,0) -- (1,0) -- (1,1) -- (0,1) -- (0,0);
            \node[xshift=6cm] at (.5,.3) {\small \textcolor{blue}{$r$}};
            \node[xshift=6cm] at (.5,.7) {\small \textcolor{red}{$r'$}};
            \draw[thick,xshift=6cm] (1,0) -- (2,0) -- (2,1) -- (1,1) -- (1,0);
            \node[xshift=6cm] at (1.5,.3) {\small \textcolor{blue}{$r$}};
            \node[xshift=6cm] at (1.5,.7) {\small \textcolor{red}{$r'-1$}};
            \node[below,xshift=6cm] at (.5,-.12) {$p$};
            \node[below,xshift=6cm] at (1.5,0) {$p'$};

            \draw[thick,xshift=9cm] (0,0) -- (1,0) -- (1,1) -- (0,1) -- (0,0);
            \node[xshift=9cm] at (.5,.3) {\small \textcolor{blue}{$r$}};
            \node[xshift=9cm] at (.5,.7) {\small \textcolor{red}{$r'$}};
            \draw[thick,xshift=9cm] (1,0) -- (2,0) -- (2,1) -- (1,1) -- (1,0);
            \node[xshift=9cm] at (1.5,.3) {\small \textcolor{blue}{$r-1$}};
            \node[xshift=9cm] at (1.5,.7) {\small \textcolor{red}{$r'-1$}};
            \node[below,xshift=9cm] at (.5,-.12) {$p$};
            \node[below,xshift=9cm] at (1.5,0) {$p'$};
        \end{tikzpicture}
        .
    \end{equation*}
Since $p\in R(h<h')$, we have $r<r'$. Since $p' \notin R(h<h')$, we have $h(p')\ge h'(p')$. But, among the above four cases, this is only possible when $r'=r+1$. This proves our claim in the case $p'=p+\mathbf{e}_1$. All other cases $p'=p-\mathbb{e}_1$, $p'=p+\mathbb{e}_2$, $p'=p-\mathbb{e}_2$ can be proven analogously.
\end{proof}

By virtue of \Cref{lem:adjacent faces}, for any pair of heights $h,h' \in \mathscr{H}$ the sets $R(h<h')$ and $R(h'<h)$ do not have any adjacent faces. In other words, in terms of Manhattan set distance, $d_1(R(h<h'), R(h'<h)) > 1$. A consequence of this fact is that, for any $p$ the set $R_p(h\neq h')$ is either equal to $R_p(h < h')$ or $R_p(h' < h)$.

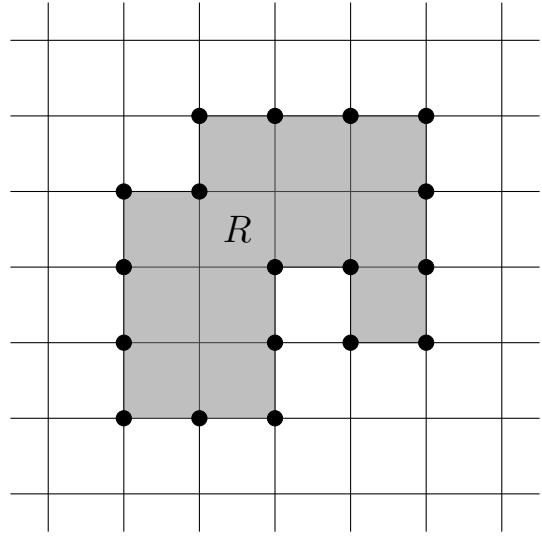
\begin{figure}
    \centering
    \begin{tikzpicture}
        \foreach \i in {1,...,7}
        {
            \draw[] (.5,\i) -- (7.5,\i);
            \draw[] (\i,.5) -- (\i,7.5);
        }
        \draw[fill=gray,opacity=.5] (2,2) -- (2,5) -- (3,5) -- (3,6) -- (6,6) -- (6,3) -- (5,3) -- (5,4) -- (4,4) -- (4,2);
        \foreach \p in {(2,2),(2,3),(2,4),(2,5),(3,5),(3,6),(4,6),(5,6),(6,6),(6,5),(6,4),(6,3),(5,3),(5,4),(4,4),(4,3),(4,2),(3,2)}
        {
            \draw[fill] \p circle(.1cm);
        }
        \node[] at (3.5,4.5) {\Large $R$};
    \end{tikzpicture}  
    \caption{A connected region $R\subset \Lambda'$. Thick dots denote the boundary set $\partial R$.}
    \label{fig:boundary set}
\end{figure}

\medskip

We next introduce the shift map $\tau_k : \mathscr{H} \to \mathscr{H}$ for any $k\in \Z$ by setting 
$$\tau_k(h) (p) := h(p)+k.$$
In words, the shift map $\tau_k$ lifts a height by $k$ units.  The following lemma shows that given any vertex $v$ on the boundary of two height functions, i.e., $\partial R(h\neq h')$, if we lift one of the heights by at least $2$ units, the vertex $v$ no longer lies in the boundary $\partial R(\tau_k(h)\neq h')$.

\begin{lemma} \label{lem:lift by 4}
    Fix $h,h' \in \mathscr{H}$. If $v \in \partial R(h \neq h')$, then, for all $k\ge 2$, $v\notin \partial R(\tau_k(h) \neq h')$.
\end{lemma}

\begin{proof} Note that if $v\in \partial R(h'<h)$, then $h(p) \ge h'(p)-1$ for all faces $p$ around $v$. But then for $k\ge 2$, $\tau_k(p)> h'(p)$ for all faces $p$ around $v$. Thus $v\notin \partial R(h'<\tau_k(h))$ for all $k\ge 2$. Suppose $v\in \partial R(h<h')$.
    Then among the faces adjacent to the vertex $v$ at least one is in $R(h<h')$, while at least one is not in $R(h<h')$. Coloring faces in $R(h<h')$ in gray and in white those not in $R(h<h')$, we can depict cases in which only one of the faces adjacent to $v$ is in $R(h<h')$ as
    \begin{equation} \label{eq:1face}
        \begin{tikzpicture}
            \draw[thick] (0,0) -- (1,0);
            \draw[thick] (.5,-.5) -- (.5,.5);
            \draw[fill=gray,opacity=.5] (0,0) -- (.5,0) -- (.5,.5) -- (0,.5) -- (0,0);

            \draw[thick,xshift=2cm] (0,0) -- (1,0);
            \draw[thick,xshift=2cm] (.5,-.5) -- (.5,.5);
            \draw[fill=gray,opacity=.5,xshift=2.5cm] (0,0) -- (.5,0) -- (.5,.5) -- (0,.5) -- (0,0);

            \draw[thick,xshift=4cm] (0,0) -- (1,0);
            \draw[thick,xshift=4cm] (.5,-.5) -- (.5,.5);
            \draw[fill=gray,opacity=.5,xshift=4cm,yshift=-.5cm] (0,0) -- (.5,0) -- (.5,.5) -- (0,.5) -- (0,0);

            \draw[thick,xshift=6cm] (0,0) -- (1,0);
            \draw[thick,xshift=6cm] (.5,-.5) -- (.5,.5);
            \draw[fill=gray,opacity=.5,xshift=6.5cm,yshift=-.5cm] (0,0) -- (.5,0) -- (.5,.5) -- (0,.5) -- (0,0);
        \end{tikzpicture}
        ,
    \end{equation}
    cases in which two of the faces adjacent to $v$ are in $R(h<h')$ as
    \begin{equation} \label{eq:2faces}
        \begin{tikzpicture}
            \draw[thick] (0,0) -- (1,0);
            \draw[thick] (.5,-.5) -- (.5,.5);
            \draw[fill=gray,opacity=.5] (0,0) -- (.5,0) -- (.5,.5) -- (0,.5) -- (0,0);
            \draw[fill=gray,opacity=.5,xshift=.5cm] (0,0) -- (.5,0) -- (.5,.5) -- (0,.5) -- (0,0);

            \draw[thick,xshift=2cm] (0,0) -- (1,0);
            \draw[thick,xshift=2cm] (.5,-.5) -- (.5,.5);
            \draw[fill=gray,opacity=.5,xshift=2cm] (0,0) -- (.5,0) -- (.5,.5) -- (0,.5) -- (0,0);
            \draw[fill=gray,opacity=.5,xshift=2cm,yshift=-.5cm] (0,0) -- (.5,0) -- (.5,.5) -- (0,.5) -- (0,0);

            \draw[thick,xshift=4cm] (0,0) -- (1,0);
            \draw[thick,xshift=4cm] (.5,-.5) -- (.5,.5);
            \draw[fill=gray,opacity=.5,xshift=4cm,yshift=-.5cm] (0,0) -- (.5,0) -- (.5,.5) -- (0,.5) -- (0,0);
            \draw[fill=gray,opacity=.5,xshift=4.5cm,yshift=-.5cm] (0,0) -- (.5,0) -- (.5,.5) -- (0,.5) -- (0,0);

            \draw[thick,xshift=6cm] (0,0) -- (1,0);
            \draw[thick,xshift=6cm] (.5,-.5) -- (.5,.5);
            \draw[fill=gray,opacity=.5,xshift=6.5cm,yshift=-.5cm] (0,0) -- (.5,0) -- (.5,.5) -- (0,.5) -- (0,0);
            \draw[fill=gray,opacity=.5,xshift=6.5cm] (0,0) -- (.5,0) -- (.5,.5) -- (0,.5) -- (0,0);

            \draw[thick,xshift=8cm] (0,0) -- (1,0);
            \draw[thick,xshift=8cm] (.5,-.5) -- (.5,.5);
            \draw[fill=gray,opacity=.5,xshift=8.5cm,yshift=-.5cm] (0,0) -- (.5,0) -- (.5,.5) -- (0,.5) -- (0,0);
            \draw[fill=gray,opacity=.5,xshift=8cm] (0,0) -- (.5,0) -- (.5,.5) -- (0,.5) -- (0,0);

            \draw[thick,xshift=10cm] (0,0) -- (1,0);
            \draw[thick,xshift=10cm] (.5,-.5) -- (.5,.5);
            \draw[fill=gray,opacity=.5,xshift=10cm,yshift=-.5cm] (0,0) -- (.5,0) -- (.5,.5) -- (0,.5) -- (0,0);
            \draw[fill=gray,opacity=.5,xshift=10.5cm] (0,0) -- (.5,0) -- (.5,.5) -- (0,.5) -- (0,0);
        \end{tikzpicture}
    \end{equation}
    and finally cases in which three of the faces adjacent to $v$ are in $R(h<h')$ as
    \begin{equation}\label{eq:3faces}
        \begin{tikzpicture}
            \draw[thick] (0,0) -- (1,0);
            \draw[thick] (.5,-.5) -- (.5,.5);
            \draw[fill=gray,opacity=.5] (0,0) -- (.5,0) -- (.5,.5) -- (0,.5) -- (0,0);
            \draw[fill=gray,opacity=.5,xshift=.5cm] (0,0) -- (.5,0) -- (.5,.5) -- (0,.5) -- (0,0);
            \draw[fill=gray,opacity=.5,yshift=-.5cm] (0,0) -- (.5,0) -- (.5,.5) -- (0,.5) -- (0,0);

            \draw[thick,xshift=2cm] (0,0) -- (1,0);
            \draw[thick,xshift=2cm] (.5,-.5) -- (.5,.5);
            \draw[fill=gray,opacity=.5,xshift=2.5cm] (0,0) -- (.5,0) -- (.5,.5) -- (0,.5) -- (0,0);
            \draw[fill=gray,opacity=.5,xshift=2.5cm,yshift=-.5cm] (0,0) -- (.5,0) -- (.5,.5) -- (0,.5) -- (0,0);
            \draw[fill=gray,opacity=.5,xshift=2cm] (0,0) -- (.5,0) -- (.5,.5) -- (0,.5) -- (0,0);

            \draw[thick,xshift=4cm] (0,0) -- (1,0);
            \draw[thick,xshift=4cm] (.5,-.5) -- (.5,.5);
            \draw[fill=gray,opacity=.5,xshift=4cm,yshift=-.5cm] (0,0) -- (.5,0) -- (.5,.5) -- (0,.5) -- (0,0);
            \draw[fill=gray,opacity=.5,xshift=4.5cm,yshift=-.5cm] (0,0) -- (.5,0) -- (.5,.5) -- (0,.5) -- (0,0);
            \draw[fill=gray,opacity=.5,xshift=4.5cm] (0,0) -- (.5,0) -- (.5,.5) -- (0,.5) -- (0,0);

            \draw[thick,xshift=6cm] (0,0) -- (1,0);
            \draw[thick,xshift=6cm] (.5,-.5) -- (.5,.5);
            \draw[fill=gray,opacity=.5,xshift=6.5cm,yshift=-.5cm] (0,0) -- (.5,0) -- (.5,.5) -- (0,.5) -- (0,0);
            \draw[fill=gray,opacity=.5,xshift=6cm,yshift=-.5cm] (0,0) -- (.5,0) -- (.5,.5) -- (0,.5) -- (0,0);
            \draw[fill=gray,opacity=.5,xshift=6cm] (0,0) -- (.5,0) -- (.5,.5) -- (0,.5) -- (0,0);
        \end{tikzpicture}
        .
    \end{equation}
    By \Cref{lem:adjacent faces} we know that, white faces $p'$ adjacent to gray faces are such that $h(p')=h'(p')$, while gray faces $p$ adjacent to white faces are such that $h'(p)=h(p)+1$. This shows that in all cases \eqref{eq:1face}, \eqref{eq:2faces} none of the faces adjacent to $v$ will belong to $R(\tau_1(h)<h')$. On the other hand if faces configurations fall into case \eqref{eq:3faces}, then, after lifting up by 1 height $h$, around vertex $v$ there could still be a single face of $R(\tau_1(h)<h')$, as in \eqref{eq:1face}. Then, lifting further $h$, we have that $R(\tau_2(h)<h')$ cannot intersect any of the vertices adjacent to $v$. Thus, $v\notin \partial R(\tau_k(h)<h')$ for all $k\ge 2$. This completes the proof.
\end{proof}

As a consequence of the above lemma, we observe that $\{\partial R(\tau_{2k}(h)\neq h')\}_{k\ge 0}$ is a disjoint family of sets. By a pigeonhole principle argument, this will allow us to obtain $k^*$ where $|\partial R(\tau_{2k^*}(h)\neq h') \cap ([0,M]\times [0,N])|$ is of smaller order than $MN$. We record a version of this consequence in the following lemma. 

\begin{lemma} \label{lem:bound k}
    Fix any $p\in \Lambda'$ and $M,N \ge 8$ and $r < r'$. For any pair of height functions $h,h' \in \mathscr{H}_0^{}$ with $h(M,N)=r$, $h'(M,N)=r'$, define
    \begin{equation} \label{eq:def k}
        k^*_{M,N}(p;h,h') = \min \left\{ k\ge \left\lceil \frac{r'-r}{2} \right\rceil \,: \, \left| \partial R_{p}(\tau_k(h)\neq h') \cap \Lambda_{M,N} \right| \le (MN)^{4/5} \right\},
    \end{equation}
    where $\Lambda_{M,N} := \left\{ 1,\dots, M \right\} \times \left\{ 1,\dots, N  \right\}$.
    Then, we have
    \begin{equation} \label{eq:bound k}
        \left\lceil \frac{r'-r}{2} \right\rceil \le k^*_{M,N}(p;h,h') \le \left\lceil \frac{r'-r}{2} \right\rceil + (MN)^{2/5}.
    \end{equation}
\end{lemma}
\begin{proof}
    For any $p\in \Lambda'$, pairs of height functions $h,h'$, and for $i\in \{0,1\}$, the sets $\left\{ \partial R_p(\tau_{2k+i}(h) \neq h') \right\}_{k \in \mathbb{Z}}$
    form a disjoint family, by \Cref{lem:lift by 4}. Thus,
   \begin{equation*}
       \sum_{k=\left\lceil \frac{r'-r}{2} \right\rceil}^{\left\lceil \frac{r'-r}{2} \right\rceil + (MN)^{2/5}} \left| \partial R_p(\tau_{k}(h)\neq h') \cap \Lambda_{M,N} \right| \le 2 |\Lambda_{M,N}| \le 2MN.
   \end{equation*}
   Then, whenever $(MN)^{\frac{4}{5}+\frac{2}{5}-1}=(MN)^{\frac{1}{5}}>2$, which is guaranteed by our choice of $M,N$, we can be sure that not all of the $(MN)^{2/5}$ terms in the above summation can be greater than $(MN)^{4/5}$, by a straightforward counting argument. This proves the bound \eqref{eq:bound k}.
\end{proof}

\begin{definition}[Involution of surfaces] \label{iotamap}
    For $p\in \Lambda'$ and two height functions $h,h' \in \mathscr{H}$, we define 
    \begin{equation*}
        \iota_p(h,h'):= (\overline{h},\overline{h}') \in \mathscr{H} \times \mathscr{H},
    \end{equation*}
    where
    \begin{equation} \label{eq:def ip}
        (\overline{h}(z),\overline{h}'(z)) := 
        \begin{cases}
            (h'(z),h(z)) \qquad & \text{if } z \in R_p(h \neq h')
            \\
            (h(z),h'(z)) \qquad & \text{else}.
        \end{cases}
    \end{equation}
    If $h(p)=h'(p)$ we set $\iota_p (h,h'):=(h,h')$.
\end{definition}
Note that $(\overline{h},\overline{h}')$, defined via \eqref{eq:def ip}, indeed lies in $\mathscr{H}\times \mathscr{H}$ as on the faces that are not in $R_p(h\neq h')$ but are adjacent to some face in $R_p(h \neq h')$, $h$ and $h'$ are equal. In words, the $\iota_p$ swaps the height function $h$ and $h'$ in the region $R_p(h\neq h')$. Clearly, it is an involution.

\begin{definition}[Bijective contraction maps] \label{upsilonmap}
    For $p\in \Lambda'$, $k \in \mathbb{Z}$ and two height functions $h,h' \in \mathscr{H}$, we define the map $\Upsilon_p^k:\mathscr{H} \times \mathscr{H} \to \mathscr{H} \times \mathscr{H}$ as
    \begin{equation*}
        \Upsilon_p^k := (\tau_{-k} \times \mathrm{Id}) \circ \iota_p \circ (\tau_{k} \times \mathrm{Id}),
    \end{equation*}
    where $(\tau_k \times \mathrm{Id})(h,h') = ( \tau_k(h),h' )$.
\end{definition}

As $\iota_p$ is an involution, it is clear that for every $k\in \mathbb{Z}$ and $p \in \mathbb{Z}$ the map $\Upsilon_p^k$ is a bijection on the space of pairs of height functions with inverse given by $\Upsilon_p^{-k}$.  A closer look to the map $\Upsilon_p^k$ shows that, calling $(\widetilde{h},\widetilde{h}') = \Upsilon_p^k(h,h')$, we have
    \begin{equation} \label{eq:shrinkage at p}
         \widetilde{h}(p)=h'(p)-k, \widetilde{h}'(p)= h(p)+k \mbox{ whenever } h'(p)\ge h(p)+k.
    \end{equation}
    Thus whenever $h'(p)$ and $h(p)$ are far apart, the $\Upsilon_p^k$ map brings the heights at $p$ closer. For this reason, we call it a \textit{contraction} map. 
A further property of the map $\Upsilon_p^k$ stating that it preserves the step initial condition, which will be instrumental to establish convexity properties of the law of the six vertex model height function, is stated next.

\begin{proposition} \label{prop:contraction property}
    Fix $h,h' \in \mathscr{H}_0$ (recall the definition from \eqref{eq: H0}) and $p\in \Lambda'$. Suppose $h(p)=r$, $h'(p)=r'$ with $r'>r$. Then for all $k \in \{ 0, \dots, r'-r \}$,  $\Upsilon_p^k(h,h') \in \mathscr{H}_0^{}$.
\end{proposition}

\begin{proof}
    Consider the left and bottom boundary of the lattice $\Lambda'$
    \begin{equation*}
        B_{\llcorner} = \left\{ (i',j') \in \Lambda' : \min(i',j')=1/2 \right\}.
    \end{equation*}
    In case $k \in \{ 0, \dots, r'-r \}$ we have $h(p)+k = r+k \le r' = h'(p).$
    On the other hand for each boundary face $v \in B_{\llcorner}$ we have $h(v) + k \ge h'(v)$.  Therefore $R_p(\tau_k(h) \neq h')$ has empty intersection with the boundary $B_{\llcorner}$ and calling $(\overline{h}, \overline{h}') = \iota_p (\tau_k(h),h')$ we have
    \begin{equation*}
        \overline{h}(p) = h(p)+k, \qquad \overline{h}'(p) = h'(p) \qquad \text{for all } p \in B_{\llcorner}
    \end{equation*}
    and as a result for $(\widetilde{h},\widetilde{h}')=(\tau_{-k}(\overline{h}), \overline{h}')$ we have
    \begin{equation*}
        \widetilde{h}(p) = h(p) \qquad \widetilde{h}'(p) = h'(p) \qquad \text{for all } p \in B_{\llcorner}.
    \end{equation*}
    Therefore boundary conditions of $h,h',\widetilde{h},\widetilde{h}'$ are the same, completing the proof.
\end{proof}

In words, \Cref{prop:contraction property} allows us to transform a pair $(h,h')$ of height functions in $\mathscr{H}_0^{}$ which assume values far apart at a fixed point $p$ into a pair $(\widetilde{h},\widetilde{h}')$, where both heights still lie in $\mathscr{H}_0^{}$, but their value at $p$ is closer by $2k$ as prescribed by \eqref{eq:shrinkage at p}.

\subsection{Log-concavity bounds}\label{sec3.2} In this section, we prove weak log-concavity bounds for the tail probabilities of the stochastic six vertex model. Towards this end, we first view the S6V height function $\h$ as a function on the faces by setting $$\h(p):=\h(p-\mathbb{e}_1/2-\mathbb{e}_2/2)$$ for all $v\in \Lambda'$ where the right-hand side is defined via \eqref{eq:height}. The Markov evolution rules for vertex configurations defining the S6V model can clearly be phrased in terms of height functions, as shown by \Cref{fig:vertex weights}. Recall the notation of adjacent faces from \eqref{dir}. Given a height function $h \in \mathscr{H}$ and any vertex $v \in \Lambda$ we introduce the notation
    \begin{equation*}
        w_{v}(h) = w \left( -\nabla_1 h (\mathsf{SW}(v) ), \nabla_2 h(\mathsf{SW}(v)) ; -\nabla_1 h(\mathsf{NW}(v)), \nabla_2 h(\mathsf{SE}(v))  \right),
    \end{equation*}
where  $w(i_1,j_1;i_2,j_2)$ are vertex weights from \Cref{fig:vertex weights}. $w_v(h)$ is precisely the stochastic weight of the vertex $v$  corresponding to a given height function $h$. With a slight abuse of notation, for any finite set of vertices $V \subset \Lambda$ we define
    \begin{equation*}
        w_{V}(h) = \prod_{v \in V} w_{v}(h).
    \end{equation*}
We shall call $w_{V}(h)$ the Boltzmann weight of the height $h$ restricted to $V$. With this notation in place, it is easy to see that for all $h \in \mathscr{H}_0^{}$ we have
\begin{align}\label{probs6v}
    \Pr\big(\h(v)=h(v) \mbox{ for all } v\in \Lambda_{M,N}\big)=w_{\Lambda_{M,N}}(h).
\end{align}

\medskip

Our next two results shows certain weak log-concavity of the Boltzmann weight under the $\iota_p$ and $\Upsilon_p^k$ map defined in Definition \ref{iotamap} and \ref{upsilonmap} respectively.
Towards this end, we define the constant
\begin{equation} \label{eq:constant c}
    \mathsf{C} := a^{-2}q^{-2}(1-a)^{-2}(1-q)^{-2}(1-aq)^{4},
\end{equation}
where $a,q$ are the parameters of the S6V model. This constant has the property that
\begin{equation} \label{cineq}
    \frac{w_{v_1}(h_1)w_{v_2}(h_2)}{w_{v_3}(h_3)w_{v_4}(h_4)} \le \frac{1}{w_{v_3}(h_3)w_{v_4}(h_4)}  \le \mathsf{C}
\end{equation}
for all $v_i\in \Lambda'$ and $h_i\in \mathscr{H}$ which can be readily checked.

\begin{proposition}[Weak log-concavity under $\iota_p$ map] \label{prop:inequality involution}
    Fix $M,N\in \mathbb{Z}_{\ge 1}$ and $p\in \Lambda_{M,N}$. Consider $h,h' \in \mathscr{H}$ and let $ (\overline{h},\overline{h}') = \iota_p(h,h')$. Then
    \begin{equation*}
        w_{\Lambda_{M,N}}(h) w_{\Lambda_{M,N}} (h') \le  \mathsf{C}^{\#\left| \partial R_p(h \neq h') \cap \Lambda_{M,N} \right|} w_{\Lambda_{M,N}} (\overline{h}) w_{\Lambda_{M,N}} (\overline{h}').
    \end{equation*}
\end{proposition}
\begin{proof}
     Let $v\in \Lambda$ be a vertex and let $\mathsf{SW}(v),\mathsf{SE}(v),\mathsf{NW}(v),\mathsf{NE}(v) \in \Lambda'$ be the quadruple of faces adjacent to $v$. To lighten the notation, we write $R_p$ to mean $R_p(h\neq h')$. If $\mathsf{SW}(v),\mathsf{SE}(v),\mathsf{NW}(v),\mathsf{NE}(v)$ all belong to $R_p^{\mathrm{c}}$ we have, by construction
    \begin{equation*}
        h(p) = \overline{h}(p), \qquad h'(p) = \overline{h}'(p),
        \qquad \text{for all } p\in \{ \mathsf{SW}(v),\mathsf{SE}(v),\mathsf{NW}(v),\mathsf{NE}(v) \},
    \end{equation*}
    which implies that $w_v(h) = w_v(\overline{h})$ and $w_v(h') = w_v(\overline{h}').$
    Similarly, if $\mathsf{SW}(v),\mathsf{SE}(v),\mathsf{NW}(v),\mathsf{NE}(v)$ all belong to $R_p$ we have, by construction
    \begin{equation*}
        h(p) = \overline{h}'(p), \qquad h'(p) = \overline{h}(p),
        \qquad \text{for all } p\in \{ \mathsf{SW}(v),\mathsf{SE}(v),\mathsf{NW}(v),\mathsf{NE}(v) \},
    \end{equation*}
    which implies that $w_v(h) = w_v(\overline{h}')$ and $w_v(h') = w_v(\overline{h})$.
    This shows that
    \begin{equation*}
        w_v(h) w_v(h') = w_v(\overline{h}) w_v(\overline{h}'), \qquad \text{if } v \notin \partial R_p
    \end{equation*}
    and hence
    \begin{equation} \label{eq:weights not on Rp}
        w_{\Lambda_{M,N} \setminus \partial R_p }(h) w_{\Lambda_{M,N} \setminus \partial R_p }(h') = w_{\Lambda_{M,N} \setminus \partial R_p }(\overline{h}) w_{\Lambda_{M,N} \setminus \partial R_p }(\overline{h}').
    \end{equation}
    On the other hand, when $v\in \partial R_p(h\neq h')$, in general $w_v(h) w_v(h') \neq w_v(\overline{h}) w_v(\overline{h}')$. However, due to the property \eqref{cineq} we have
    \begin{equation*}
        \frac{ w_v(h) w_v(h') }{ w_v(\overline{h}) w_v(\overline{h}') }  \le \mathsf{C}.
    \end{equation*}
    As a result we have 
    \begin{equation} \label{eq:weights on Rp}
    \begin{split}
        &w_{\partial R_p \cap \Lambda_{M,N} }(h) w_{\partial R_p \cap \Lambda_{M,N} }(h') 
         \le \mathsf{C}^{\#\left| \partial R_p\cap \Lambda_{M,N} \right|} w_{\partial R_p\cap \Lambda_{M,N} }(\overline{h}) w_{\partial R_p \cap \Lambda_{M,N} }(\overline{h}').
    \end{split}
    \end{equation}
    Combining \eqref{eq:weights not on Rp} and \eqref{eq:weights on Rp} completes the proof.
\end{proof}

We can easily upgrade the above proposition to $\Upsilon_p^k$ maps defined in Definition \ref{upsilonmap}.

\begin{proposition} \label{prop:bound weights contraction}
    Fix $M,N\in \mathbb{Z}_{\ge 1}$ and $p\in \Lambda_{M,N}$. Consider $h,h' \in \mathscr{H}$ and let $ (\widetilde{h},\widetilde{h}') = \Upsilon_p^k (h,h')$. Then
    \begin{equation*}
        w_{\Lambda_{M,N}}(h) w_{\Lambda_{M,N}} (h') \le  \mathsf{C}^{\#\left| \partial R_p(\tau_k(h)\neq h') \cap \Lambda_{M,N} \right|} w_{\Lambda_{M,N}} (\widetilde{h}) w_{\Lambda_{M,N}} (\widetilde{h}'),
    \end{equation*}
    where the constant $\mathsf{C}$ is defined in \eqref{eq:constant c}.
\end{proposition}
\begin{proof}
    This is an immediate consequence of \Cref{prop:inequality involution} after noticing that for every height $h \in \mathscr{H}$ and any finite set $A$, for every $k\in \mathbb{Z}$, we have $w_A(h) = w_A(\tau_k(h))$.
\end{proof}

Since the above proposition holds for all $k$, using Lemma \ref{lem:bound k}, we may choose a $k$ for which $\#\left| \partial R_p(\tau_k(h)\neq h') \cap \Lambda_{M,N} \right|$ is of smaller order than $MN$. This will allow us to deduce a density version of \Cref{t.main0}.

\begin{proposition} \label{prop: weak log concavity}
    Fix $M,N \ge 8$ and $r,r'\in \Z_{\ge 0}$. Then, we have
    \begin{equation} \label{eq: weak log concavity}
        \begin{split}
            & \mathbb{P}(\h(M,N)=r) \mathbb{P}(\h(M,N)=r')
            \le \mathsf{C}^{(MN)^{4/5}} \mathbb{P} \left(  \left| \h(M,N) - \frac{r+r'}{2} \right| \le 4(MN)^{2/5} \bigg| \right)^2,
        \end{split}
    \end{equation}
   where the coefficient $\mathsf{C}$ is given in \eqref{eq:constant c}. 
\end{proposition}

\begin{proof}[Proof of \Cref{prop: weak log concavity}] Without loss of generality, assume $r<r'$.
    When $\frac{r'-r}{2} \le 4(MN)^{2/5}$, 
    the event
    \begin{equation*}
        \left\{\left| \h(M,N) - \frac{r+r'}{2} \right| \le 4(MN)^{2/5} \right\},
    \end{equation*}
     contains both the events $\{\h(M,N)=r \}$ and $\{\h(M,N)=r' \}$. Hence \eqref{eq: weak log concavity} is trivial. So, let us assume $\frac{r'-r}{2} >4(MN)^{2/5}$. Set $p=(M,N)\in \Lambda$ and consider the set
     $$T_{p;r,r'}:=\big\{ (h,h')\in \mathscr{H}_0^{}\times \mathscr{H}_0^{} : h(p)=r, h'(p)=r'\big\}.$$
     Using \eqref{probs6v}, we may write the left-hand side of \eqref{eq: weak log concavity} as follows.
     \begin{align}\label{eq:ltr}
         \mathbb{P}(\h(M,N)=r) \mathbb{P}(\h(M,N)=r')=\sum_{(h,h')\in T_{p;r,r'} } w_{\Lambda_{M,N}}(h) w_{\Lambda_{M,N}} (h'). 
     \end{align}
   Consider the $\Upsilon_{p}^{k^*}$ map restricted to $T_{p;r,r'}$ with $k^*=k_{M,N}^*(p;h,h')$ coming from \eqref{eq:def k}. Since $\frac{r'-r}{2} >4(MN)^{2/5}$, by the bound on $k^*$ from \eqref{eq:bound k}, we have that $k^* \le r'-r=h'(p)-h(p)$ for all $(h,h')\in T_{p;r,r'}$. Hence the condition in \eqref{eq:shrinkage at p} holds. By \eqref{eq:shrinkage at p} and the bound on $k^*$ from \eqref{eq:bound k} we have that
    \begin{equation}\label{eqbdf}
        \frac{r+r'}{2} - 4(MN)^{2/5} \le \widetilde{h}(p) , \widetilde{h}'(p) \le \frac{r+r'}{2} + 4(MN)^{2/5}
    \end{equation}
    for all $(\widetilde{h},\widetilde{h}') \in \Upsilon_{p}^{k^*}(T_{p;r,r'})$. Moreover, by definition of $k^*=k_{M,N}^*(p;h,h')$ from \eqref{eq:def k} we have
    \begin{equation*}
        \#\left| \partial R_{p}(\tau_{k^*}(h) \neq h') \cap \Lambda_{M,N} \right| \le (MN)^{4/5},
    \end{equation*}
    which shows, using \Cref{prop:bound weights contraction}, that
    \begin{equation} \label{eq:log concavity h htilde}
        w_{\Lambda_{M,N}}(h) w_{\Lambda_{M,N}} (h') \le  \mathsf{C}^{(MN)^{4/5}} w_{\Lambda_{M,N}} (\widetilde{h}') w_{\Lambda_{M,N}} (\widetilde{h}').
    \end{equation}
We claim that $\Upsilon_{p}^{k^*}$ is injective when restricted to $T_{p;r,r'}$ (note that $k^*$ depends on $(h,h')$).  Indeed, if $(\widetilde{h},\widetilde{h}')=\Upsilon_p^{k_1^*}(h_1,h_1')=\Upsilon_p^{k_2^*}(h_2,h_2')$, we have $\widetilde{h}(p)=h_1'(p)-k_1^*=h_2'(p)-k_2^*$. But on $T_{p;r,r'}$, $h_1'(p)=h_2'(p)=r'$. Thus $k_1^*=k_2^*$. Now injectivity of $\Upsilon_p^k$ for fixed $k$ leads to $(h_1,h_1')=(h_2,h_2')$.  Thus,
    \begin{equation} \label{eq:weak log concavity expansion}
    \begin{split}
       \mbox{r.h.s.~of \eqref{eq:ltr}} & \le 
        \mathsf{C}^{(NM)^{4/5}} \sum_{ \substack{ (h,h')\in T_{p;r,r'} \\ (\widetilde{h},\widetilde{h}') =\Upsilon_p^{k^*}(h,h')} }  w_{\Lambda_{M,N}}(\widetilde{h}) w_{\Lambda_{M,N}} (\widetilde{h}')  \\ & \le 
\mathsf{C}^{(NM)^{4/5}}\sum_{\substack{\widetilde{h} \in \mathscr{H}_0^{} \\ |\widetilde{h}(p)-\frac{r+r'}2|\le 4(MN)^{2/5}}} \,\, \sum_{\substack{\widetilde{h}'\in \mathscr{H}_0^{} \\ |\widetilde{h}'(p)-\frac{r+r'}2|\le 4(MN)^{2/5}}}  w_{\Lambda_{M,N}}(\widetilde{h}) w_{\Lambda_{M,N}} (\widetilde{h}')
    \end{split}
    \end{equation}
    where in the first inequality we used \eqref{eq:log concavity h htilde} and in the second inequality we used \eqref{eqbdf} and injectivity of $\Upsilon_p^{k^*}$. This completes the proof after recognizing that the right hand side of \eqref{eq:weak log concavity expansion} equals the right hand side of \eqref{eq: weak log concavity}.
\end{proof}

  \Cref{t.main0} now follows as a  corollary of \Cref{prop: weak log concavity}. 
\begin{proof}[Proof of \Cref{t.main0}]   Fix any $M,N \ge 8$ and $r_1, r_2\ge0$. We have $\Pr(\h(M,N)\ge r_i)=\Pr(\h(M,N)\ge \lceil r_i \rceil)$. Using the fact that $\h(M,N)$ can be at most $N$, we have the following chain of inequality
    \begin{equation*}
    \begin{split}
        \prod_{i=1}^2\Pr\big(\h(M,N) \ge  \lceil r_i \rceil \big) &= 
        \sum_{k_1 = \lceil r_1 \rceil}^N \sum_{k_2 = \lceil r_2 \rceil}^N \prod_{i=1}^2\Pr\big(\h(\lfloor \kappa N \rfloor,N) =  k_i \big)
        \\
        & \le \mathsf{C}^{(MN)^{4/5}} \sum_{k_1 = \lceil r_1 \rceil}^N \sum_{k_2 = \lceil r_2 \rceil}^N  \mathbb{P} \left(   \h(M,N)  \ge  \frac{k_1+k_2}{2} - 4(MN)^{2/5} \right)^2
        \\
        & \le N^2\mathsf{C}^{(MN)^{4/5}} \mathbb{P} \left(   \h(M,N)  \ge  \frac{r_1+r_2}{2} - 4(MN)^{2/5} \right)^2.
    \end{split}
    \end{equation*}
   The first inequality is due to \Cref{prop: weak log concavity} and in the second inequality we simply used the fact that $m\mapsto \mathbb{P} \left(   \h(\lfloor \kappa N \rfloor,N)  \ge m \right)$ is a  decreasing function and hence all terms of the summation over $k_1,k_2$ can be bounded with the first term $k_1=\lceil r_1 \rceil, k_2=\lceil r_2 \rceil$.
\end{proof}

\section{Large Deviation Principle for multiplicative expectation of the Schur measure}\label{sec: LDP_schur}

In this section, we prove the lower-tail LDP for the shifted height function and discuss some properties of the rate function $\mathcal{F}_\kappa(s)$ defined in \eqref{def:fs}. In \Cref{sec.2.1}, we show the existence of lower-tail rate function using the moment matching formula from \cite{borodin2016stochastic_MM}. In \Cref{sec.2.2}, we establish the parabolic behavior of $\mathcal{F}_\kappa(s)$ for large enough $s$. As explained in the introduction, our argument in this section heavily relies on potential theory tools. We shall introduce them in the text as needed.

\subsection{Existence of lower tail rate function $\mathcal{F}_\kappa(s)$} \label{sec.2.1} We introduce a few notations to explain the moment matching formula from \cite{borodin2016stochastic_MM}. 

A partition $\lambda$ with at most $n$ non-zero elements is a decreasing sequence of non-negative integers $\lambda_1 \ge \lambda_2 \ge \cdots \ge\lambda_n \ge 0$ and $\lambda_{n+1}=\lambda_{n+2}=\cdots=0$. We define the size of such a partition to be sum of its elements $|\lambda|:=\lambda_1+\lambda_2+\cdots+\lambda_n$ and denote the set of all such partitions by $\mathbb{Y}^n$. For later purposes, we also introduce the set
\begin{align}
    \label{defwn}
    \mathbb{W}^n := \left\{\ell =(\ell_1,\ell_2,\ldots,\ell_n) \mid \exists \ \lambda\in \mathbb{Y}^n, \mbox{ such that } \ell_i=\lambda_i+n-i \mbox{ for all } i\in \{1,2,\ldots,n\}\right\}.
\end{align}
Given $a\in (0,1)$, $z,z'\in \Z_{\ge 1}$, $z$-measures are probability measures on $\mathbb{Y}^{\min(z,z')}$ defined through
\begin{equation}\label{eq:schurmeasure}
    \mathfrak{M}(a;z,z')(\lambda):= (1-a)^{zz'}a^{|\lambda|}s_\lambda(1^z)s_\lambda(1^{z'}),
\end{equation}
where $s_{\lambda}(x_1,x_2,\ldots,x_n)$ are Schur polynomials (c.f.~\cite{Macdonald1995}) and in particular we have 
\begin{equation*}
    s_\lambda(1^{K})=s_{\lambda}(1,1,\ldots,1) = \prod_{1\leq i<j\leq K}\frac{\lambda_i-i-\lambda_j+j}{j-i}.
\end{equation*}
$z$-measures were first used in the context of harmonic
analysis on the infinite symmetric group in \cite{Kerov1993}. They are a special case of Schur measures introduced in \cite{okounkov2001infinite} and  are related to Meixner ensembles (see \cite{johansson2000shape} for more details).  

It was observed by Borodin \cite{borodin2016stochastic_MM,BO2016_ASEP} that the $q$-Laplace transform of the height function of S6V coincides with the expectation of a certain multiplicative functional of the $z$-measures. We recall this result as follows.

\begin{proposition}[Proposition 8.4 in \cite{BO2016_ASEP}]
    Let $M,N\geq 1$ and recall the height function $\h(M,N)$ from \eqref{eq:height}. Then for any $\zeta>0$ we have 
    \begin{equation}\label{e.mommatch}
        \mathbb{E} \left[\prod_{i=0}^\infty\frac{1}{1+\zeta q^{\mathfrak{h}(M,N)+i}}\right]=\left[\prod_{j=0}^{\infty} \frac1{1+\zeta q^j}\right]\mathbb{E}_{\gint}\left[\prod_{j=0}^{N-1} (1+\zeta q^{\lambda_{N-j}+j})\right].
    \end{equation}
\end{proposition}

We shall see in \Cref{sec.4} how the left-hand side can be viewed as a lower tail of the height function after a random shift.
The above formula is our starting point for the asymptotics of $q$-Laplace transform of the height function:
 \begin{theorem}\label{p.multldp}
    For each $s\in \R$ and $\kappa>0$ we have 
    \begin{equation}\label{e.multldp}
        -\lim_{N\to \infty} \frac{1}{N^2}\log \mathbb{E}\left[\prod_{i=0}^\infty\frac{1}{1+ q^{\mathfrak{h}(\lfloor\kappa N\rfloor,N)+i-sN}}\right] = \mathcal{F}_\kappa(s),
    \end{equation}
    where $\mathcal{F}_\kappa$ is defined in \eqref{def:fs}.
\end{theorem}

Let us consider the case $\kappa\ge 1$. The $\kappa<1$ case can be treated analogously. For notational convenience we shall assume $\kappa N$ is an integer and work with $\kappa N+1$ instead of $\lfloor \kappa N \rfloor$. Before going into the proof of \Cref{p.multldp}, it is instructive to understand why the rate function has relation to potential theory (recall the notions and basic facts from \Cref{sec.rate}). It is known from the work of \cite{johansson2000shape} that the $z$-measures can be viewed as certain discrete log-gases and has remarkable connection to potential theory. To see this, let us introduce the shifted variables $\ell_j:=\lambda_j+N-j$ for $1\leq j\leq N$, the measure in \eqref{eq:schurmeasure} with $z=N, z'=\kappa N$ can be rewritten as 
\begin{equation}\label{zmeas}
\begin{aligned}
   (1-a)^{\kappa N^2}a^{-N(N-1)/2} \prod_{i=1}^{N}\left(a^{\ell_i}\prod_{j=N+1}^{\kappa N}\frac{\ell_i+j-N}{j-i}\right)\cdot \prod_{1\leq i<j\leq N}\frac{(\ell_i-\ell_j)^2}{(j-i)^2}.
    \end{aligned}
\end{equation}
The above measure can be realized as $\exp(-N^2 I_{\widetilde{V}_N}(\mu_N(\ell))$ for some explicit external field $\widetilde{V}_N$. Here $\mu_N$ is the empirical measure given by
\begin{align*}
    \mu_N(\ell):=N^{-1}\sum_{i=1}^{N} \delta(\ell_i/N), \qquad \ell_i:=\lambda_i+N-i,
\end{align*}
and $I_V(\cdot)$ is defined in \eqref{e.IVmu}.
Invoking potential theory results, one can then obtain large deviation estimates for the empirical measures. Based on this large deviation estimate, the analysis of the multiplicative functional of the Schur measure can essentially be done using a Varadhan's lemma type argument, which eventually leads to \Cref{p.multldp}.

Before going into the details, we recall few technical estimates from \cite{das2022large}. 
\begin{lemma}\label{l.dd22} Suppose $V_N: [0,\infty)\to \R$ be a sequence of continuously differentiable functions satisfying $|V_N'(x)|\le C_1(1+\min(x^{-1},\log N))$ for some $C_1>0$. Suppose further that there exists a function $\mathpzc{V}$ and $C_2>0$ such that
\begin{align} \label{e.vclose}
    0 \le \mathpzc{V}(x)-V_N(x) \le \frac{C_2}{N}+\frac{1}{N}\log\frac{x+\kappa-1}{x},
\end{align}
for some $\kappa \ge 1$. We have the following
    \begin{enumerate}[label=(\alph*),leftmargin=18pt]
    \item \label{dda} There exists $K_1$ depending on $C_1$ such that for each $\ell\in \mathbb{W}^N$, we have \begin{align}
        \label{aineq}
        |I_{V_N}(\mu_N(\ell))-I_{V_N}(\widetilde{\mu}_N(\ell))| \le K_1 \cdot N^{-1}\log N,
    \end{align} where $\widetilde{\mu}_N(\ell)$ is a measure on $(\R_+,\mathcal{B}(\R_+))$ with density
    \begin{align*}
        \psi_{N;\ell}(x):=\sum_{i=1}^N \ind_{[\ell_i/N,\ell_i/N+1/N]}.
    \end{align*}
        \item \label{ddb} Recall the set $\mathcal{A}_{\infty}$ and the notation $\mu_\phi$ from the discussion around \eqref{cala}. There exists $K_2$ depending on $C_2, \kappa$ such that for any $\phi\in \mathcal{A}_\infty$, we have $|I_{V_N}(\mu_\phi)- I_{\mathpzc{V}}(\mu_\phi)| \le K_2\cdot N^{-1}$.
        \item \label{ddc} Let $\phi_{\mathpzc{V}}$ be the equilibrium measure corresponding to potential $\mathpzc{V}$ (recall \eqref{l.ds97}). $\phi_{\mathpzc{V}}$ has a compact support. Let $b$ be the right end point of the support of $\phi_{\mathpzc{V}}$. Set $K_3:=\sup_{x\in [0,b]} \mathpzc{V}(x)$.
There exists $\ell_*\in \mathbb{W}^N$ and $K_4$ depending on $C_1$ and $b$ such that
        \begin{align*}
            I_{V_N}(\mu_N(\ell_*)) \le I_{\mathpzc{V}}(\mu_{\phi_{\mathpzc{V}}})+K_3N^{-1}+K_4 \cdot N^{-1}\log N,
        \end{align*}
    \end{enumerate}
\end{lemma}
The proof of the above lemma follows from minor modification of results in \cite{das2022large} and so we defer its proof to \Cref{sec.appb}.

\begin{proof}[Proof of \Cref{p.multldp}]  Fix $s>0$. 
    Taking $\zeta=q^{-sN}$ and $M=\kappa N+1$ in \eqref{e.mommatch} we get
    \begin{equation*}
        \mathbb{E}\left[\prod_{i\geq 0}\frac{1}{1+q^{\mathfrak{h}(M,N)+i-sN}}\right]=\left[\prod_{j=0}^\infty \frac{1}{1+q^{j-sN}}\right]\mathbb{E}_{\mathfrak{M}}\left[\prod_{j=1}^N(1+q^{\lambda_{N-j}+j-sN})\right],
    \end{equation*}
    where we use the abbreviation $\mathfrak{M}$ to mean $\mathfrak{M}(a;N,\kappa N)$. We shall proceed by demonstrating $\liminf -\frac1{N^2}\log$ and $\limsup -\frac1{ N^2}\log$ of the r.h.s.~of the above expression is lower bounded and upper bounded by $\mathcal{F}_\kappa(s)$ respectively. 
    
    Throughout the proof we shall use the standard  big $O$ notation and write $a_N=b_N+O(c_N)$ or $a_N\le b_N+O(c_N)$ to mean that $|a_N-b_N| \le C\cdot c_N$ and $a_N \le b_N+C\cdot c_N$ for some constant $C$ depending on $a,\kappa, q,s$, and  $\mathpzc{V}_s$ (defined in \eqref{def.vs}).

\medskip

\noindent\textbf{Lower Bound.}  Let us define 

\begin{align}
    \label{def:mrs}
    \mathsf{R}(s):=\logq\frac{s^2-2s}{2}-\kappa\log (1-a)+\frac{\log(a)}{2}+C_{\kappa},
\end{align} where $C_\kappa$ is defined in \eqref{def.ca} and $\logq:=\log q^{-1}$. Recall the equilibrium energy $F_\kappa(s)=F_{\mathpzc{V}_s}$ defined via \eqref{l.ds97} for the $\mathpzc{V}_s$ potentials defined in \eqref{def.vs}. The $\mathcal{F}_{\kappa}$ function defined in \eqref{def:fs} satisfies
\begin{align}
    \label{idenf}
    \mathcal{F}_\kappa(s)= \mathsf{R}(s)+I_{\mathpzc{V}_s}(\mu_{\phi_{\mathpzc{V}_s}}).
\end{align}
It is well known that $|\lambda|\stackrel{d}{=}\sum_{i=1}^N\sum_{j=1}^{\kappa N} g_{ij}$, where $g_{ij}$'s are i.i.d.~Geometric$(a)$ random variables, i.e., $\Pr(g_{ij}=k)=(1-a)a^k \mbox{ for } k\in \Z_{\ge 0}$, see for example \cite{Johansson_growth_matrices,johansson2000shape}. Thus, by the Chernoff's inequality for all $K>\Ex[|\lambda|]/N^2=\frac{a\kappa}{1-a}$  we have
    \begin{equation}\label{e.tail}
      \left[\prod_{j=0}^\infty \frac{1}{1+q^{j-sN}}\right]\mathbb{E}_{\mathfrak{M}}\left[\prod_{j=1}^N(1+q^{\lambda_{N-j}+j-sN})\mathbf{1}_{|\lambda|>KN^2}\right]  \le \mathbb{P}_{\mathfrak{M}}\left({|\lambda|>KN^2}\right) \le c^{-1}\exp(-cKN^2).
    \end{equation}
for some constant $c>0$ depending on $a$. 
We choose $K$ large enough so that $cK> \mathsf{R}(s)+F_\kappa(s)+1$ and $K>a\kappa/(1-a)$. Next we observe the following string of equalities (to be explained in a moment): 
\begin{align}
\label{e.smidens}
   &  \left[\prod_{j=0}^\infty \frac{1}{1+q^{j-sN}}\right]\mathbb{E}_{\mathfrak{M}}\left[\prod_{j=1}^N(1+q^{\lambda_{N-j}+j-sN})\mathbf{1}_{|\lambda| \le KN^2}\right]  
   \\ \label{e.smiden3} & = \sum_{\lambda : |\lambda|\le KN^2}(1-a)^{\kappa N^2}a^{-N(N-1)/2}\prod_{i=1}^{N}\left(a^{\ell_i}\prod_{j=N+1}^{\kappa N}\frac{\ell_i/N+j/N-1}{j/N-i/N}\right) \\ & \hspace{4cm}\cdot \prod_{1\leq i<j\leq N}\frac{(\ell_i/N-\ell_j/N)^2}{(i/N-j/N)^2}\prod_{j=1}^{N}(1+q^{\ell_{j}-sN})\prod_{j\ge 0}\frac{1}{1+q^{j-sN}}\\
   & \label{e.smidenm} =\sum_{\lambda : |\lambda|\le KN^2}
   \left(\frac{(1-a)^{\kappa N^2}}{ a^{\frac{N(N-1)}{2}}}\prod_{1\leq i<j\leq N}\frac{1}{(\frac{i}{N}-\frac{j}{N})^{2}}\prod_{i=0}^{N-1}\prod_{j=1}^{\kappa N-N}\frac{N}{j+i}\prod_{j\geq 0}\frac{1}{(1+q^{j-sN})}\right)\\ \label{e.smidenm0}
   &\hspace{2cm} \cdot \left(\prod_{i=1}^N a^{\ell_i}\prod_{1\leq i<j\leq N}\left(\frac{\ell_i}{N}-\frac{\ell_j}{N}\right)^2\prod_{i=1}^N\prod_{j=N+1}^{\kappa N} \left(\frac{\ell_i}{N}+\frac{j-N}{N}\right)\prod_{j=1}^N(1+q^{\ell_j-sN})\right)
   \\ & = \sum_{\lambda : |\lambda|\le KN^2} \exp\left(- N^2 [R_N(s)+I_{V_{s,N}}\big(\mu_N(\ell)\big)]\right), \label{e.smiden}
\end{align}
  where
\begin{equation}
\label{defrn}
\begin{aligned}
    R_N(s):=& -\kappa \log(1-a)+\frac{\log a}{2}\left(1-\frac{1}{N}\right) +\frac{2}{ N^2}\sum_{0\leq i<j\leq N-1}\log|(j-i)/N| \\
    &\hspace{1cm}+\frac{1}{ N^2}\sum_{j=1}^{\kappa N-N}\sum_{i=0}^{N-1} \log|(j+i)/N|+\frac{1}{ N^2}\sum_{j=0}^{\infty}\log(1+q^{j-sN})-\logq s,
\end{aligned}
\end{equation}
and 
\begin{equation*}
\begin{aligned}
    V_{s,N}(x):= &x\log\frac1a -\frac{1}{ N}\sum_{j=1}^{\kappa N-N} \log|x+(j/N)|-\frac{1}{ N}\log(1+q^{N(x-s)})+ \logq s.
\end{aligned}
\end{equation*}
The functional $I_{V_{s,N}}(\mu)$ is defined as in \eqref{e.IVmu} with $V=V_{s,N}$. 
The equality in \eqref{e.smiden3} follows using the explicit form of the $z$-measures from \eqref{zmeas}.
The equality in \eqref{e.smidenm} is a straightforward algebraic step, separating the $\ell$-independent and $\ell$-dependent terms into two brackets. The equality in \eqref{e.smiden} follows by noting that $e^{-N^2(R_N(s)+\logq s)}$ is precisely the term within the brackets in \eqref{e.smidenm} and $e^{-N^2(I_{V_{s,N}}(\mu_N(\ell))-\logq s)}$ is precisely the term within the brackets in \eqref{e.smidenm0}.

It is not hard to check that for each fixed $s$
\begin{align*}
    R_N(s) & = -\kappa \log(1-a)+\frac{\log a}{2} +\int_{0}^1 \int_0^1 \log|x-y|\diff x\diff y \\
    &\hspace{1cm}+\int_0^1\int_0^{\kappa -1} \log|x+y|\diff x\diff y+\logq\int_0^\infty (s-x)_+\diff x -(\logq)s +O(N^{-1}) \\ & = \mathsf{R}(s)+O(N^{-1}), 
\end{align*}
where $\mathsf{R}(s)$ is defined in \eqref{def:mrs}. Hereafter, we will drop the $s$ from the notations $\mathsf{R}(s)$, $R_N(s)$, $V_{s,N}$ and $\mathpzc{V}_s$ for convenience. 
Let us quickly verify that $V_N$ and $\mathpzc{V}$ from \eqref{def.vs} satisfy the assumptions in Lemma \ref{l.dd22}. Note that
\begin{align*}
   V_N'(x)=\log\frac1a -\frac{1}{ N}\sum_{j=1}^{\kappa N-N} \frac1{x+(j/N)}+\logq\frac{q^{N(x-s)}}{1+q^{N(x-s)}}.   
\end{align*}
It follows from the above explicit expression that $|V_N'(x)|$ is uniformly bounded above by $C(1+\min(x^{-1},\log N))$. On the other hand, for $y\in [(j-1)/N,j/N]$ we have
\begin{align*}
  0\le  \log \bigg(x+\frac{j}N\bigg) -\log (x+y)= \log \left[1+\frac{\frac{j}N-y}{x+y}\right] \le \frac{\frac{j}N-y}{x+y} \le \frac{1}{N(x+y)}.
\end{align*}
Thus,
\begin{align*}
  0 \le \frac1N\sum_{j=1}^{\kappa N-N}\log (x+\frac{j}{N}) - \int_0^{\kappa -1} \log (x+y)\diff y \le \frac{1}{N} \int_0^{\kappa-1}\frac1{x+y}\diff y = \frac1N\log\frac{x+\kappa-1}{x},
\end{align*}
which in turn implies \eqref{e.vclose}. Thus $V_N$ and $\mathpzc{V}$ satisfy the assumptions of Lemma \ref{l.dd22}. Thus, for each $\ell\in \mathbb{W}^N$, by Lemma \ref{l.dd22}\ref{dda} and \ref{ddb} we have 
 \begin{align*}
  I_{V_N}(\mu_N(\ell)) & \ge I_{V_N}(\widetilde{\mu}_N(\ell))+ O(N^{-1}\log N) \\ & \ge  I_{\mathpzc{V}}(\widetilde{\mu}_N(\ell))+O(N^{-1}\log N) \ge I_{\mathpzc{V}}(\mu_{{\phi}_\mathpzc{V}})+O(N^{-1}\log N),
\end{align*}
where in the last line we used the fact the $\mu_{\phi_{\mathpzc{V}}}$ minimizes $I_{\mathpzc{V}}(\cdot)$. This along with $R_N=\mathsf{R}+O(N^{-1})$ implies
\begin{align*}
 \eqref{e.smiden}  \le \exp\big(-N^2[\mathsf{R}+I_{\mathpzc{V}}(\mu_{{\phi}_\mathpzc{V}})]+O(N\log N)\big)\sum_{n=1}^{KN^2} \mathsf{p}_n,
\end{align*}
where $\mathsf{p_n}$ is the number of partitions with $|\lambda|=n$. Using the well known fact (see \cite[Eq.~(1.15)]{romik_2015} for example) that $\mathsf{p}_n \le e^{C\sqrt{n}}$ for some absolute constant $C>0$, we see that the sum on the r.h.s.~of the above equation is at most $e^{\widetilde{C}N}$ for some constant $\widetilde{C}>0$ depending only on $K$. Combining this with the tail estimate from \eqref{e.tail}, in view of \eqref{idenf}, we arrive at the lower bound.

\medskip

\noindent\textbf{Upper Bound.}  By the exact same computation as in \eqref{e.smidens}-\eqref{e.smiden} we have 
\begin{equation}\label{e.smiden2}
    \begin{aligned}
    \left[\prod_{j=0}^\infty \frac{1}{1+q^{j-sN}}\right]\mathbb{E}_{\mathfrak{M}}\left[\prod_{j=1}^N(1+q^{\lambda_{N-j}+j-sN})\right] 
   & = \sum_{\ell\in \mathbb{W}^N} \exp\left(-N^2\cdot [R_N+I_{V_N}\big(\mu_N(\ell)\big)]\right) \\ & \ge \exp\left(-N^2\cdot [R_N+I_{V_N}\big(\mu_N(\ell_*)\big)]\right),
\end{aligned}
\end{equation}
where $\ell_*\in \mathbb{W}^N$ is the one coming from Lemma \ref{l.dd22}\ref{ddc}. By Lemma \ref{l.dd22}\ref{ddc},  we have
\begin{align*}
    N^2 I_{V_N}(\mu_N(\ell_*)) &  \le N^2 I_{\mathpzc{V}}(\mu_{\phi_{\mathpzc{V}}}) +O(N\log N).
\end{align*}
Plugging this inequality back in \eqref{e.smiden2} and using the fact that $R_N=\mathsf{R}+O(N^{-1})$ we see that 
\begin{equation*}
    \begin{aligned}
\mathbb{E}_{\mathfrak{M}}\left[\prod_{j\geq 0}\frac{1+q^{\lambda_{N-j}+j-sN}}{1+q^{j-sN}}\right]  
   & \ge  \exp\big(- N^2[\mathsf{R}+I_{{\mathpzc{V}}}(\mu_{\phi_{\mathpzc{V}}})]+O(N\log N)\big).
\end{aligned}
\end{equation*}
Taking $\limsup_{N\to \infty} -\frac1{ N^2}\log$ both sides, in view of \eqref{idenf}, we get the desired upper bound. 
\end{proof}

An easy consequence of the above theorem is the following:

\begin{lemma}\label{l.fha} $\mathcal{F}_\kappa$ is a non-negative non-decreasing function. $\mathcal{F}_\kappa(s)=0$ for all $s<\lln$.
\end{lemma}
\begin{proof} From the relation \eqref{e.multldp} $\mathcal{F}_\kappa$, it is clear that $\mathcal{F}_\kappa$ is non-negative and non-decreasing. Fix any $s<\lln$. Note that
\begin{align*}
    \mathbb{E} \left[\prod_{i= 0}\frac{1}{1+ q^{\mathfrak{h}(\lfloor \kappa N \rfloor,N)+i-sN}}\right] \ge \Pr(\mathfrak{h}(\lfloor \kappa N \rfloor,N) \ge s N) \prod_{i=0}^{\infty} \frac{1}{1+q^i}.
\end{align*}  
The product is a constant and $\Pr(\mathfrak{h}(\lfloor \kappa N \rfloor,N) \ge s N) \to 1$ as $s<\lln$ and $\lln$ is the law of large numbers. Thus, taking $\lim_{N\to\infty} -\frac1{N^2}\log$, in view of \Cref{p.multldp}, we see that $\mathcal{F}_\kappa(s)=0$.
\end{proof}

\subsection{Parabolic behavior of $\mathcal{F}_\kappa(s)$} \label{sec.2.2} In this subsection we determine an explicit value of $\mathcal{F}_\kappa(s)$ for large $s$. The key ideas essentially comes from potential theory arguments.

\begin{proposition} \label{p.energystab} For large enough $s$ we have 
\begin{align}\label{f.para}
    \mathcal{F}_\kappa(s)=\logq \frac{(s-1)^2}{2}+\kappa \log \frac{1-aq}{1-a}.
\end{align}
\end{proposition}

Recall the expression of $\mathcal{F}_\kappa(s)$ from \eqref{def:fs} and \eqref{def:fs2}. There are two key steps in proving \Cref{p.energystab}:
\begin{enumerate}[label=(\Alph*),leftmargin=18pt]
    \item \label{1ststep} We first show that the equilibrium energy $F_\kappa(s)$ becomes $F_\kappa(\infty)$ for large enough $s$ (\Cref{l.energystab}). We achieve this by controlling the support of $\phi_{\mathpzc{V}_s}$ uniformly as $s\to \infty$ (\Cref{l.ds972}). 
    \item \label{2ndstep} We then obtain explicit expression for $\phi_{\mathcal{V}_\infty}$ using the variational characterization of the equilibrium measure. This allows us to compute $F_\kappa(\infty)$ explicitly.
\end{enumerate}
     
Let us now work out the details of the above two steps. We begin with a uniform compact support lemma for a class of well-behaved potentials.

\begin{lemma} \label{l.ds972} Let $\mathcal{C}$ be a collection of potentials satisfying the following.  Suppose $\sup_{V\in \mathcal{C}, x\in [0,2]} V(x) <\infty,$ and for every $K>0$, there exists $R>0$ such that for all $V\in \mathcal{C}$ we have
\begin{align}\label{1cond}
    & \log \left[|x-y|e^{-\frac12V(x)-\frac12V(y)}\right]^{-1} > K \mbox{ for all } x,y\ge R, \\
    & \log \left[|x-y|e^{-\frac12V(x)}\right]^{-1} > K \mbox{ for all } x\ge R, y\in [0,R]. \label{2cond}
\end{align}
Then there exists $R_0$ such that the support of $\phi_V$ lies in $[0,R_0]$ for all $V\in \mathcal{C}$.
\end{lemma}
\begin{proof} The proof essentially follows the proof of the claim in the proof of Theorem 2.1 in \cite{ds97},  We write it here for completeness. Take any $\mu \in \mathcal{A}_\infty$ with density $\phi$. Decompose $\mu:=\mu_1+\mu_2$ where $\mu_1=\mu|_{[0,R]}$ and $\mu_2=\mu|_{[R,\infty)}$. Set $\delta=\mu_2(\mathbb{R}_+)$. Our goal is to show that for all $R$ large enough, there exists $\eta \in \mathcal{A}_{\infty}$ which is supported on $[0,R]$ and satisfies $I_V(\eta)\le I_V(\mu)$ for all $V\in \mathcal{C}$. Let $\nu\in \mathcal{A}_\infty$ whose density is proportional to $(1-\phi)(x)\ind_{x\in [0,2]}$. Note that the proportionality constant $\int_0^2 (1-\phi(x))dx=2-\int_0^2 \phi(x)dx \ge 1.$ Thus, for every $x\ge 0$ 
\begin{align}\label{nuine}
   \mbox{the density of $\nu$ at $x$ is at most $1-\phi(x) \le 1$.}
\end{align}
Let $\eta:=\mu_1+\delta \nu \in \mathcal{A}_\infty$. Note that the $\eta$'s support is contained in $[0,R]$. For two measures $\nu_1,\nu_2$, define
\begin{align}\nonumber
    \langle \nu_1,\nu_2\rangle & :=\int\int \log \left[|x-y|e^{-\frac12V(x)-\frac12V(y)}\right]^{-1} \diff \nu_1(x)\diff \nu_2(y) \\ & = \int\int \log \left[|x-y|e^{-\frac12V(y)}\right]^{-1} \diff \nu_1(x)\diff \nu_2(y)+\frac12\nu_2(\R_{+})\int V(x)\diff \nu_1(x). \label{2exp}
\end{align}
Let $A:=\sup_{V\in \mathcal{C},x\in [0,2]} V(x)$.
Let us note that
\begin{align}\nonumber
    I_V(\eta)-I_V(\mu) & =\langle \mu_1 +\delta\nu, \mu_1+\delta\nu\rangle-\langle \mu_1 +\mu_2, \mu_1+\mu_2\rangle \\ & 
    =\delta^2\langle \nu,\nu\rangle+2\delta\langle \mu_1 ,\nu\rangle -  \langle \mu_2 ,\mu_2\rangle-2\langle \mu_1 ,\mu_2\rangle \label{3exp}
\end{align}
Using \eqref{1cond}, we obtain $\langle \mu_2 ,\mu_2\rangle \ge \delta^2 K$. 
As $\nu(\mathbb{R}_+)=1$ and $\mu_2(\mathbb{R}_+)=\delta$, using the second expression for $\langle \cdot,\cdot \rangle$ from \eqref{2exp}, we obtain
\begin{align*}
    & 2\delta\langle \mu_1 ,\nu\rangle -2\langle \mu_1 ,\mu_2\rangle \\ & \hspace{1cm}= 2\delta \int\int \log \frac{1}{|x-y|e^{-\frac12V(y)}}\diff \mu_1(x)\diff\nu(y) - 2\int\int \log \frac{1}{|x-y|e^{-\frac12V(y)}}\diff \mu_1(x)\diff\mu_2(y) \\ & \hspace{1cm} \le 2\delta \int\int \log \frac{1}{|x-y|e^{-\frac12V(y)}}\diff \mu_1(x)\diff\nu(y) - 2\int\int K\diff \mu_1(x)\diff\mu_2(y)
\end{align*}
where the last inequality follows from \eqref{2cond}. As $\mu_1(\R_+)=1-\delta$ and $\mu_2(\R_+)=\delta$, plugging these estimates back in \eqref{3exp} yields
\begin{align}\label{4exp}
    I_V(\eta)-I_V(\mu)  \le \delta^2\langle \nu,\nu\rangle+2\delta\int\int \log \frac{1}{|x-y|e^{-\frac12V(y)}}\diff \mu_1(x)\diff\nu(y)-\delta^2K-2\delta(1-\delta)K.
\end{align}
We will now bound $\langle \nu,\nu\rangle$ and the double integral above. Let us note that
\begin{align*}
\langle \nu,\nu\rangle & = \int_0^2\int_0^2 \log \left[|x-y|e^{-\frac12V(x)-\frac12V(y)}\right]^{-1}\diff \nu(x) \diff \nu(y) \\ & \le \int_0^2\int_0^2 \left[A-\log 2+\log\frac2{|x-y|}\right]\diff \nu(x) \diff \nu(y) \\ & =A-\log 2+ \int_0^2\int_0^2 \log\frac2{|x-y|}\diff \nu(x) \diff \nu(y)  \le  A-\log 2+\int_0^2\int_0^2 \log\frac{2}{|x-y|}\diff x \diff y =: C_1.
\end{align*}
Let us briefly explain the above inequalities. As $\nu$ is supported on $[0,2]$, using the fact that $V(x)\le A$ for $x\in [0,2]$, we derive the first inequality. We also added and subtracted $\log 2$ in the second line so that $\log \frac2{|x-y|}$ remains non-negative on the range of the double integral. The second inequality is a consequence of \eqref{nuine}. A similar trick leads to
\begin{align*}
  \int\int \log \frac{1}{|x-y|e^{-\frac12V(y)}}\diff \mu_1(x)\diff\nu(y) &   \le \frac12A+\int_0^R \int_0^2 \log\frac{1}{|x-y|}\diff\mu_1(x)\diff \nu(y) \\ & \le \frac12A+\int_0^3 \int_0^2 \left[\log\frac{1}{|x-y|}\right]\diff \mu_1(x)\diff \nu(y) \\ & \le \frac12A+\int_0^3 \int_0^2 \left[\log\frac{3}{|x-y|}\right]\diff \mu_1(x)\diff \nu(y) \\ & \le \frac12A+\int_0^3 \int_0^2 \left[\log\frac{3}{|x-y|}\right]\diff x\diff y  =:C_2. 
\end{align*}
Let us briefly explain the above lines. The first inequality follows by using the fact that $V(x)\le A$ for $x\in [0,2]$. The second inequality follows by noting that the log term is negative when $x \ge 3$ and $y\in [0,2]$. In third line, we have just increased $1$ to $3$, so that $\log \frac{3}{|x-y|}$ remains non-negative on the range of integral. Finally, the last inequality is due to the fact that the densities corresponding to $\mu_1$ and $\nu$ are less than $1$ due to the definition of $\mathcal{A}_{\infty}$ and \eqref{nuine} respectively.

Choosing $K>3\max(C_1,C_2)$, the above bounds and \eqref{4exp} ensure $I_V(\eta)\le I_V(\mu)$. Thus the support of the minimizers all lie in $[0,R_0]$ for some $R_0$ large enough. 
\end{proof}

Using the above lemma we can now complete Step \ref{1ststep}:
\begin{lemma} \label{l.energystab} For large enough $s$ we have 
    $F_\kappa(s)=F_\kappa(\infty):=F_{\mathpzc{V}_\infty}$.
\end{lemma}
\begin{proof} Let us write $I_{s}(\cdot)$ for $I_{\mathpzc{V}_s}(\cdot)$, $\mu_s$ for $\mu_{\phi_{\mathpzc{V}_s}}$, and $\mathpzc{V}_s$ for $\mathpzc{V}$ to stress the dependence on $s$.  Note that for all large enough $x$, $\mathpzc{V}_{s}(x) \ge \beta x$ uniformly over all $s\in [0,\infty]$ for some $\beta>0$. Then one can check that $(\mathpzc{V}_s)_{s\ge 2}$ satisfies the hypothesis of Lemma \ref{l.ds972}. Hence there exists $R_0$, such that for all $s\in [0,\infty]$, the support of $\phi_{\mathpzc{V}_s}$ lies in $[0,R_0]$. However for $s\ge R_0$, $\mathpzc{V}_s|_{[0,R_0]}=\mathpzc{V}_{\infty}|_{[0,R_0]}$. Thus, $I_s(\mu_s)=I_{\infty}(\mu_{\infty})=F_\kappa({\infty})$ for $s\ge R_0$. 
\end{proof}

\begin{proof}[Proof of \Cref{p.energystab}] 
Let us now carry out Step \ref{2ndstep} described just after the proposition, i.e., computing $\phi_{\mathpzc{V}_\infty}$ and $F_\kappa(\infty)$. There is a standard route in potential theory that gives us the equilibrium measure and the corresponding equilibrium energy. For completeness, we write the key steps here and postpone the tedious calculations to the \Cref{sec.app}. The following lemma is the key to obtaining exact expression for $\phi_{\mathpzc{V}_\infty}$.

\begin{lemma}[Theorem 2.1(d) in \cite{ds97}] \label{l.ds973} Suppose $V :[0,\infty)\to \R$ is continuous potential satisfying $V(x)\ge 2\log(1+x^2)$. Assume that there exists $\phi\in \mathcal{A}_\infty$ such that
\begin{enumerate}[label=(\alph*),leftmargin=18pt]
    \item $\displaystyle \int_0^\infty k_V(x,y)\phi(x)\diff x \ge \lambda$ if $\phi(y)=0$,
     \item $\displaystyle \int_0^\infty k_V(x,y)\phi(x)\diff x \le \lambda$ if $\phi(y)=1$,
      \item \label{cd} $\displaystyle \int_0^\infty k_V(x,y)\phi(x)\diff x = \lambda$ if $\phi(y) \in (0,1)$,
\end{enumerate}
for some $\lambda\in \R$. Then $\phi=\phi_V$.
\end{lemma}

Given the above lemma, the usual strategy to obtain the equilibrium measure is to first make an educated guess for $\phi_V$ then verify that it indeed satisfies the conditions of \Cref{l.ds973}. For nice enough $V$s, the ansatz can be obtained by differentiating the relation in \Cref{l.ds973}\ref{cd} (c.f.~\cite[Section 6]{johansson2000shape}) or by the so called Nekrasov's equation (c.f.~\cite[Section 2]{borodin2017gaussian}). Fortunately, $\mathpzc{V}_\infty$ is nice enough to apply either of the aforementioned two techniques. We skip the ansatz calculations (as it is not rigorous anyway) and report here the expression for the measure.

\medskip

Write $p=aq$, $c=\frac{(1-\sqrt{p\kappa})^2}{1-p}$, and $d=\frac{(\sqrt{p\kappa}+1)^2}{1-p}$.  
When $\kappa>1/p$ define 
\begin{align}\label{equi.nosat}
    \phi(x):=\begin{dcases}
        \frac{1}{2}-\frac{1}{\pi}\tan^{-1}\left(\frac{xp+x+p\kappa-1}{\sqrt{4xp(x+\kappa-1)-(xp+x+p\kappa-1)^2}}\right), & x \in  \left[c,d\right], \\
        0 & x \notin [c,d].
    \end{dcases}
\end{align}
When $\kappa\le 1/p$ define 
\begin{align}\label{equi.sat}
    \phi(x):=\begin{dcases}
        1 & x\in \left[0,c\right], \\
        \frac{1}{2}-\frac{1}{\pi}\tan^{-1}\left(\frac{xp+x+p\kappa-1}{\sqrt{4xp(x+\kappa-1)-(xp+x+p\kappa-1)^2}}\right), & x \in  \left[c,d\right], \\
        0 & x \ge d.
    \end{dcases}
\end{align}

\begin{figure}[h!]
    \centering
    \begin{overpic}[width=7cm]{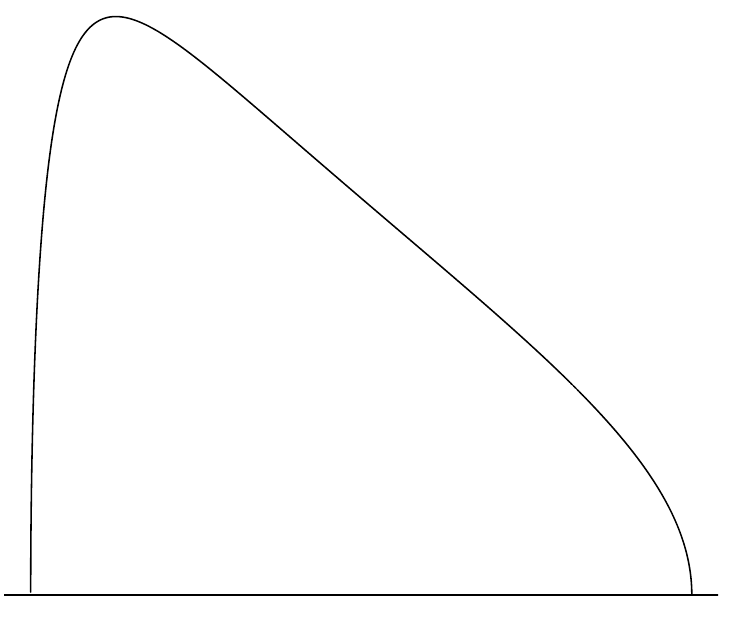}
        \put(3.35,2.5){$\mid$}
        \put(3.3,-2){$c$}
        \put(93.45,2.5){$\mid$}
        \put(93.3,-2){$d$}
    \end{overpic}
    \hspace{1cm}
    \begin{overpic}[width=7cm]{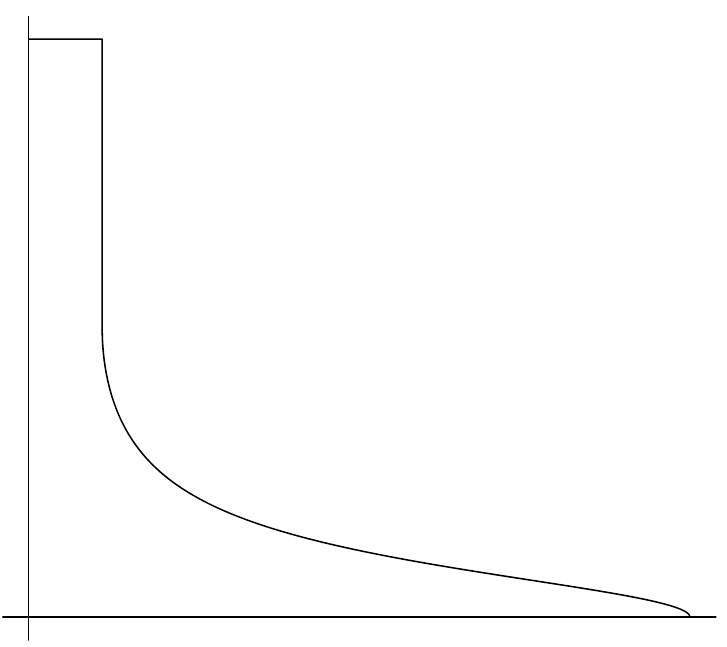}
        \put(13.35,2.5){$\mid$}
        \put(13.3,-2){$c$}
        \put(94.25,2.5){$\mid$}
        \put(94.2,-2){$d$}
    \end{overpic}
    \caption{$\phi=\phi_{\mathpzc{V}_\infty}$ when $\kappa>1/p$ (left) and $\kappa<1/p$ (right).}
    \label{fig:phi}
\end{figure}
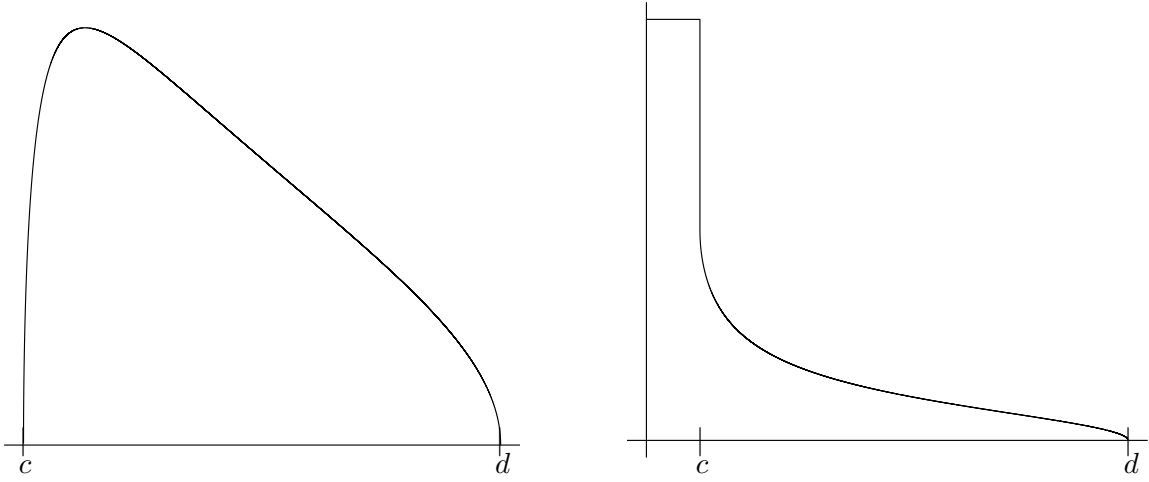

A plot for $\phi$ is given in \Cref{fig:phi}. Recall $k_V$ from \eqref{e.IVmu}. We claim that when $\kappa > 1/p$
\begin{align}\label{intb}
    \int_c^d k_{\mathpzc{V}_\infty}(x,y)\phi(x)\diff x  = \begin{dcases}
        \zeta_-(y) & y\in [0,c] \\
        \kappa \log(1-p)-\frac12\log p-C_\kappa & y\in [c,d] \\
       \zeta_+(y)  & y\in [d,\infty)
    \end{dcases}
\end{align} 
where $\zeta_-$ is decreasing on $[0,c]$ and $\zeta_+$ is increasing on $[d,\infty)$. One has exact expressions for these functions; but they are not important in our analysis. Here the constant $C_\kappa$ comes from \eqref{def.ca}. When $\kappa \le 1/p$
\begin{equation}
\begin{aligned}\label{inta}
    &\int_0^d k_{\mathpzc{V}_\infty}(x,y)\phi(x)\diff x  \\
    &= \begin{dcases} \overline\zeta_-(y)
         & y\in [0,c] \\
    -\frac{\kappa\log \kappa+\log p}{2}+\frac{\kappa+1}{2}(\log(1-p)+1)+\frac{1}{2}\int_0^d \mathpzc{V}_\infty(x)\phi(x)\diff x    & y\in [c,d] \\
      \overline\zeta_+(y)   & y\in [d,\infty)
    \end{dcases}
\end{aligned}
\end{equation}
where $\overline\zeta_-(y), \overline\zeta_+(y)$ are both increasing functions. Again, exact expressions of these functions are not important.  Furthermore, for both the case $\kappa>1/p$ and $\kappa \leq 1/p$
we have
\begin{align}\label{intc}
    F_\kappa(\infty):= \int_0^d \int_0^d k_{\mathpzc{V}_\infty}(x,y)\phi(x)\phi(y)\diff x \diff y=\kappa \log (1-p)-\frac12\log p-C_\kappa.
\end{align}

 The verification of \eqref{intb}, \eqref{inta}, and \eqref{intc} is purely computational and relies on few integration tricks (see \cite[Section 6]{das2022large} for similar computations involving Jack measures with different  specializations).  We postpone this to \Cref{sec.app}. Note that, in view of \Cref{l.ds973}, \eqref{intb} and \eqref{inta} implies that $\phi=\phi_{\mathpzc{V}_\infty}$. Recall that $p=aq$. In view of \Cref{l.energystab} and the definition of $\mathcal{F}_\kappa(s)$ from \eqref{def:fs}, when $\kappa \ge 1$, we have
\begin{align*}
    \mathcal{F}_\kappa(s) & =\logq \frac{s^2-2s}{2}-\kappa\log(1-a)+\frac{\log(a)}{2}+\kappa \log(1-aq)-\frac{\log (aq)}{2} \\ & =\logq \frac{(s-1)^2}2+\kappa\log\frac{1-aq}{1-a},
\end{align*}
for large enough $s$. When $\kappa\in (0,1)$, using \eqref{def:fs2}, for large enough $s$ we have
\begin{align*}
    \mathcal{F}_\kappa(s) & =\logq\frac{s^2-2s\kappa+2\kappa-2\kappa^2}{2}-\logq\int_{\kappa}^1 (s+x-1)\diff x \\ & \hspace{2cm}-\kappa\log (1-a)+\kappa^2\frac{\log(a)}{2}+\kappa\log(1-p)-\frac{\kappa^2}2\log (aq) \\ & =\logq \frac{(s-1)^2}2+\kappa\log\frac{1-aq}{1-a}.
\end{align*}
We thus arrive at \eqref{f.para}.
\end{proof}

\section{Proof of LDP results} \label{sec.4}

In this section, we combine the results from the previous two sections to prove \Cref{thm.main,thm.main2}. Towards this end, let us first define the prelimiting rate function. For $s \le 1$ define
\begin{align}\label{def:tta}
\Phi_{\kappa,N}(s):=-\frac1{N^2}\log\Pr(\h(\lfloor\kappa N\rfloor,N) \ge sN).
\end{align}
For $s>1$, we set $\Phi_{\kappa,N}(s):=\infty$. We record some easy properties of $\Phi_{\kappa,N}$ in the following lemma.
\begin{lemma}\label{l.tprop} $\Phi_{\kappa,N}$ is non-decreasing on $(-\infty,1]$ and $\Phi_{\kappa,N}(0)=0$. We have $\lim_{N\to\infty}\Phi_{\kappa,N}(s)=0$ for all $s<\lln$ and $\Phi_{\kappa,N}(1)=\kappa \log \frac{1-aq}{1-a}$.
\end{lemma}
\begin{proof} From the definition, it is clear that $\Phi_{\kappa,N}$ is non-decreasing on $(-\infty,1]$ and $\Phi_{\kappa,N}(0)=0$. For $s<\lln$, we have $\Pr(\h(\lfloor\kappa N\rfloor,N) \ge sN) \to 1$ as $\lln$ is the law of large numbers. Thus,  $\lim_{N\to\infty}\Phi_{\kappa,N}(s)=0$ for $s<\lln$. Finally, note that $\h(\lfloor\kappa N\rfloor,N)\ge N$ implies $\h(\lfloor\kappa N\rfloor,N)=N$, which happens when all the paths travel horizontally. Thus, 
    \begin{align*}
        \Pr(\h(\lfloor\kappa N\rfloor,N) \ge N)=\Pr(\h(\lfloor\kappa N\rfloor,N) =N)= \left(\frac{1-a}{1-aq}\right)^{\kappa N^2}.
    \end{align*}
    This leads to the desired value of $\Phi_{\kappa,N}(1)$.
\end{proof}

A quick consequence of \Cref{t.main0} is the following weak midpoint convexity for $\Phi_{\kappa,N}$. 
\begin{proposition}
\label{cor:tconvex} For any $\e>0$, there exists $N_{\e} \in \Z_{\ge 1}$ such that for all $v_1,v_2\in \R$ and for all  $N\ge N_\e$ we have
\begin{align*}
\frac12\big(\Phi_{\kappa,N}(v_1)+\Phi_{\kappa,N}(v_2)\big)+\e \ge \Phi_{\kappa,N}\bigg(\frac{v_1+v_2}{2}-N^{-1/6}\bigg).
\end{align*}
\end{proposition}

\begin{proof} Assume $v_1<v_2$. If $v_2\le 0$ there is nothing to prove. So, assume $v_2>0$. If $v_1\ge0$, using \Cref{t.main0} with $r_1=v_1N$ and $r_2=v_2N$ and $M=\kappa N$, we have
\begin{align}\label{eref}
\frac12\big(\Phi_{\kappa,N}(v_1)+\Phi_{\kappa,N}(v_2)\big)+\frac{C}{N^{2/5}} \ge \Phi_{\kappa,N}\bigg(\frac{v_1+v_2}{2}-4\kappa^{2/5}N^{-1/5}\bigg).
\end{align}
If $v_1<0$, using the fact that $\Phi_{\kappa,N}(v_1)=\Phi_{\kappa,N}(0)=0$  we have
\begin{equation}\label{eref2}
    \begin{aligned} 
    &\frac12\big(\Phi_{\kappa,N}(v_1)+\Phi_{\kappa,N}(v_2)\big)+\frac{C}{N^{2/5}} \\
    & = \frac12\big(\Phi_{\kappa,N}(0)+\Phi_{\kappa,N}(v_2)\big)+\frac{C}{N^{2/5}} \\ &  \ge \Phi_{\kappa,N}\bigg(\frac{v_2}{2}-4\kappa^{2/5}N^{-1/5}\bigg) \ge \Phi_{\kappa,N}\bigg(\frac{v_1+v_2}{2}-4\kappa^{2/5}N^{-1/5}\bigg).
\end{aligned}
\end{equation}
Here in \eqref{eref} and \eqref{eref2}, the constant $C$ depends only on $a$, $q$ and $\kappa$. The penultimate inequality follows by applying \eqref{eref} with $v_1=0$. Assuming $N$ large enough, we can ensure $4\kappa^{2/5}N^{-1/5} \le N^{-1/6}$. This proves the proposition.
\end{proof}

We next relate $\Phi_{\kappa,N}$ to the auxillary rate function $\mathcal{F}_\kappa$ that we obtain in \Cref{sec: LDP_schur}. Let us first see why the $q$-Laplace transform appearing in \eqref{e.multldp} is indeed a tail probability of a shifted height function. {This relationship is essentially due to a matching from \cite{IMS_matching}}. Towards this end, we define two independent random variables $S$ supported on $\mathbb{Z}$ and $\chi$ supported on $\Z_{\ge 0}$ by the following probability mass functions
\begin{equation}
\label{chis}
    \begin{aligned}
    & \Pr(S =k) \propto q^{k(k-1)/2}, \qquad  k\in \mathbb{Z} \\ & \Pr(\chi=k)=\prod_{j=k+1}^{\infty} (1-q^j)-\prod_{j=k}^{\infty} (1-q^j)\qquad k\in \Z_{\ge 0}.
\end{aligned}
\end{equation}
For reference, $S$ is said to follow $\operatorname{Theta}(1/\sqrt{q};q)$ distribution and $\chi$ is said to follow $q$-$\operatorname{Geo}(q)$ distribution. We will only require the expression for the cumulative distribution function of $\chi+S$ which we quote from Lemma 2.4 in \cite{IMS_matching}: 
\begin{equation}\label{chisd}
    \mathbb{P}(\chi + S \le n) =  \prod_{i =0}^\infty \frac{1}{1+q^{n+i}}.
\end{equation}
In view of this, we have that $q$-Laplace transform of the height function of S6V is nothing but the tail probability of the height function shifted by $\chi+S$:
\begin{align*}
    \mathbb{E} \left[\prod_{i= 0}^\infty\frac{1}{1+ q^{\mathfrak{h}(\lfloor\kappa N\rfloor,N)+i-sN}}\right] = \Pr\big(\mathfrak{h}(\lfloor\kappa N\rfloor,N)-\chi-S \ge sN\big)
\end{align*}
where $\chi,S$ are drawn independently of the S6V dynamics and has the distributions described in \eqref{chis}. Note that $\chi$ is a non-negative random variable and $S$ is a discrete Gaussian-type random variable. Indeed it is straightforward to check that for any $v\in \R$ 
\begin{align*}
   e^{-N^2g(v)-CN} \le \Pr(S=vN) \le e^{-N^2g(v)+CN} 
\end{align*}
 for some constant $C>0$ depending on $q$, where $g(v)=\frac{v^2}{2}\logq$ (recall $\logq=\log q^{-1}$). Since \Cref{p.multldp} gives us that $\mathcal{F}_\kappa$ is the lower-tail rate function of $\chi+S-\h(\lfloor \kappa N\rfloor,N)$, it is natural to expect the following relation between $\mathcal{F}_\kappa$ and $\Phi_{\kappa,N}$.
\begin{proposition} \label{prop:f_limit_g+T}
				Let $x \in \mathbb{R}$ and define the function $g(x)=\frac{x^2}{2}\logq$. Then  we have
				\begin{equation}
                \label{eq:lim_inf_conv}
					\mathcal{F}_\kappa(s) = \lim_{N\to \infty}  \left( g \oplus \Phi_{\kappa,N} \right)(s):=\lim_{N\to \infty}\inf_{y \in \mathbb{R}} \left\{ g(y) + \Phi_{\kappa,N}(s-y) \right\}.
				\end{equation}
			\end{proposition}
\begin{proof} Suppose $s< \lln$. We have $\mathcal{F}_\kappa(s)=0$ from \Cref{l.fha}. On the other hand, $g(0)+\Phi_{\kappa,N}(s)=\Phi_{\kappa,N}(s) \to 0$ as $N\to \infty$ by \Cref{l.tprop}. Thus \eqref{eq:lim_inf_conv} follows for $s<\lln$. Let us now fix any $s\ge \lln$ and write $\h_N:=\h(\lfloor\kappa N\rfloor,N)$. Let us set
				\begin{equation*}
					\hat{y} := \mathrm{argmax} \left\{ \mathbb{P}(S = -yN) \mathbb{P}(\h_N \ge (s-y)N) : y\in \tfrac{1}{N} \mathbb{Z}\cap [s-1,s+1] \right\},
				\end{equation*}
  and  define
    \begin{align*}
        (g\hat\oplus\Phi_{\kappa,N})(s):=\inf_{y\in \tfrac{1}{N} \mathbb{Z}\cap [s-1,s+1]} \{g(y)+\Phi_{\kappa,N}(s-y)\}.
    \end{align*} Using the fact that $\h_N\le N$ and $\chi\ge 0$, we have
\begin{align*}
    \Pr(\h_N-\chi-S \ge sN) & \le \Pr(\h_N-S \ge sN) \\ 
& \le \Pr(S\in [-(s+1)N,(1-s)N],\h_N-S\ge sN)+\Pr(S\le -(s+1)N) \\ & \le 2N\Pr(S=-\hat{y}N,\h_N\ge -\hat{y}N+sN)+\Pr(S\le -(s+1)N),
\end{align*}
and
\begin{align*}
    \Pr(\h_N-\chi-S \ge sN) & \ge \Pr(\chi=0)\Pr(\h_N-S \ge sN) \\ 
& \ge \Pr(\chi=0)\Pr(S\in [-(s+1)N,(s-1)N],\h_N-S\ge sN) \\ 
& \ge \Pr(\chi=0)\Pr(S=-\hat{y}N,\h_N\ge -\hat{y}N+sN).
\end{align*}
By the definition $\hat{y}$ and explicit distribution of $S,\chi$, we have
\begin{align*}
    \Pr(S=-\hat{y}N,\h_N\ge -\hat{y}N+sN) = e^{-N^2(g\hat\oplus\Phi_{\kappa,N})(s)+O(N)}, \qquad \Pr(S\le -(s+1)N) \le e^{-N^2g(s+1)+O(N)},
\end{align*}
where the $O$ term depends only on $q$.
Thus we get
\begin{equation}\label{eq:limconv2}
					\limsup_{N\to \infty} \min\left\{ \left( g \hat\oplus \Phi_{\kappa,N} \right)(s), g(s+1) \right\} \le \mathcal{F}_\kappa(s) \le \liminf_{N\to\infty} \left( g \hat\oplus \Phi_{\kappa,N} \right)(s). 
				\end{equation}

Note that $ \left(g\oplus \Phi_{\kappa,N} \right)(s) \le \left(  g \hat\oplus \Phi_{\kappa,N} \right)(s)   \le g(s+1)+\Phi_{\kappa,N}(-1)=g(s+1)$.  Thus, $\mathcal{F}_\kappa(s) \ge \limsup_{N\to \infty} \left(g\oplus \Phi_{\kappa,N} \right)(s)$. Thus, the proposition follows once we show
    \begin{align}\label{eq:tiloplus}
        \lim_{N\to \infty} |\left(g\oplus \Phi_{\kappa,N} \right)(s)-\left(g\hat\oplus \Phi_{\kappa,N} \right)(s)|=0.
    \end{align}
Towards this end, fix any $\varepsilon>0$. Since $g(y)\to \infty$ as $|y|\to \infty$ and $\sup_{N>0, y\le 1} \Phi_{\kappa,N}(y) \le \kappa \log \frac{1-aq}{1-a}$, we may find a sequence $\{z_N\}_N$ such that $\sup_N |z_N|<\infty$ and 
$$g(z_N)+\Phi_{\kappa,N}(s-z_N)-(g\oplus \Phi_{\kappa,N})(s) \le \e.$$ 
Clearly, $z_N\ge s-1$ for all $N$. Let $z$ be any limit point of the sequence $\{z_N\}_N$. We have 
$$g(z_N)-\e\le g(z_N)+\Phi_{\kappa,N}(s-z_N)-\varepsilon \le (g\oplus \Phi_{\kappa,N})(s) \le g(s)+\Phi_{\kappa,N}(0)=g(s).$$ 
Taking subsequential limit followed by $\e\downarrow 0$ we obtain $g(z) \le g(s)$. Thus, we have $z\le s$. Note that $\Phi_{\kappa,N}(N^{-1}\lfloor Nz_N\rfloor)=\Phi_{\kappa,N}(z_N)$. Since all limit points of $\{z_N\}$ are in $[s-1,s]$, for all large enough $N$ we can ensure $N^{-1}\lfloor Nz_N\rfloor \in \frac1N\Z\cap[s-1,s+1]$. Thus,
$$(g\hat\oplus\Phi_{\kappa,N})(x)-g(\oplus \Phi_{\kappa,N})(x) \le |g(z_N)-g(N^{-1}\lfloor Nz_N\rfloor)|+\varepsilon.$$
Taking $\limsup_{N\to\infty}$ on both sides and noticing that $\e$ is arbitrary, we arrive at \eqref{eq:tiloplus}. \end{proof}

\begin{proposition}\label{prop:fcon} $\mathcal{F}_\kappa$ is a non-negative convex non-decreasing function taking values in $[0,+\infty)$.
\end{proposition}

\begin{proof} From the definition of $\mathcal{F}_\kappa$, it is clear that it is non-negative and non-decreasing. From \Cref{p.energystab}, we note that $\mathcal{F}_\kappa$ is indeed real-valued. Thus we only need to justify convexity. Since $\mathcal{F}_\kappa$ is non-decreasing, it suffices to prove midpoint convexity, that is
\begin{align*}
     \mathcal{F}_\kappa\left(\frac{x+x'}{2} \right) \le \frac12\left(\mathcal{F}_\kappa(x)+\mathcal{F}_\kappa(x')\right)
 \end{align*}
 for all $x,x'\in \R$. Towards this end fix any $x,x\in \R$ and $\e>0$. By \Cref{prop:f_limit_g+T}, we can get $u,u',v,v'\in \R$ with $u+v=x, u'+v'=x'$ and large enough $N$ such that
 \begin{align*}
     g(u)+\Phi_{\kappa,N}(v) \le \mathcal{F}_\kappa(x)+\e, \quad g(u')+\Phi_{\kappa,N}(v') \le \mathcal{F}_\kappa(x')+\e. 
 \end{align*}
 $u, u', v, v'$ depends on $N$, but we have suppressed it from the notation. Note that $g(s) \to \infty$ when $|s|\to \infty$, $\Phi_{\kappa,N}(s)=0$ remains uniformly bounded on $(-\infty,1]$, and $\Phi_{\kappa,N}(s)=\infty$ for $s>1$. Thus, $u,u',v,v'$ remains bounded as $N \to \infty$.  Let $u''=(u+u')/2$ and $v''=(v+v')/2$. Then by \Cref{cor:tconvex} we have
 \begin{align*}
     \Phi_{\kappa,N}\left(v''-N^{-1/6}\right) \le \frac12\left(\Phi_{\kappa,N}\left(v\right)+\Phi_{\kappa,N}\left(v'\right)\right)+\e.
 \end{align*}
 Note that $u''+v''=(x+x')/2$. Consequently, using the convexity of $g$ and the above fact we have
\begin{equation*}
		\begin{split}
			\left( g \oplus \Phi_{\kappa,N} \right)  \left(\frac{x+x'}{2} \right) &\le g(u''+\tfrac1{N^{1/6}}) + \Phi_{\kappa,N}(v''-\tfrac1{N^{1/6}})
			\\ & = g(u''+\tfrac1{N^{1/6}})-g(u'')+g(u'') + \Phi_{\kappa,N}(v''-\tfrac1{N^{1/6}})
			\\
			&
			\le g(u''+\tfrac1{N^{1/6}})-g(u'')+\frac{1}{2}(g(u) + g(u')) + \frac{1}{2} \left( \Phi_{\kappa,N}(v) + \Phi_{\kappa,N}(v') \right) + \varepsilon
			\\
			&
			\le g(u''+\tfrac1{N^{1/6}})-g(u'')+\frac{1}{2} \bigg( \left(g \oplus \Phi_{\kappa,N} \right) (x) + \varepsilon + \left(g \oplus \Phi_{\kappa,N} \right) (x') + \varepsilon \bigg) + \varepsilon.
		\end{split}
	\end{equation*}
Finally, since $u''$ remains bounded as $N\to \infty$, taking $N\to \infty$, and noting that $\e$ is arbitrary,  in view of \Cref{prop:f_limit_g+T}, we have the midpoint convexity of $\mathcal{F}_\kappa$.
 \end{proof}

\begin{proposition}\label{prop:equicont} $\Phi_{\kappa,N}$ is equicontinuous in the following sense. For any $\varepsilon>0$, there exists $\delta>0$ and $N_\varepsilon>0$ such that for all $N\ge N_\varepsilon$ we have
    $$
    |\Phi_{\kappa,N}(x)-\Phi_{\kappa,N}(y)|\le \varepsilon,
    $$ 
    for all $x,y\in[\lln,1]$ with $|x-y|\le \delta$.
\end{proposition}

\begin{proof} 
    Let $x>y$ and assume $x-y=\delta$. By the midpoint convexity stated in \Cref{cor:tconvex}, for any fixed $\varepsilon'$, there exists $N_{\varepsilon'}$ such that
   \begin{equation*}
        2 \Phi_{\kappa,N}(x) \le \Phi_{\kappa,N}(x-\delta) + \Phi_{\kappa,N}(x+\delta+2N^{-1/6}) + \varepsilon',
    \end{equation*}
    for any $N>N_{\varepsilon'}$, which implies
    \begin{equation} \label{eq:T_t_continuity_1}
        \Phi_{\kappa,N}(x) - \Phi_{\kappa,N}(x-\delta) \le \Phi_{\kappa,N}\big(x+\delta+2N^{-1/6}\big) - \Phi_{\kappa,N}(x) + \varepsilon'. 
    \end{equation}
    Consider a non-negative integer $k$ such that
    \begin{equation*}
       x+k\delta+k(k+1)N^{-1/6} \le 1 < x+(k+1)\delta+(k+1)(k+2)N^{-1/6}.
    \end{equation*}
    Then, iterating \eqref{eq:T_t_continuity_1} we obtain
    \begin{align} \nonumber
            \Phi_{\kappa,N}(x)-\Phi_{\kappa,N}(y) &\le \Phi_{\kappa,N}\left(x+k\delta+k(k+1)N^{-\frac16}\right)-\Phi_{\kappa,N}\left(x+(k-1)\delta+k(k-1)N^{-\frac16}\right) + k\varepsilon'
            \\
            &\le \Phi_{\kappa,N}(1) - \Phi_{\kappa,N}\bigg(1-2\delta-(4k+2)N^{-1/6}\bigg) + k\varepsilon'. \label{eq:T_t_continuity_2}
    \end{align}
    Next, we estimate the term $\Phi_{\kappa,N}(1) - \Phi_{\kappa,N}\big(1-2\delta-(4k+2)N^{-1/6}\big)$. By \Cref{p.energystab} and \Cref{prop:f_limit_g+T} we know that there exists $x_0=x_0(q,\kappa,a)>0$ such that 
    \begin{equation*}
    \begin{split}
        \mathcal{F}_\kappa(x_0) = \Phi_{\kappa,N}(1) + g(x_0-1) = \lim_{N \to +\infty} \inf_{y\in [0,1]} \left\{ \Phi_{\kappa,N}(y) + g(x_0-y) \right\},
    \end{split}
    \end{equation*}
    where $g(y)= \logq y^2/2 $.
    This implies that for any fixed $\varepsilon''$ we can pick $N_{\varepsilon''}$ such that, for all $N>N_{\varepsilon''}$
    \begin{equation*}
        -\varepsilon'' \le \inf_{y\in [0,1]} \left\{ \Phi_{\kappa,N}(y) + g(x_0-y) \right\} - \Phi_{\kappa,N}(1) - g(x_0-1) \le \varepsilon''.
    \end{equation*}
    Then, for any $y\in[0,1]$ we have
    \begin{equation}\label{eq:T_t_continuity_3}
        0\le \Phi_{\kappa,N}(1)- \Phi_{\kappa,N}(y) \le \varepsilon'' + g(x_0-y)-g(x_0-1)  \le \varepsilon'' + \beta_q(1-y),
    \end{equation}
    where $\beta_q:=\eta_q\cdot x_0$. Combining the estimates \eqref{eq:T_t_continuity_2}, \eqref{eq:T_t_continuity_3} we arrive at the bound
    \begin{equation} \label{eq:T_t_continuity_4}
       0\le \Phi_{\kappa,N}(x)-\Phi_{\kappa,N}(y) \le \varepsilon'' + 2\beta_q\delta+(4k+2)\beta_q\cdot N^{-1/6} + k \varepsilon',
    \end{equation}
    which holds for any $N>\max\{ N_{\varepsilon'}, N_{\varepsilon''} \}$. It is now clear that the right-hand side of \eqref{eq:T_t_continuity_4} can be made arbitrarily small, since $k<2/\delta$ and $\varepsilon', \varepsilon''$ are independent of $\delta$. Moreover, we can also allow $|x-y|<\delta$ using the fact that $\Phi_{\kappa,N}$ is non-decreasing. This completes the proof. 
\end{proof}

We recall a real analysis result from our previous paper.

    \begin{lemma} \label{lem:deconv1}
        Let $h_n:\mathbb{R} \to [0,+\infty]$ be a family of non-decreasing functions such that
        \begin{itemize}[leftmargin=20pt]
            \item $h_n(x)=+\infty$ for $x>1$, and there exists $M>0$ such that $h_n(x)\in[0,M]$ for $x\in [\lln,1]$ and $\sup_{x\le \lln} h_n(x) \to 0$ as $n\to \infty$.
            \item For all $\varepsilon>0$, there exists $\delta>0$ and $n_\e>0$ such that for all $n\ge n_\e$ and for all $x,y \in [\lln,1]$ with $|x-y|\le \delta$  we have $$|h_n(x)-h_n(y)|\le \varepsilon.$$
            \item Every subsequential limit of $\{h_n\}$ is convex.
        \end{itemize}
          Let $g(x)=\frac{x^2}{2}\logq$. Assume that $(h_n \oplus g)(x)$ converges pointwise to a proper, lower-semicontinuous convex function $f(x)$ (a function $\mathfrak{f}:\mathbb{R}\to [-\infty,+\infty]$ is said to be proper if $\mathfrak{f}(x)>-\infty$ for all $x\in \mathbb{R}$ and $\mathfrak{f}(x)<+\infty$ for some $x\in \mathbb{R}$). Then $h_n(x)$ converges pointwise to  $$h(x)=(f\ominus g)(x) := \sup_{y\in\R} \{ f(y) - g(x-y) \}.$$ Moreover we have $f=g\oplus h$,  the function $h$ is continuous on $[0,\infty)$, and the function $f$ is differentiable with derivative $f'$ being $\logq$-Lipschitz.
    \end{lemma}

  The above lemma essentially appears as Lemma 4.18 in \cite{das2023large} with $h_n$ being non-increasing. The proof is exactly the same for $h_n$ being non-decreasing. We now have all the ingredients to complete the proof of our main theorems.

    \begin{proof}[Proof of \Cref{thm.main}] Fix $s\in [\lln,1]$. Recall $\Phi_{\kappa,N}$ from \eqref{def:tta}. We would like to apply \Cref{lem:deconv1} with $h_n = \Phi_{\kappa,N}$ to deduce that $\lim_{N\to\infty} \Phi_{\kappa,N}(s)$ exists. Note that $\{\Phi_{\kappa,N}\}$ satisfies the three conditions of \Cref{lem:deconv1}. Indeed, the first condition follows from \Cref{l.tprop}, whereas the second one follows from \Cref{prop:equicont}. The third one is a consequence of \Cref{cor:tconvex}. Since by \Cref{prop:f_limit_g+T} we have $g\oplus\Phi_{\kappa,N}\to \mathcal{F}_\kappa$ pointwise and $\mathcal{F}_\kappa$ is proper, lower semicontinuous and convex by \Cref{prop:fcon}, we thus have that
    \begin{align*}
        \Phi_{\kappa,N}(s)\xrightarrow[t\to\infty]{} \Phi_\kappa^{(-)}(s):=\sup_{y\in \R} \{\mathcal{F}_\kappa(y)-g(s-y)\}
    \end{align*}
and $\Phi_\kappa^{(-)}$ is continuous on $[-\infty,1]$. Due to the properties of $\Phi_{\kappa,N}$ from \Cref{l.tprop} and \Cref{cor:tconvex}, we readily have that $\Phi_\kappa^{(-)}$ is non-decreasing, non-negative and convex with $\Phi_\kappa^{(-)}(\lln)=0$ and $\Phi_\kappa^{(-)}(1)=\kappa \log\frac{1-aq}{1-a}$. 
    \end{proof}
   
\begin{proof}[Proof of \Cref{thm.main2}] A few of the properties of $\mathcal{F}_\kappa$ are already proven in \Cref{l.fha}, \Cref{p.energystab}, and \Cref{prop:fcon}. Differentiability and derivative being $\logq$-Lipschitz follow by an application of \Cref{lem:deconv1} to the $\Phi_{\kappa,N}$ sequence.
\end{proof}

\appendix

\section{Proof of Lemma \ref{l.dd22}} \label{sec.appb}

\begin{proof}[Proof of Part \ref{dda}]  Observe that, since $V_N$ is continuously differentiable, using $V_N(a)-V_N(b)=\int_{a}^b V_N'(y)\diff y$, we have that
\begin{align*} \left|\int_{\ell_i/N}^{\ell_i/N+1/N} V_N(x)\diff x-\frac1{N}V_N(\ell_i/N)\right| & \le \int_{\ell_i/N}^{\ell_i/N+1/N}\left| V_N(x)-V_N(\ell_i/N)\right|\diff x \\ & \le \int_{\ell_i/N}^{\ell_i/N+1/N}\int_{\ell_i/N}^{x}\left| V_N'(y)\right|\diff y\diff x \le C_1N^{-2}(1+\log N),
\end{align*}
where the last inequality is via the assumption $|V_N'(x)|\le C_1(1+\min(x^{-1},\log N))$. Summing over $i\in \{1,2,\ldots,N\}$, we get that
\begin{align}\label{vest}
    \left|\int V_N(x)\psi_{N;\ell}(x)\diff x-\frac1{N}\sum_{i=1}^N V_N(\ell_i/N)\right| \le C_1N^{-1}(1+\log N).
\end{align}
On the other hand, there exists an absolute constant $C>0$ such that for all $\ell\in \mathbb{W}^N$ we have
\begin{align}\label{logest}
    \left|\frac{2}{N^2}\sum_{1\le i<j\le N}\log \left|\frac{\ell_i}{N}-\frac{\ell_j}N\right|-\int\int \log |x-y|\psi_{N;\ell}(x)\psi_{N;\ell}(y)\diff x\diff y\right|\le CN^{-1}\log N.
\end{align}
The above estimate follows via direct computation and is proven within Lemma 7.13 of \cite{das2022large} (take $\theta=1$ and $r=N$ in their setting and see Eq.~(7.43) and (7.44) and the equation after (7.44)). Combining \eqref{vest} and \eqref{logest} leads to \eqref{aineq} with $K_1=2C_1+C$.
\end{proof}

\begin{proof}[Proof of Part \ref{ddb}] Note that using \eqref{e.vclose} we have
\begin{align*}
    |I_{V_N}(\mu_\phi)- I_{\mathpzc{V}}(\mu_\phi)| & = \left|\int V_N(x)\phi(x)\diff x-\int \mathpzc{V}(x)\phi(x)\diff x\right| \\ & \le  \left|\int_{1}^{\infty} \left[\frac{C}N+\frac{\log\kappa}{N} \right]\phi(x)\diff x\right| + \int_{0}^1 \left|\frac{C}N+\frac1N\log\left(1+\frac{\kappa-1}x\right)\right|\phi(x)\diff x.
\end{align*}
For the first integral we use the fact that $\phi$ integrates to $1$, to get that it is at most $K_{2,1}\cdot N^{-1}$ for some $K_{2,1}$ depending on $C,\kappa$. For the second integral, we apply the fact $\phi(x)\le 1$ and then evaluate the integral to get that it is at most $K_{2,2}\cdot N^{-1}$ for some $K_{2,2}$ depending on $C,\kappa$. Taking $K_2=K_{2,1}+K_{2,2}$ proves the claim.
\end{proof}
\begin{proof}[Proof of \ref{ddc}] This proof is an adaptation of Lemma 2.16 in \cite{das2022large}.
As $\mathpzc{V}(x) \ge V_N(x)$ by our assumption, for any $\ell\in \mathbb{W}^N$ we have $I_{V_N}(\mu_N(\ell)) \le I_{\mathpzc{V}}(\mu_N(\ell))$.
 Thus, it suffices to show that there exists $\ell_\star\in \mathbb{W}^N$ and $K_4$ depending on $C_1,B$ such that
\begin{equation}
    \label{bddd}
    \begin{aligned}
    I_{\mathpzc{V}}(\mu_N(\ell_\star)) \le I_{\mathpzc{V}}(\mu_{\phi_{\mathpzc{V}}})+K_3N^{-1}+K_4 \cdot N^{-1}\log N.
\end{aligned}
\end{equation}
We split the rest of the proof into two steps.

\medskip

\noindent\textbf{Step 1. Construction of ${\ell}_\star$.} From now we shall only work with $\phi_{\mathpzc{V}}$ and hence we drop the $\mathpzc{V}$ and simply write $\phi=\phi_{\mathpzc{V}}$.
 We let $y_i$, $i = 1, \dots, N$ be the quantiles of $\phi$, defined as the smallest positive numbers such that
$$\int_0^{y_i} \phi(x)dx = \frac{i - 1/2}{N}.$$
Since $\phi$ is supported on $[0,b]$ and is bounded we have that $y_i$'s are all well-defined and $y_i \in [0, b]$ for all $i = 1, \dots ,N$. 

We now let ${\ell}_{\star,i}$ denote the largest element in $\mathbb{Z}$, which is less than or equal to $N y_{N-i+1}$.  We claim that ${\ell}_{\star} = ({\ell}_{\star,1}, \dots, {\ell}_{\star,N}) \in \mathbb{W}^{N}$, or equivalently we have
$$ {\lambda}_1 \geq \cdots \geq {\lambda}_N \geq 0, \mbox{ where }{\lambda}_i = {\ell}_{\star,i} - (N-i).$$
To see the latter, notice that $y_1 \geq 0$, which implies ${\lambda}_N \geq 0$. 
Suppose, for the sake of contradiction, that ${\lambda}_i - {\lambda}_{i-1} \geq 1 \mbox{ for some $i \in \{2, \dots, N\}$.}$ Then
$${\ell}_{\star,i-1} + 1 = {\lambda}_{i-1} + 1 + (N-i + 1) \leq {\ell}_{\star,i} + 1 \leq Ny_{N-i+1} + 1 =  Ny_{N-i + 2} + N(y_{N-i + 1} - y_{N-i+2}) + 1.$$ 
On the other hand, as $\phi(x) \in [0, 1]$, we have 
\begin{align}\label{y_gap}
\frac1N = \int_{y_{N-i + 1}}^{y_{N-i + 2}} \phi(x)dx \leq  (y_{N-i + 2} - y_{N-i+1})  \implies N(y_{N-i + 1} - y_{N-i+2})  \le -1.
\end{align}
Combining the last two inequalities we get ${\ell}_{\star,i-1} + 1  \leq Ny_{N-i + 2},$
which contradicts the maximality of ${\ell}_{\star,i-1}$ and so we conclude that ${\ell}_\star$ as constructed is in $\mathbb{W}^{N}$.

\medskip
	
\noindent\textbf{Step 2: Verifying \eqref{bddd}.} We first claim that there exists a constant $C_b'>0$ depending only on $b$ such that
\begin{align}\label{log_bd}
N^2\iint\limits_{x_1>x_2}\log(x_1-x_2)\phi(x_2)\phi(x_1)\d x_2 \d x_1 & \le \sum_{1 \leq i < j \leq N}\log\left(\frac{\ell_{\star,i}}{N}-\frac{\ell_{\star,j}}{N}\right)+C_b' \cdot N\log N.
\end{align}
\eqref{log_bd} is proven in Step 2 of the proof of Lemma 2.18 in \cite{das2022large} (see Eq.~(2.35) with $r=N$ therein).

Next, we claim that
\begin{align}\label{v_bdd2}
N\sum_{i=1}^N V\left(\frac{\ell_{\star,i}}{N}\right) \le N^2\int_{\mathbb{R}} V(t)\phi(t)dt + K_3N + 2C_1(b+1)N\log N.
\end{align}
Assuming (\ref{log_bd}) and (\ref{v_bdd2}) we see that 
$$ N^2I_{\mathpzc{V}}(\mu_{N}(\ell_\star)) \leq N^2 I_{\mathpzc{V}}(\mu_{\phi_{\mathpzc{V}}}) + K_3N+(C_b'+2C_1(b+1))N\log N.$$
Taking $K_4:=C_b'+2C_1(b+1)$ yields \eqref{bddd}. We thus focus on proving \eqref{v_bdd2}. We write
\begin{equation}\label{S2GR1}
\begin{split}
&N\sum_{i=1}^N V\left(\frac{\ell_{\star,i}}{N}\right)  = NV\left(\frac{\ell_{\star,1}}{N}\right) +  N^2\sum_{i=1}^{N-1}V\left(\frac{\ell_{\star, N-i + 1}}{N}\right)\int_{y_{i}}^{y_{i+1}} \phi(t)\d t    \\ 
& = NV\left(\frac{\ell_{\star,1}}{N}\right)+N^2\sum_{i=1}^{N-1} \int_{y_{i}}^{y_{i+1}}  \left[ V\left(\frac{\ell_{\star, N-i + 1}}{N}\right) - V(t)\right]\phi(t)\d t + N^2\int_{y_1}^{y_N} V(t) \phi(t) \d t. 
\end{split}
\end{equation}
The first term above, $NV({\ell_{\star,1}}/{N})$, is at most $K_3 N$ (recall the definition of $K_3$ from the statement of the lemma). It thus suffices to show 
\begin{equation}\label{S2GR3}
\left|N^2\sum_{i=1}^{N-1} \int_{y_{i}}^{y_{i+1}}  \left[ V\left(\frac{\ell_{\star,N-i + 1}}{N}\right) - V(t)\right]\phi(t)\d t\right| \le  2C_1(b+1)N \log N.
\end{equation}
By mean value theorem, we have for each $i = 1, \dots, N$ and $t \in [y_i, y_{i+1}]$ that
$$  V\left(\frac{\ell_{\star,N-i + 1}}{N}\right) - V(t) = V'( \kappa(t)) \cdot \left[  \frac{\ell_{\star,N-i + 1}}{N} -t\right].$$
By our assumption, we have $|V'( \kappa(t))| \le C_1(1+\log N)$ and so
$$\left|N^2\sum_{i=1}^{N} \int_{y_{i}}^{y_{i+1}}  \left[ V\left(\frac{\ell_{\star, N-i + 1}}{N}\right) - V(t)\right]\phi(t)\d t\right|  \le 2C_1N^2 \log N \cdot \sum_{i=0}^{N-1} \int_{y_{i}}^{y_{i+1}}  \left[y_{i+1} - y_i + N^{-1}\right]\phi(t)\d t,$$
where we used that $\ell_{\star,N-i + 1}/N \geq y_i - 1/N$ by definition. Observe that by the definition of quantiles
$$ \sum_{i=1}^{N} \int_{y_{i}}^{y_{i+1}}  \left[y_{i+1} - y_i + N^{-1}\right]\phi(t)\d t \le \frac{1}{N} \sum_{i=1}^{N} \left[y_{i+1} - y_i + N^{-1}\right] \le (b+1)/N.$$
The last two equations imply \eqref{S2GR3}, which in turn verifies \eqref{v_bdd2}.
\end{proof}

\section{Equilibrium measure calculations} \label{sec.app}

In this appendix, we verify the claims made in the proof of \Cref{p.energystab}, namely \eqref{intb}, \eqref{inta} and \eqref{intc}. For simplicity we will write $V$ for $\mathpzc{V}_\infty$ and $k$ for $k_{\mathpzc{V}_\infty}$. To compute the integrals, we appeal to the following lemma on some integral identities:
\begin{lemma}[Lemma 6.10 of \cite{das2022large}]\label{lm:integral}
    For $a,b,c,d\geq 0$ with $cd>0$ and $a+b>0$, consider the integrals 
    \begin{equation}
        \mathcal{I}^{\pm}_{a,b,c,d;n}:= \int_0^\infty \frac{\log|a^2\pm b^2z^2|}{(c^2+d^2z^2)^n}\diff z,\quad \mathcal{J}_{c,d;n}:= \int_0^\infty \frac{\diff z}{(c^2+d^2z^2)^n}.
    \end{equation}
    We have the following exact expressions for the above integrals for certain choices of parameters:
    \begin{multicols}{2}
    \begin{enumerate}
        \setlength\itemsep{0.8em}
        \item  $\displaystyle\mathcal{I}^{-}_{a,b,c,d;1}=\frac{\pi\log|a^2+\frac{b^2c^2}{d^2}|}{2cd}$,
        \item  $\mathcal{I}^{+}_{a,b,c,d;1}=\frac{\pi}{cd}\log|a+\frac{bc}{d}|$,
        \item $\displaystyle\mathcal{I}^{-}_{a,b,1,1;2}=\frac{\pi}{4}\log(a^2+b^2)-\frac{\pi b^2}{2(a^2+b^2)}$,
        \item $\displaystyle\mathcal{I}^{+}_{a,b,1,1;2}=\frac{\pi}{2}\log(a+b)-\frac{\pi b}{2(a+b)}$,
        \item $\displaystyle\mathcal{J}_{c,d;1}=\frac{\pi}{2cd}$,
        \item $\displaystyle\mathcal{J}_{1,1;2}=\frac{\pi}{4}$.
    \end{enumerate}
    \end{multicols}
\end{lemma}

\noindent\textbf{Computation of log integral.}
Let us start with computing 
\begin{equation*}
    U(y):= -\int_c^d\log|x-y|\phi(x)\diff x.
\end{equation*}
 Note that $-\log |y-x|=\frac{\diff}{\diff x}((y-x)\log|y-x|+x)$ and 
\begin{equation*}
    \frac{\diff}{\diff x}\phi(x) = \frac{x(p-1)(\kappa+1)+(\kappa p-1)(\kappa-1)}{2\pi x(x+\kappa-1)\sqrt{4xp(x+\kappa-1)-(x(p+1)+\kappa p-1)^2}}.
\end{equation*}
An integration by parts gives 
\begin{equation*}
    U(y) = \left((y-x)\log|y-x|+x\right)\phi(x)\Big|_{x=c}^d-\int_c^d \left((y-x)\log|y-x|+x\right)\phi'(x)\diff x.
\end{equation*}
Note that $4xp(x+\kappa-1)-(x(p+1)+\kappa p-1)^2=(1-p)^2(d-x)(x-c)$. Under the change of variable $x=\frac{c+dz^2}{1+z^2}$ we have 
\begin{equation*}
\begin{aligned}
    &\int_c^d \left((y-x)\log|y-x|+x\right)\phi'(x)\diff x \\   
    &= \frac{1}{\pi(1-p)}\int_0^\infty \diff z\left(\left(y-\frac{c+dz^2}{1+z^2}\right)\log\left|y-\frac{c+dz^2}{1+z^2}\right|+\frac{c+dz^2}{1+z^2}\right)\\
    &\hspace{1cm}\cdot\left(\frac{\kappa p-1}{c+dz^2}+\frac{p-\kappa}{(c+\kappa-1)+(d+\kappa-1)z^2}\right).
\end{aligned}
\end{equation*}
Assume first that $y>d$ or $y<c$. From the decomposition
\begin{equation*}
    y-\frac{c+dz^2}{1+z^2}=\frac{y-c+(y-d)z^2}{1+z^2}=y+\kappa-1-\frac{(c+\kappa-1)+(d+\kappa-1)z^2}{1+z^2},
\end{equation*} 
we have 
\begin{equation*}
\begin{aligned}
    &\pi(1-p) \int_c^d \left((y-x)\log|y-x|+x\right)\phi'(x)\diff x \\   
    &=  (\kappa p-1)\cdot y\left(\mathcal{I}^+_{\sqrt{|y-c|},\sqrt{|y-d|},\sqrt{c},\sqrt{d};1}-\mathcal{I}^+_{1,1,\sqrt{c},\sqrt{d};1}\right)\\
    &+(p-\kappa)\cdot (y+\kappa-1)\left(\mathcal{I}^+_{\sqrt{|y-c|},\sqrt{|y-d|},\sqrt{c+\kappa-1},\sqrt{d+\kappa-1};1}-\mathcal{I}^+_{1,1,\sqrt{c+\kappa-1},\sqrt{d+\kappa-1};1}\right)\\
    &-(\kappa+1)(p-1)\cdot \left(\mathcal{I}^+_{\sqrt{|y-c|},\sqrt{|y-d|},1,1;1}-\mathcal{I}^+_{1,1,1,1;1}\right)\\
    &+(\kappa+1)(p-1)\mathcal{J}_{1,1;1}-(p-\kappa)(\kappa-1)\mathcal{J}_{\sqrt{c+\kappa-1},\sqrt{d+\kappa-1};1}.
\end{aligned}
\end{equation*}
Using the identites from Lemma \ref{lm:integral}, we conclude that for $y\leq c$ or $y\geq d$
\begin{equation}\label{eq:logpotential_1}
    \begin{aligned}
        &\int_c^d \left((y-x)\log|y-x|+x\right)\phi'(x)\diff x \\
        &= y\cdot \mathrm{sgn}(\kappa p-1)\cdot\log \left(\frac{\sqrt{|y-c|}\cdot |\sqrt{\kappa p}+1|+\sqrt{|y-d|}\cdot |\sqrt{\kappa p}-1|}{|\sqrt{\kappa p}+1|+|\sqrt{\kappa p}-1|}\right)\\
        &- (y+\kappa-1)\cdot \log \left(\frac{\sqrt{|y-c|}\cdot (\sqrt{\kappa}+\sqrt{p}))+\sqrt{|y-d|}\cdot (\sqrt{\kappa}-\sqrt{p})}{2\sqrt{\kappa}}\right)\\
        &+(\kappa+1)\cdot \log\left(\frac{\sqrt{|y-c|}+\sqrt{|y-d|}}{2}\right)-1.
    \end{aligned}
\end{equation}
On the other hand, for $c<y<d$, applying Lemma \ref{lm:integral} with $\mathcal{I}^{+}_{\sqrt{y-c},{\sqrt{y-d}},*,*;1}$ replaced by $\mathcal{I}^{-}_{\sqrt{y-c},{\sqrt{d-y}},*,*;1}$, we get 
\begin{equation*}
    \begin{aligned}
        &\int_c^d \left((y-x)\log|y-x|+x\right)\phi'(x)\diff x \\
        &= y\cdot \mathrm{sgn}(\kappa p-1)\cdot\log \left(\frac{\sqrt{4y\sqrt{\kappa p}}}{|\sqrt{\kappa p}+1|+|\sqrt{\kappa p}-1|}\right)\\
        &- (y+\kappa-1)\cdot \log \left(\frac{\sqrt{4(y+\kappa-1)\sqrt{\kappa p}}}{2\sqrt{\kappa}}\right)+(\kappa+1)\cdot \log\left(\frac{\sqrt{d-c}}{2}\right)-1.
    \end{aligned}
\end{equation*}

\noindent\textbf{Computation of potential integral.} We now compute $\displaystyle \int_0^d V(x)\phi(x)\diff x$. A similar integration by parts yields that 
\begin{equation*}
    \begin{aligned}
        \int_0^d V(x)\phi(x)\diff x &= \frac{3(\kappa-1)}{2}+\frac{\log p}{2}\int_{c}^{d} x^2 \phi'(x)\diff x -\frac{1}{2}\int_c^d x^2 \log x\cdot\phi'(x)\diff x\\
        &+ \frac{1}{2}\int_c^d (x+\kappa-1)^2\log (x+\kappa-1)\cdot \phi'(x)\diff x +\frac{(\kappa-1)^2\log(\kappa-1)}{2}\mathbf{1}_{\kappa p\leq 1}.
    \end{aligned}
\end{equation*}
Applying the change of variable $\displaystyle x=\frac{c+dz^2}{1+z^2}$ again we have 
\begin{equation*}
\begin{aligned}
    \pi(1-p)\int_{c}^d x^2\phi'(x)\diff x 
    &= (\kappa p-1)\cdot \left(d\mathcal{J}_{1,1;1}+(c-d)\mathcal{J}_{1,1;2}\right)\\
    &+(p-\kappa)\cdot \left((\kappa-1)^2 \mathcal{J}_{\sqrt{c+\kappa-1},\sqrt{d+\kappa-1};1}+(d-\kappa+1)\mathcal{J}_{1,1;1}+(c-d)\mathcal{J}_{1,1;2}\right)\\ &= \pi(-2\kappa p+p-1).
\end{aligned}
\end{equation*}
Similarly
\begin{equation*}
\begin{aligned}
    &\pi(1-p)\int_{c}^d x^2\log x\cdot \phi'(x)\diff x \\
    &= (\kappa p-1)\cdot \left(d(\mathcal{I}^+_{\sqrt{c},\sqrt{d},1,1;1}-\mathcal{I}^+_{1,1,1,1;1})+(c-d)(\mathcal{I}^+_{\sqrt{c},\sqrt{d},1,1;2}-\mathcal{I}^+_{1,1,1,1;2})\right)\\
    &+(p-\kappa)\cdot (\kappa-1)^2 \cdot (\mathcal{I}^+_{\sqrt{c},\sqrt{d},\sqrt{c+\kappa-1},\sqrt{d+\kappa-1};1}-\mathcal{I}^+_{1,1,\sqrt{c+\kappa-1},\sqrt{d+\kappa-1};1})\\
    &+(p-\kappa)\left((d-\kappa+1)(\mathcal{I}^+_{\sqrt{c},\sqrt{d},1,1;1}-\mathcal{I}^+_{1,1,1,1;1})+(c-d)(\mathcal{I}^{+}_{\sqrt{c},\sqrt{d},1,1;2}-\mathcal{I}^{+}_{1,1,1,1;2})\right),
\end{aligned}
\end{equation*}
and 
\begin{equation*}
\begin{aligned}
    &\pi(1-p)\int_{c}^d (x+\kappa-1)^2\log (x+\kappa-1)\cdot \phi'(x)\diff x \\
    &= (\kappa p-1)\cdot \left((d+2\kappa-2)(\mathcal{I}^+_{\sqrt{c+\kappa-1},\sqrt{d+\kappa-1},1,1;1}-\mathcal{I}^+_{1,1,1,1;1})+(c-d)(\mathcal{I}^+_{\sqrt{c+\kappa-1},\sqrt{d+\kappa-1},1,1;2}-\mathcal{I}^+_{1,1,1,1;2})\right)\\
    &+ (\kappa p-1)\cdot (\kappa-1)^2\cdot (\mathcal{I}^+_{\sqrt{c+\kappa-1},\sqrt{d+\kappa-1},\sqrt{c},\sqrt{d};1}-\mathcal{I}^+_{1,1,\sqrt{c},\sqrt{d};1})\\
    &+(p-\kappa)\left((d+\kappa-1)(\mathcal{I}^+_{\sqrt{c+\kappa-1},\sqrt{d+\kappa-1},1,1;1}-\mathcal{I}^+_{1,1,1,1;1})+(c-d)(\mathcal{I}^{+}_{\sqrt{c+\kappa-1},\sqrt{d+\kappa-1},1,1;2}-\mathcal{I}^{+}_{1,1,1,1;2})\right).
\end{aligned}
\end{equation*}
Combining all the terms together and using Lemma \ref{lm:integral} we conclude that 
\begin{equation}\label{eqvx}
    \begin{aligned}
    \int_0^d V(x)\phi(x)\diff x 
    = (\kappa-1)\log(1-p)+(\kappa-1)^2\log(\kappa-1)-\kappa^2\log \kappa+\kappa\log \kappa+2\kappa-1
    \end{aligned}
\end{equation}
for $\kappa p>1$, and 
\begin{equation}\label{eqvx2}
\begin{aligned}
    \int_0^d V(x)\phi(x)\diff x 
    &= \frac{\left(\kappa ^2-1\right) (p-1) \log (1-p)+(2 \kappa  p-p+1) \log (p)}{2 (p-1)}\\
    &+\frac{(\kappa -1) (-\kappa  p+(\kappa -1) (p-1) \log (\kappa -1)+2 p-3)+\kappa  \log (\kappa ) (\kappa -\kappa  p+2 p)}{2(p-1)}
\end{aligned}
\end{equation}
for $\kappa p\leq 1$.

\noindent\textbf{Verification of the variational conditions.}
Now we are ready to check \eqref{intb} and \eqref{inta}. First we claim that for any $y\in (c,d)$ we have (for either $\kappa p\geq 1$ or $\kappa p< 1$)
\begin{equation}\label{eq:varitional_cond}
    -\int_0^d \log|y-x|\phi(x)\diff x+\frac{V(y)}{2} = -\frac{\kappa\log \kappa+\log p}{2}+\frac{\kappa+1}{2}\log(1-p)+\frac{\kappa+1}{2}.
\end{equation}
Indeed for $\kappa p\geq 1$, since $\phi(x)\equiv 0$ for $x\in (0,c)$ we have for any given $y\in (c,d)$
\begin{equation*}
    \begin{aligned}
      &-\int_0^d \log|y-x|\phi(x)\diff x+\frac{V(y)}{2}= U(y)+\frac{V(y)}{2}\\
      &= -\frac{1}{2}y\log y+\frac{1}{4}y\log \kappa p+\frac{1}{2}(y+\kappa-1)\log (y+\kappa-1)+\frac{1}{4}(y+\kappa-1)\log \frac{p}{\kappa}\\
      &-\frac{\kappa+1}{2}\log \frac{\sqrt{\kappa p}}{1-p}+\frac{1}{2}y\log p^{-1}+\frac{1}{2}y\log y -\frac{1}{2}(y+\kappa-1)\log(y+\kappa-1)+\frac{\kappa-1}{2}+1\\
      &= -\frac{\kappa\log \kappa+\log p}{2}+\frac{\kappa+1}{2}\log(1-p)+\frac{\kappa+1}{2}.
    \end{aligned}
\end{equation*}
On the other hand, if $\kappa p<1$ then $\phi(x)\equiv 1$ for $x\in (0,c)$ and we have for any given $y\in (c,d)$
\begin{equation*}
    \begin{aligned}
        &-\int_0^d \log|y-x|\phi(x)\diff x+\frac{V(y)}{2} = -\int_0^c \log|y-x|\phi(x)\diff x+U(y)+ \frac{V(y)}{2}\\
        & = -y\log y -\int_{c}^{d}\left((y-x)\log(y-x)+x\right)\phi'(x)\diff x +\frac{V(y)}{2}\\
        &= -\frac{1}{2}y\log y+\frac{1}{4}y\log \kappa p+\frac{1}{2}(y+\kappa-1)\log (y+\kappa-1)+\frac{1}{4}(y+\kappa-1)\log \frac{p}{\kappa}\\
      &-\frac{\kappa+1}{2}\log \frac{\sqrt{\kappa p}}{1-p}+\frac{1}{2}y\log p^{-1}+\frac{1}{2}y\log y -\frac{1}{2}(y+\kappa-1)\log(y+\kappa-1)+\frac{\kappa-1}{2}+1\\
      &= -\frac{\kappa\log \kappa+\log p}{2}+\frac{\kappa+1}{2}\log(1-p)+\frac{\kappa+1}{2}.
    \end{aligned}
\end{equation*}
Now for $\kappa p>1$ note that 
\begin{equation*}
    \int_c^d k(x,y)\phi(x)\diff x = \int_c^d \frac{1}{2}V(x)\phi(x)\diff x+\int_c^d -\log|y-x|\phi(x)\diff x+\frac{1}{2}V(y), 
\end{equation*}
which is equal to $$\displaystyle \frac{(\kappa-1)^2\log (\kappa-1)-\kappa^2 \log \kappa+\kappa \log p+3\kappa}{2}+\kappa \log(1-p)-\frac{\log p}{2}$$ for $y\in (c,d)$, by \eqref{eq:varitional_cond} and \eqref{eqvx}.  Moreover from \eqref{eq:logpotential_1} it is not hard to check that $$\displaystyle -\int_0^d \log|y-x|\phi(x)\diff x+\frac{V(y)}{2},$$ as a function of $y$, is increasing for $y\in (d,\infty)$ and decreasing for $y\in (0,c)$ if $\kappa p>1$, These complete the verification of \eqref{intb}. \eqref{inta} is verified in a similar way using \eqref{eqvx2} instead of \eqref{eqvx}, note that for $\kappa p\leq 1$, the function $\displaystyle -\int_0^d \log|y-x|\phi(x)\diff x+\frac{V(y)}{2}$ is increasing in $y$ for $y\in (0,c)$.

\medskip

\noindent\textbf{Computation of the logarithmic energy.}
Finally we compute the logarithmic energy
\begin{equation*}
    F_\kappa(\infty):=\int_0^\infty\int_0^\infty k(x,y)\phi(x)\phi(y)\diff x\diff y.
\end{equation*}
For $\kappa >1/p$, using to \eqref{eq:varitional_cond} and \eqref{eqvx} we have 
    \begin{align*}
        F_\kappa(\infty)&=\int_c^d\int_c^d k(x,y)\phi(x)\phi(y)\diff x\diff y\\
        &= \int_c^d \frac{1}{2} V(x)\phi(x)\diff x+\int_c^d \left(-\int_c^d \log|y-x|\phi(x)\diff x +\frac{1}{2}V(y)\right)\phi(y)\diff y\\
        &=  \frac{(\kappa-1)^2\log (\kappa-1)-\kappa^2 \log \kappa+3\kappa}{2}+\kappa \log(1-p)-\frac{\log p}{2}.
    \end{align*}
The computation for $\kappa\leq 1/p$ is a bit more complicated. In this case we have 
\begin{align*}
        F_\kappa(\infty)&=\int_0^d\int_0^d k(x,y)\phi(x)\phi(y)\diff x\diff y\\
        &= \int_0^d \frac{1}{2} V(x)\phi(x)\diff x+\int_c^d \left(-\int_0^d \log|y-x|\phi(x)\diff x +\frac{1}{2}V(y)\right)\phi(y)\diff y\\
        &+\int_0^c\left(-\int_c^d \log|y-x|\phi(x)\diff x\right)\diff y-\int_0^c\int_0^c\log|y-x|\diff x\diff y+\frac{1}{2}\int_0^c V(y)\diff y.
    \end{align*}
A direct computation gives 
\begin{equation}\label{eq:I3}
    \int_0^c\int_0^c\log|y-x|\diff x\diff y = -\frac{3}{2}c^2+c^2\log c,
\end{equation}
and
\begin{equation}\label{eq:I4}
\begin{aligned}
    &\int_0^c V(y)\diff y = \frac{c^2\log \frac{c}{p}-(c+\kappa-1)^2\log(c+\kappa-1)+(\kappa-1)^2\log(\kappa-1)+3c(\kappa-1)}{2}.
\end{aligned}
\end{equation}
To compute $\displaystyle\int_0^c\left(-\int_c^d \log|y-x|\phi(x)\diff x\right)\diff y$ we use Fubini's theorem:
\begin{equation*}
    \begin{aligned}
        \int_0^c\left(-\int_c^d \log|y-x|\phi(x)\diff x\right)\diff y &= \int_c^d \left(\int_0^c -\log|y-x|\diff y\right)\phi(x)\diff x\\
        &= \int_c^d \left((x-c)\log(x-c)-x\log x+c\right)\phi(x)\diff x.
    \end{aligned}
\end{equation*}
The last integral is computed in a very similar manner as $\displaystyle \int_c^d V(x)\phi(x)\diff x$ so we only record the result here and skip the details:
\begin{equation}\label{eq:I5}
\begin{aligned}
    &\int_c^d \left((x-c)\log(x-c)-x\log x+c\right)\phi(x)\diff x\\
    &= \frac{p \left(6 \kappa +\kappa ^2 p-4 \kappa  \sqrt{\kappa  p}-4 \sqrt{\kappa  p}+p\right) \log \left(\frac{\sqrt{\kappa  p}}{1-p}\right)-3 \left(\sqrt{\kappa  p}-1\right)^4-3 (p-1) \left(\sqrt{\kappa  p}-1\right)^2}{2 (p-1)^2}\\
    &+\frac{\kappa  (\kappa +1) (p-1) p-(\kappa +1) (p-1) \sqrt{\kappa  p}+3\left(\sqrt{\kappa  p}-1\right)^4 \log \left(1-\sqrt{\kappa  p}\right)}{2 (p-1)^2}\\
    &+\frac{\left(\kappa -2 \sqrt{\kappa  p}+p\right)^2 \log \left(1-\sqrt{\frac{p}{\kappa }}\right)-\left(\kappa  (p-1) (\kappa  (p-1)+2)+(1 - \sqrt{p \kappa})^4\right) \log (1-p)}{2(p-1)^2}.
\end{aligned}
\end{equation}
Combining \eqref{eqvx2},\eqref{eq:varitional_cond},\eqref{eq:I3},\eqref{eq:I4} and \eqref{eq:I5}, after some simplifications we arrive at 
\begin{equation*}
    \int_0^d\int_0^d k(x,y)\phi(x)\phi(y)\diff x\diff y = \frac{(\kappa-1)^2\log (\kappa-1)-\kappa^2 \log \kappa+3\kappa}{2}+\kappa \log(1-p)-\frac{\log p}{2},
\end{equation*}
for $\kappa<1/p$ as well. This completes the verification of \eqref{intc}.

\bibliographystyle{alpha}

\begin{thebibliography}{}

\end{thebibliography}


\begin{thebibliography}{HMMSD22}

\bibitem[AB19]{abphase}
A.~Aggarwal and A.~Borodin.
\newblock {Phase transitions in the ASEP and stochastic six-vertex model}.
\newblock {\em The Annals of Probability}, 47(2):613 -- 689, 2019.

\bibitem[ACG23]{aggarwal2023asep}
A.~Aggarwal, I.~Corwin, and P.~Ghosal.
\newblock {The ASEP speed process}.
\newblock {\em Advances in Mathematics}, 422:109004, 2023.

\bibitem[ACH24]{ach24}
A.~Aggarwal, I.~Corwin, and M.~Hegde.
\newblock Scaling limit of the colored {ASEP} and stochastic six-vertex models.
\newblock {\em arXiv preprint arXiv:2403.01341}, 2024.

\bibitem[Agg16]{Aggarwal_6v_to_ASEP}
A~Aggarwal.
\newblock {Convergence of the Stochastic Six-Vertex Model to the ASEP}.
\newblock {\em Mathematical Physics, Analysis and Geometry}, 20, Dec 2016.

\bibitem[Agg18]{Amol2016Stationary}
A.~Aggarwal.
\newblock {Current Fluctuations of the Stationary ASEP and Six-Vertex Model}.
\newblock {\em {Duke Math J.}}, 167(2):269--384, 2018.
\newblock arXiv:1608.04726 [math.PR].

\bibitem[Agg20]{aggarwal2020limit}
A.~Aggarwal.
\newblock Limit shapes and local statistics for the stochastic six-vertex model.
\newblock {\em Communications in Mathematical Physics}, 376:681--746, 2020.

\bibitem[AGZ10]{AndersonGuionnetZeitouniBook}
G.W. Anderson, A.~Guionnet, and O.~Zeitouni.
\newblock {\em {An introduction to random matrices}}.
\newblock Cambridge University Press, 2010.

\bibitem[BBW18]{BorodinBufetovWheeler2016}
A.~Borodin, A.~Bufetov, and M.~Wheeler.
\newblock {Between the stochastic six vertex model and Hall-Littlewood processes}.
\newblock {\em Duke Mathematical Journal}, 167(13):2457--2529, 2018.

\bibitem[BCG16]{BCG6V}
A.~Borodin, I.~Corwin, and V.~Gorin.
\newblock Stochastic six-vertex model.
\newblock {\em Duke Mathematical Journal}, 165(3):563--624, 2016.

\bibitem[BG15]{BufetovGorinLogConcavity}
A.~Bufetov and V.~Gorin.
\newblock Stochastic monotonicity in young graph and thoma theorem.
\newblock {\em International Mathematics Research Notices}, 2015(23):12920--12940, 03 2015.

\bibitem[BGG17]{borodin2017gaussian}
A.~Borodin, V.~Gorin, and A.~Guionnet.
\newblock Gaussian asymptotics of discrete $\beta$-ensembles.
\newblock {\em Publications math{\'e}matiques de l'IH{\'E}S}, 125(1):1--78, 2017.

\bibitem[BGS21]{basu2021upper}
R.~Basu, S.~Ganguly, and A.~Sly.
\newblock Upper tail large deviations in first passage percolation.
\newblock {\em Communications on Pure and Applied Mathematics}, 74(8):1577--1640, 2021.

\bibitem[BKM24]{bkm24}
Jnaneshwar Baslingker, Manjunath Krishnapur, and Mokshay Madiman.
\newblock {Log-concavity in one-dimensional Coulomb gases and related ensembles}.
\newblock {\em arXiv preprint arXiv:2412.15116}, 2024.

\bibitem[BLS17]{bls17}
Mikl{\'o}s B{\'o}na, Marie-Louise Lackner, and Bruce~E Sagan.
\newblock Longest increasing subsequences and log concavity.
\newblock {\em Annals of Combinatorics}, 21:535--549, 2017.

\bibitem[BO17]{BO2016_ASEP}
A.~Borodin and G.~Olshanski.
\newblock {The ASEP and determinantal point processes}.
\newblock {\em Communications in Mathematical Physics}, 353(2):853--903, 2017.

\bibitem[Bor18]{borodin2016stochastic_MM}
A.~Borodin.
\newblock {Stochastic higher spin six vertex model and Macdonald measures}.
\newblock {\em Journal of Mathematical Physics}, 59(2):023301, 2018.

\bibitem[CC22]{cafasso_claeys_KPZ}
M.~Cafasso and T.~Claeys.
\newblock {A Riemann-Hilbert Approach to the lower tail of the Kardar-Parisi-Zhang Equation}.
\newblock {\em Communications on Pure and Applied Mathematics}, 75(3):493--540, 2022.

\bibitem[CG18]{ciech}
F.~Ciech and N.~Georgiou.
\newblock A large deviation principle for last passage times in an asymmetric {B}ernoulli potential.
\newblock {\em arXiv preprint arXiv:1810.11377}, 2018.

\bibitem[CG20a]{corwin20general}
I.~Corwin and P.~Ghosal.
\newblock {KPZ} equation tails for general initial data.
\newblock {\em Electron J Probab}, 25, 2020.

\bibitem[CG20b]{lwtail}
I.~Corwin and P.~Ghosal.
\newblock {Lower tail of the KPZ equation}.
\newblock {\em Duke Mathematical Journal}, 169(7):1329 -- 1395, 2020.

\bibitem[CGK{\etalchar{+}}18]{corwin2018coulomb}
I.~Corwin, P.~Ghosal, A.~Krajenbrink, P.~Le~Doussal, and L.-C. Tsai.
\newblock Coulomb-gas electrostatics controls large fluctuations of the {Kardar-Parisi-Zhang} equation.
\newblock {\em Physical review letters}, 121(6):060201, 2018.

\bibitem[CGST20]{corwin2020stochastic}
I.~Corwin, P.~Ghosal, H.~Shen, and L.-C. Tsai.
\newblock {Stochastic PDE limit of the six vertex model}.
\newblock {\em Communications in Mathematical Physics}, 375(3):1945--2038, 2020.

\bibitem[CH24]{corwin_hegde2022lower}
I.~Corwin and M.~Hegde.
\newblock {The lower tail of q-pushTASEP}.
\newblock {\em Communications in Mathematical Physics}, 405(3):64, 2024.

\bibitem[Cor12]{CorwinKPZ}
I.~Corwin.
\newblock {The Kardar-Parisi-Zhang equation and universality class}.
\newblock {\em Random Matrices Theory Appl.}, 1, 2012.
\newblock arXiv:1106.1596 [math.PR].

\bibitem[DD22]{das2022large}
S.~Das and E.~Dimitrov.
\newblock Large deviations for discrete $\beta$-ensembles.
\newblock {\em Journal of Functional Analysis}, 283(1):109487, 2022.

\bibitem[DDV24]{das2024upper}
S.~Das, D.~Dauvergne, and B.~Vir{\'a}g.
\newblock Upper tail large deviations of the directed landscape.
\newblock {\em arXiv preprint arXiv:2405.14924}, 2024.

\bibitem[Dei00]{deift2000orthogonal}
P.~Deift.
\newblock {\em Orthogonal Polynomials and Random Matrices: A Riemann-Hilbert Approach: A Riemann-Hilbert Approach}, volume~3.
\newblock American Mathematical Soc., 2000.

\bibitem[DHS24]{hindy24}
H.~Drillick and L.~Haunschmid-Sibitz.
\newblock The stochastic six-vertex model speed process.
\newblock {\em In preparation}, 2024.

\bibitem[Dim23]{dimitrov2023two}
E.~Dimitrov.
\newblock Two-point convergence of the stochastic six-vertex model to the {Airy} process.
\newblock {\em Communications in Mathematical Physics}, 398(3):925--1027, 2023.

\bibitem[DL23]{drillick2023strong}
H.~Drillick and Y.~Lin.
\newblock Strong law of large numbers for the stochastic six vertex model.
\newblock {\em Electronic Journal of Probability}, 28:1--21, 2023.

\bibitem[DLM25]{das2023large}
Sayan Das, Yuchen Liao, and Matteo Mucciconi.
\newblock Large deviations for the {{\(q\)}}-deformed polynuclear growth.
\newblock {\em Ann. Probab.}, 53(4):1223--1286, 2025.

\bibitem[DOV22]{DOV18}
D.~Dauvergne, J.~Ortmann, and B.~Virag.
\newblock The directed landscape.
\newblock {\em Acta Mathematica}, 229(2), 2022.

\bibitem[DS97]{ds97}
PD~Dragnev and EB~Saff.
\newblock Constrained energy problems with applications to orthogonal polynomials of a discrete variable.
\newblock {\em Journal d’Analyse Mathematique}, 72(1):223--259, 1997.

\bibitem[DT21]{dt21}
S.~Das and L.-C. Tsai.
\newblock Fractional moments of the stochastic heat equation.
\newblock In {\em Annales de l'Institut Henri Poincar{\'e}, Probabilit{\'e}s et Statistiques}, volume~57, pages 778--799. Institut Henri Poincar{\'e}, 2021.

\bibitem[DT24]{das2024solving}
S.~Das and L.-C. Tsai.
\newblock {Solving marginals of the LDP for the directed landscape}.
\newblock {\em arXiv preprint arXiv:2405.17041}, 2024.

\bibitem[DZ99]{deuschel_zeitouni_1999}
J.D. Deuschel and O.~Zeitouni.
\newblock On increasing subsequences of i.i.d. samples.
\newblock {\em Combinatorics, Probability and Computing}, 8(3):247–263, 1999.

\bibitem[DZ22]{dz1}
S.~Das and W.~Zhu.
\newblock Upper-tail large deviation principle for the {ASEP}.
\newblock {\em Electronic Journal of Probability}, 27:1--34, 2022.

\bibitem[EJ17]{emrah}
E.~Emrah and C.~Janjigian.
\newblock Large deviations for some corner growth models with inhomogeneity.
\newblock {\em Markov Process. Related Fields}, 23(1):267--312, 2017.

\bibitem[For10]{Forrester-LogGas}
PJ~Forrester.
\newblock {\em {Log-gases and random matrices}}.
\newblock Princeton University Press, 2010.

\bibitem[Gan21]{ganguly2021random}
S.~Ganguly.
\newblock {Random metric geometries on the plane and Kardar-Parisi-Zhang universality}.
\newblock {\em arXiv preprint arXiv:2110.11287}, 2021.

\bibitem[GH22]{ganguly2022sharp}
S.~Ganguly and M.~Hegde.
\newblock Sharp upper tail estimates and limit shapes for the {KPZ} equation via the tangent method.
\newblock {\em arXiv preprint arXiv:2208.08922}, 2022.

\bibitem[GH23]{ganguly_hegde2023optimal}
S.~Ganguly and M.~Hegde.
\newblock {Optimal tail exponents in general last passage percolation via bootstrapping \& geodesic geometry}.
\newblock {\em Probability Theory and Related Fields}, 186(1):221--284, 2023.

\bibitem[GHZ23]{ganguly2023brownian}
S.~Ganguly, M.~Hegde, and L.~Zhang.
\newblock Brownian bridge limit of path measures in the upper tail of {KPZ} models.
\newblock {\em arXiv preprint arXiv:2311.12009}, 2023.

\bibitem[GL23]{gl20}
P.~Ghosal and Y.~Lin.
\newblock Lyapunov exponents of the {SHE} under general initial data.
\newblock In {\em Annales de l'Institut Henri Poincare (B) Probabilites et statistiques}, volume~59, pages 476--502. Institut Henri Poincar{\'e}, 2023.

\bibitem[GLLT23]{gaudreaulamarre2023kpz}
P.~Y. Gaudreau~Lamarre, Y.~Lin, and L.-C. Tsai.
\newblock {KPZ} equation with a small noise, deep upper tail and limit shape.
\newblock {\em Probab Theory Related Fields}, pages 1--36, 2023.

\bibitem[GS92]{GwaSpohn1992}
L.-H. Gwa and H.~Spohn.
\newblock Six-vertex model, roughened surfaces, and an asymmetric spin {H}amiltonian.
\newblock {\em Physical Review Letters}, 68(6):725--728, 1992.

\bibitem[GS13]{georgiou2013large}
N.~Georgiou and T.~Sepp{\"a}l{\"a}inen.
\newblock Large deviation rate functions for the partition function in a log-gamma distributed random potential.
\newblock {\em The Annals of Probability}, 41(6):4248--4286, 2013.

\bibitem[GS24]{gs24}
P.~Ghosal and G.L.F. Silva.
\newblock Six vertex model and the meixner ensemble.
\newblock {\em In preparation}, 2024.

\bibitem[GV85]{gessel1985binomial}
I.~Gessel and G.~Viennot.
\newblock Binomial determinants, paths, and hook length formulae.
\newblock {\em Advances in Mathematics}, 58(3):300--321, 1985.

\bibitem[HMMSD22]{huh_et_al_schur_log_concavity}
J.~Huh, J.P. Matherne, K.~Mészáros, and A.~St.~Dizier.
\newblock Logarithmic concavity of schur and related polynomials.
\newblock {\em Trans. Amer. Math. Soc.}, 375:4411--4427, 2022.

\bibitem[HMS19]{hartmann19}
A.~K. Hartmann, B.~Meerson, and P.~Sasorov.
\newblock Optimal paths of nonequilibrium stochastic fields: The {K}ardar{-P}arisi{-Z}hang interface as a test case.
\newblock {\em Physical Review Research}, 1(3):032043, 2019.

\bibitem[IMS22]{IMS_KPZ_free_fermions}
T.~Imamura, M.~Mucciconi, and T.~Sasamoto.
\newblock {Solvable models in the KPZ class: approach through periodic and free boundary Schur measure}.
\newblock {\em arXiv preprint arXiv:2204.08420}, 2022.

\bibitem[IMS24]{IMS_matching}
T.~Imamura, M.~Mucciconi, and T.~Sasamoto.
\newblock {Identity between restricted Cauchy sums for the q-Whittaker and skew Schur polynomials}.
\newblock {\em SIGMA. Symmetry, Integrability and Geometry: Methods and Applications}, 20:064, 2024.

\bibitem[Jan15]{janjigian2015large}
C.~Janjigian.
\newblock Large deviations of the free energy in the {O’C}onnell--{Y}or polymer.
\newblock {\em Journal of Statistical Physics}, 160(4):1054--1080, 2015.

\bibitem[Jan19]{jan19}
C.~Janjigian.
\newblock {Upper tail large deviations in Brownian directed percolation}.
\newblock 2019.

\bibitem[Jen00]{jen00}
L.~Jensen.
\newblock {\em The asymmetric exclusion process in one dimension}.
\newblock PhD thesis, Ph. D. dissertation, New York Univ., New York, 2000.

\bibitem[Joh00a]{Johansson_growth_matrices}
K.~Johannson.
\newblock Random growth and random matrices.
\newblock {\em Unpublished notes}, 2000.
\newblock \url{http://web.mit.edu/18.325/www/johansso.pdf}.

\bibitem[Joh00b]{johansson2000shape}
K.~Johansson.
\newblock {Shape fluctuations and random matrices}.
\newblock {\em Communications in Mathematical Physics}, 209(2):437--476, 2000.
\newblock arXiv:math/9903134 [math.CO].

\bibitem[Kes86]{kesten}
H~Kesten.
\newblock {Aspects of first passage percolation. {\'E}cole d’{\'E}t{\'e} de Probabilit{\'e}s de Saint Flour XIV-1984, volume 1180 of Lecture Notes in Mathematics}, 1986.

\bibitem[Kim21]{kim21}
Y.H. Kim.
\newblock The lower tail of the half-space {KPZ} equation.
\newblock {\em Stochastic Processes and their Applications}, 142:365--406, 2021.

\bibitem[KK07]{kolokolov07}
IV~Kolokolov and SE~Korshunov.
\newblock Optimal fluctuation approach to a directed polymer in a random medium.
\newblock {\em Phys Rev B}, 75(14):140201, 2007.

\bibitem[KK09]{kolokolov09}
IV~Kolokolov and SE~Korshunov.
\newblock Explicit solution of the optimal fluctuation problem for an elastic string in a random medium.
\newblock {\em Phys Rev E}, 80(3):031107, 2009.

\bibitem[KLD17]{krajenbrink17short}
A.~Krajenbrink and P.~Le~Doussal.
\newblock Exact short-time height distribution in the one-dimensional {K}ardar{--P}arisi{--Z}hang equation with {B}rownian initial condition.
\newblock {\em Phys Rev E}, 96(2):020102, 2017.

\bibitem[KLD18a]{krajenbrink18half}
A.~Krajenbrink and P.~Le~Doussal.
\newblock Large fluctuations of the {KPZ} equation in a half-space.
\newblock {\em SciPost Phys}, 5:032, 2018.

\bibitem[KLD18b]{krajenbrink18simple}
A.~Krajenbrink and P.~Le~Doussal.
\newblock Simple derivation of the {$(-\lambda H)^{5/2}$} tail for the 1{D} {KPZ} equation.
\newblock {\em J Stat Mech Theory Exp}, 2018(6):063210, 2018.

\bibitem[KLD21]{krajenbrink21}
A.~Krajenbrink and P.~Le~Doussal.
\newblock Inverse scattering of the {Z}akharov{--S}habat system solves the weak noise theory of the {K}ardar{--P}arisi{--Z}hang equation.
\newblock {\em Phys Rev Lett}, 127(6):064101, 2021.

\bibitem[KLD22]{krajenbrink22flat}
A.~Krajenbrink and Pierre Le~Doussal.
\newblock Inverse scattering solution of the weak noise theory of the {K}ardar{--P}arisi{--Z}hang equation with flat and {B}rownian initial conditions.
\newblock {\em Phys Rev E}, 105:054142, 2022.

\bibitem[KLD23]{krajenbrink23}
A.~Krajenbrink and P.~Le~Doussal.
\newblock Crossover from the macroscopic fluctuation theory to the {K}ardar{--P}arisi{--Z}hang equation controls the large deviations beyond {E}instein's diffusion.
\newblock {\em Physical Review E}, 107(1):014137, 2023.

\bibitem[KLDP18]{krajenbrink2018systematic}
A.~Krajenbrink, P.~Le~Doussal, and S.~Prolhac.
\newblock Systematic time expansion for the {Kardar--Parisi--Zhang equation}, linear statistics of the {GUE} at the edge and trapped fermions.
\newblock {\em Nuclear Physics B}, 936:239--305, 2018.

\bibitem[KMS16]{kamenev16}
A.~Kamenev, B.~Meerson, and PV~Sasorov.
\newblock Short-time height distribution in the one-dimensional {K}ardar{--P}arisi{--Z}hang equation: Starting from a parabola.
\newblock {\em Phys Rev E}, 94(3):032108, 2016.

\bibitem[KOV93]{Kerov1993}
S.~Kerov, G.~Olshanski, and A.~Vershik.
\newblock Harmonic analysis on the infinite symmetric group. {A} deformation of the regular representation.
\newblock {\em Comptes Rendus Acad. Sci. Paris Ser. I}, 316:773--778, 1993.

\bibitem[KPZ86]{KPZ1986}
M.~Kardar, G.~Parisi, and Y.~Zhang.
\newblock Dynamic scaling of growing interfaces.
\newblock {\em Physical Review Letters}, 56(9):889, 1986.

\bibitem[LDMRS16]{ledoussal16short}
P.~Le~Doussal, SN~Majumdar, A.~Rosso, and G.~Schehr.
\newblock Exact short-time height distribution in the one-dimensional {K}ardar{--P}arisi{--Z}hang equation and edge fermions at high temperature.
\newblock {\em Phys Rev Lett}, 117(7):070403, 2016.

\bibitem[LDMS16]{ledoussal16long}
P.~Le~Doussal, SN~Majumdar, and G.~Schehr.
\newblock Large deviations for the height in 1{D} {K}ardar-{P}arisi-{Z}hang growth at late times.
\newblock {\em EPL (Europhysics Letters)}, 113(6):60004, 2016.

\bibitem[Lin21]{lin2020kpz}
Y.~Lin.
\newblock {Lyapunov exponents of the half-line SHE}.
\newblock {\em Journal of Statistical Physics}, 183:1--34, 2021.

\bibitem[LPP07]{Lam_Postnikov_Pylyavskyy_concavity}
T.~Lam, A.~Postnikov, and P.~Pylyavskyy.
\newblock Schur positivity and schur log-concavity.
\newblock {\em American Journal of Mathematics}, 129(6):1611--1622, 2007.

\bibitem[LS77]{logan_shepp1977variational}
B.F. Logan and L.A. Shepp.
\newblock {A variational problem for random Young tableaux}.
\newblock {\em Advances in Mathematics}, 26(2):206--222, 1977.

\bibitem[LS23a]{landon2023tail}
B.~Landon and P.~Sosoe.
\newblock {Tail estimates for the stationary stochastic six vertex model and ASEP}.
\newblock {\em arXiv preprint arXiv:2308.16812}, 2023.

\bibitem[LS23b]{landon2023upper}
B.~Landon and P.~Sosoe.
\newblock {Upper tail bounds for stationary KPZ models}.
\newblock {\em Communications in Mathematical Physics}, pages 1--25, 2023.

\bibitem[LS24]{landon2022tail}
B.~Landon and P.~Sosoe.
\newblock {Tail bounds for the O’Connell-Yor polymer}.
\newblock {\em Electronic Journal of Probability}, 29:1--47, 2024.

\bibitem[LT21]{lin21}
Y.~Lin and L.-C. Tsai.
\newblock Short time large deviations of the {KPZ} equation.
\newblock {\em Commun Math Phys}, 386(1):359--393, 2021.

\bibitem[LT25a]{lin22}
Y.~Lin and L.-C. Tsai.
\newblock A lower-tail limit in the weak noise theory.
\newblock In {\em Annales de l'Institut Henri Poincare (B) Probabilites et statistiques}, volume~61, pages 1334--1347. Institut Henri Poincar{\'e}, 2025.

\bibitem[LT25b]{lin23}
Y.~Lin and L.-C. Tsai.
\newblock {Spacetime limit shapes of the KPZ equation in the upper tails}.
\newblock {\em Communications in Mathematical Physics}, 406(5):113, 2025.

\bibitem[Mac95]{Macdonald1995}
I.G. Macdonald.
\newblock {\em Symmetric functions and {H}all polynomials}.
\newblock Oxford University Press, 2nd edition, 1995.

\bibitem[Meh04]{mehta2004random}
M.L. Mehta.
\newblock {\em {Random matrices}}.
\newblock Academic press, 2004.

\bibitem[MKV16]{meerson16}
B.~Meerson, E.~Katzav, and A.~Vilenkin.
\newblock Large deviations of surface height in the {K}ardar{--P}arisi{--Z}hang equation.
\newblock {\em Phys Rev Lett}, 116(7):070601, 2016.

\bibitem[MS17]{meerson17}
B.~Meerson and J.~Schmidt.
\newblock Height distribution tails in the {K}ardar{--P}arisi{--Z}hang equation with brownian initial conditions.
\newblock {\em J Stat Mech Theory Exp}, 2017(10):103207, 2017.

\bibitem[MV18]{meerson18}
B.~Meerson and A.~Vilenkin.
\newblock Large fluctuations of a {K}ardar{-P}arisi{-Z}hang interface on a half line.
\newblock {\em Physical Review E}, 98(3):032145, 2018.

\bibitem[Oko97]{Okounkov1997}
A.~Okounkov.
\newblock {Log-Concavity of Multiplicities with Application to Characters of $U(\infty)$}.
\newblock {\em Advances in Mathematics}, 127(2):258--282, 1997.

\bibitem[Oko01]{okounkov2001infinite}
A.~Okounkov.
\newblock {Infinite wedge and random partitions}.
\newblock {\em Selecta Mathematica}, 7(1):57--81, 2001.

\bibitem[Oko03]{Okounkov2003_why_would}
A.~Okounkov.
\newblock Why would multiplicities be log-concave?
\newblock {\em The Orbit Method in Geometry and Physics. Progress in Mathematics.}, 213:329–347, 2003.

\bibitem[OT19]{ot19}
S.~Olla and L.-C. Tsai.
\newblock {Exceedingly large deviations of the totally asymmetric exclusion process}.
\newblock {\em Electronic Journal of Probability}, 24(none):1 -- 71, 2019.

\bibitem[Pau35]{pauling1935structure}
L.~Pauling.
\newblock The structure and entropy of ice and of other crystals with some randomness of atomic arrangement.
\newblock {\em Journal of the American Chemical Society}, 57(12):2680--2684, 1935.

\bibitem[PS11]{pastur2011eigenvalue}
L.A. Pastur and M.~Shcherbina.
\newblock {\em Eigenvalue distribution of large random matrices}.
\newblock Number 171. American Mathematical Soc., 2011.

\bibitem[QT21]{qt21}
J.~Quastel and L.-C. Tsai.
\newblock {Hydrodynamic large deviations of TASEP}.
\newblock {\em arXiv preprint arXiv:2104.04444}, 2021.

\bibitem[Qua13]{quastel_introduction_to_KPZ}
J.~Quastel.
\newblock {Introduction to KPZ}.
\newblock {\em {Current Developments in Mathematics}}, 2011, 03 2013.

\bibitem[Rom15]{romik_2015}
D.~Romik.
\newblock {\em The Surprising Mathematics of Longest Increasing Subsequences}.
\newblock Institute of Mathematical Statistics Textbooks. Cambridge University Press, 2015.

\bibitem[Sep98a]{sep98a}
T.~Sepp{\"a}l{\"a}inen.
\newblock Coupling the totally asymmetric simple exclusion process with a moving interface.
\newblock {\em Markov Process. Related Fields}, 4(4):593--628, 1998.

\bibitem[Sep98b]{sepp98mprf}
T.~Sepp\"{a}l\"{a}inen.
\newblock Hydrodynamic scaling, convex duality and asymptotic shapes of growth models.
\newblock {\em Markov Process. Related Fields}, 4(1):1--26, 1998.

\bibitem[Sep98c]{seppalainen_98_increasing}
T.~Sepp{\"a}l{\"a}inen.
\newblock Large deviations for increasing sequences on the plane.
\newblock {\em Probability Theory and Related Fields}, 112:221--244, 10 1998.

\bibitem[SMP17]{sasorov2017large}
P.~Sasorov, B.~Meerson, and S.~Prolhac.
\newblock Large deviations of surface height in the 1+ 1-dimensional {Kardar--Parisi--Zhang} equation: exact long-time results for $\lambda h< 0$.
\newblock {\em Journal of Statistical Mechanics: Theory and Experiment}, 2017(6):063203, 2017.

\bibitem[ST97]{saff1997logarithmic}
E.B. Saff and V.~Totik.
\newblock {\em Logarithmic Potentials with External Fields}.
\newblock Grundlehren der mathematischen Wissenschaften. Springer, Berlin, Heidelberg, 1 edition, 1997.

\bibitem[Tsa22]{tsai_lower_tail}
L.-C. Tsai.
\newblock {Exact lower-tail large deviations of the KPZ equation}.
\newblock {\em Duke Mathematical Journal}, 171(9):1879 -- 1922, 2022.

\bibitem[Tsa25]{tsai2023high}
L.-C. Tsai.
\newblock {High moments of the SHE in the clustering regimes}.
\newblock {\em Journal of Functional Analysis}, 288(1):110675, 2025.

\bibitem[Var04]{var04}
S.R.S. Varadhan.
\newblock Large deviations for the asymmetric simple exclusion process.
\newblock {\em Stochastic analysis on large scale interacting systems}, 39:1--27, 2004.

\bibitem[Ver24]{verges2024large}
J.~Verges.
\newblock Large deviation principle at speed $ n^d $ for the random metric in first-passage percolation.
\newblock {\em arXiv preprint arXiv:2404.09589}, 2024.

\bibitem[Zyg18]{zygouras_review}
N.~Zygouras.
\newblock {Some algebraic structures in the KPZ universality}.
\newblock {\em arXiv preprint}, 2018.
\newblock arXiv:1812.07204v3 [math.PR].



\end{thebibliography}
\newcommand{\etalchar}[1]{$^{#1}$}

\end{document}